\documentclass[12pt]{article}
\usepackage{amsmath,amssymb,amsthm}
\usepackage{hyperref}
\newtheorem{theorem}{Theorem}[section]
\newtheorem{lemma}[theorem]{Lemma}
\newtheorem{proposition}[theorem]{Proposition}
\newtheorem{conjecture}[theorem]{Conjecture}
\newtheorem{question}[theorem]{Question}
\theoremstyle{definition}
\newtheorem{definition}[theorem]{Definition}
\newtheorem{remark}[theorem]{Remark}
\newtheorem{propertyinner}{Property}
\newenvironment{property}[1]{%
  \renewcommand\thepropertyinner{#1}%
  \propertyinner
}{\endpropertyinner}
\def\de{\delta}

\def\om{\omega}
\def\Om{\Omega}
\def\si{\sigma}
\def\bN{\mathbb{N}}
\def\bR{\mathbb{R}}
\def\bZ{\mathbb{Z}}
\def\cG{\mathcal{G}}
\def\cL{\mathcal{L}}
\def\cP{\mathcal{P}}

\def\F{\mathsf{F}}
\def\hF{\widehat{\mathsf{F}}}

\def\sS{\mathsf{S}}
\def\Fix{\mathop{\mathrm{Fix}}}
\def\Homeo{\mathop{\mathrm{Homeo}}}
\def\id{\mathrm{id}}{\tiny }
\def\Orb{\mathop{\mathrm{Orb}}}
\def\RiSt{\mathop{\mathrm{RiSt}}\nolimits}
\def\St{\mathop{\mathrm{St}}\nolimits}
\def\Sti{\mathop{\mathrm{St}}\nolimits^\bullet}

\def\Stp{\mathop{\mathrm{St}}\nolimits^\star}
\def\supp{\mathop{\mathrm{supp}}}
\title{On maximal subgroups of ample groups}
\author{Rostislav Grigorchuk \and Yaroslav Vorobets}
\date{}
\begin{document}

\maketitle

\begin{abstract}
The paper is concerned with maximal subgroups of the ample (better known as 
topological full) groups of homeomorphisms of totally disconnected compact 
metrizable topological spaces.  We describe all maximal subgroups that are 
stabilizers of finite sets.  Under certain assumptions on the ample group 
(including minimality), we describe all maximal subgroups that are 
stabilizers of closed sets or stabilizers of partitions into clopen sets.  
In particular, our results apply to the ample groups associated with Cantor 
minimal systems and to some Higman-Thompson groups.
\end{abstract}

%\keywords{ample group, topological full group, maximal subgroup, Cantor 
%minimal system}

%\keywords[2020 Mathematics Subject 
%Classification]{\codes[Primary]{20E28, 20F38, 20F65, 37B02, 37B05}}

%
\section{Introduction}

In this paper we study maximal subgroups of the ample groups (the latter 
are better known as the topological full groups).  Let us begin with an
overview of maximal and related to them weakly maximal subgroups.

A proper subgroup $H$ of a group $G$ is \emph{maximal} if there is no group
placed between $H$ and $G$ in the lattice of subgroups of $G$.  Maximal 
subgroups play a very important role in group theory.  The task of 
describing all maximal subgroups of a group $G$ is equivalent to the task 
of classifying all primitive actions of $G$ (an action of the group on a 
set $X$ is \emph{primitive} if it preserves no nontrivial equivalence 
relations on $X$).  Indeed, if $H$ is a maximal subgroup of $G$ then the 
natural action of $G$ on the cosets of $H$ is primitive.  Conversely, if 
$\alpha:G\curvearrowright X$ is a primitive action on a set $X$ that is not
a singleton then the point stabilizers $\St_\alpha(x)$, $x\in X$ are
maximal subgroups of the group $G$.  Maximal subgroups can also arise as 
the stabilizers of sets or collections of sets (e.g., partitions).

The problem of describing all maximal subgroups of a given group attracted 
a lot of attention.  Among the most remarkable results here is the complete 
solution of this problem for finite symmetric groups (based on the 
O'Nan-Scott theorem), which was obtained in the 1980s as one of the first 
applications of the classification of finite simple groups (see, e.g., 
Section 8.5 in the book \cite{DM}).  There is also an understanding on how 
to approach this problem for general finite groups (as outlined by 
Aschbacher and Scott \cite{AS85}).  Much less is known about maximal 
subgroups of infinite groups.  One may be interested whether such a group 
has a maximal subgroup of infinite index.  For countable groups, this 
question is closely related to the question about \emph{primitivity} of the 
group, that is, existence of a faithful primitive action.  Another question 
is to determine how many maximal subgroups are there, say, whether a given 
countable group has a \emph{continuum} (that is, an uncountable set of 
cardinality $2^{\aleph_0}$) of maximal subgroups.  If an infinite group 
does not have many maximal subgroups, it makes sense to look for 
\emph{weakly maximal subgroups}, which are subgroups of infinite index 
maximal with respect to this property.  Note that in the case a countable 
group $G$ is finitely generated, any subgroup of finite index in $G$ is 
finitely generated as well, which implies that there are at most countably 
many such subgroups.  More importantly, it follows that any proper subgroup 
of $G$ is contained in a maximal subgroup while any subgroup of infinite 
index is contained in a weakly maximal subgroup.

A substantial progress in the late 1970s was achieved by Margulis and 
Soifer \cite{MS77,MS79,MS81}.  Answering a question of Platonov, they 
obtained fundamental results on maximal subgroups of finitely generated 
linear groups.  The first of their results is that any such group admits a 
maximal subgroup of infinite index if and only if it is not virtually 
solvable.  Another result is that the free nonabelian group $F_n$ on 
$n\ge2$ generators has a continuum of maximal subgroups.  These results 
were extended by Gelander and Y. Glasner \cite{GG08} to general countable 
linear groups.  Later Gelander and Meiri \cite{GM16} showed that each of 
the groups $SL_n(\bZ)$, $n\ge3$ has a continuum of maximal subgroups.  For 
more on these and related topics, see the survey \cite{GGS20}.

A completely different story is told by the groups of branch type.  The 
class of the \emph{branch groups} was introduced in \cite{Grig00} in 
relation to the study of just infinite  groups (these are infinite groups 
in which every proper quotient is finite) and as an abstract model behind 
the family of groups $\mathcal{X}=\{G_\om\}$, 
$\om\in\{0,1,2\}^{\mathbb{N}}$ constructed in \cite{Grig84}.  This family 
mostly consists of groups of intermediate growth (faster than polynomial 
but slower than exponential).  Two notable representatives of the family 
are the groups $\mathcal{G}_{(012)^{\infty}}$ (the ``first'' Grigorchuk 
group) and $\mathcal{G}_{(01)^{\infty}}$ (the Grigorchuk-Erschler group) 
given by periodic sequences $(012)^{\infty}=012012\ldots$  and 
$(01)^{\infty}=0101\ldots$.  These two groups were also at the root of 
studies of the class of \emph{self-similar groups} (see the book 
\cite{Nek05}).  Among recent discoveries is the fact that they are related 
to substitution subshifts generated by primitive substitutions, in 
particular, the period  doubling and Morse subshifts (see 
\cite{MatteBon,GLN17,GrigVorob}).  These subshifts are important 
representatives of the Cantor minimal systems (the latter are important in 
our study).

The question of Hartly from 1993 about existence of maximal subgroups of 
infinite index in the group $\mathcal{G}_{(012)^{\infty}}$ was answered in 
the negative by Pervova \cite{Per00}.  This result made a big impact on the 
study of $\mathcal{G}_{(012)^{\infty}}$ and led to establishing such 
properties as subgroup separability (also called the LERF property) and 
decidability of generalized word problem (see \cite{GrWilson03}).  In 
\cite{Per05}, Pervova extended her result to some branch groups in a family
that was named the GGS (Grigorchuk-Gupta-Sidki) groups in the book 
\cite{Bau}.  The groups considered by Pervova are torsion groups (that is, 
all elements are of finite order), and presently there are no known 
examples of finitely generated torsion branch groups admitting a maximal 
subgroup of infinite  index.  Recently Francoeur and Thillaisundaram 
\cite{FT} extended Pervova's result to all GGS groups, including 
non-torsion groups.  On the other hand, Bondarenko \cite{Bond} constructed 
an example of a non-torsion branch group with a maximal subgroup of 
infinite index.  Later Francoeur and Garrido \cite{FG} discovered that, in 
fact, the group $\mathcal{G}_{(01)^{\infty}}$ has maximal subgroups of 
infinite index.  There are only countably many of those and they are all 
described.  Also, it was shown that the non-torsion iterated monodromy 
groups of the tent map (a special case of some groups first introduced by 
\v{S}uni\'c in \cite{Sun} as “siblings of the Grigorchuk group”) have 
exactly countably many maximal subgroups of infinite index (those are 
described up to conjugacy).  Presently there are no known examples of 
finitely generated branch groups with uncountably many maximal subgroups.  
Another relevant fact is that the branch groups cannot act 
quasi-2-transitively on infinite sets (see \cite{Francoeur21}).

The \emph{weakly branch groups} are a far going generalization of the class
of branch groups.  Their introduction was initiated by the studies around 
the Basilica group $\mathcal{B}$ invented in \cite{GriZuk01} as a group 
generated by a $3$-state automaton over the binary alphabet.  Bartholdi and 
Virag \cite{BarthVirag} proved amenability of $\mathcal{B}$, which produced 
the first example of amenable but not subexponentially amenable group 
(answering the question from \cite{Grig98} about existence of such 
groups).  Later it was discovered that $\mathcal{B}$ is isomorphic to the  
iterated monodromy group of the polynomial $z^2-1$ 
\cite{BarGrigNek,Nek05}.  In the paper \cite{Francoeur2020}, Francoeur 
extended techniques of Pervova to a large class of weakly branch groups, 
which allowed him not only to prove that any maximal subgroup in a branch 
group is itself branch, but also to prove that the Basilica group has 
maximal subgroups of infinite index and hence is primitive (as each proper 
quotient of $\mathcal{B}$ is virtually nilpotent).

The next class of infinite groups whose maximal subgroups attracted 
attention of researchers are the \emph{Higman-Thompson type groups}.  
Savchuk \cite{Sav10,Sav15} showed that all orbits of the action of 
Thompson's group $F$ on the interval $(0,1)$ are primitive, which implies 
that the point stabilizers $\St_F(x)$, $x\in(0,1)$ are maximal subgroups of
infinite index.  The associated Schreier coset graphs are quasi-isometric 
to a tree.  It was also shown that $\St_F(x)$ is not finitely generated if 
$x$ is irrational while being finitely generated for $x\in\bZ[1/2]$.  This 
study was extended by Golan and Sapir \cite{GS17,GS17b,GS17a}, Aiello and 
Nagnibeda \cite{AN21b,AN21a}, and others.  Among other results, Golan and 
Sapir produced the first example of a maximal subgroup of $F$ that does not 
arise as the stabilizer of a point.  Another Thompson's group $V$ is 
considered by Belk, Bleak, Quick and Skipper in \cite{BBQS22} where an 
uncountable family of maximal subgroups of $V$ is produced.  Moreover, it 
is shown that there are uncountably many pairwise non-isomorphic maximal 
subgroups of $V$.  In \cite{BBQS24}, the same four authors proved that 
Thompson's group $T$ is yet another maximal subgroup of the group $V$.

Weakly maximal subgroups of finitely generated infinite groups can often 
play a role similar to that of maximal subgroups of infinite index.   For 
instance, they are useful in the study of profinite groups (Barnea and 
Shalev \cite{BarneaShalev}).  As was observed in \cite{BarGri00,BarGri02}, 
stabilizers of points on the boundary of the rooted tree on which a weakly 
branch group acts are weakly maximal subgroups.  In the case of branch 
groups from the family $\mathcal{X}$, the associated Schreier coset graphs 
have surprisingly simple, linear geometric structure.  Almost all of them 
are quasi-isometric to the Cayley graph of $\bZ$.  When taking into account 
edge labels, the structure becomes more complicated but still controlled.  
For the group $\mathcal{G}_{(012)^{\infty}}$, for example, the graphs 
represent the so-called aperiodic order (see \cite{GLN17}).

Branch groups, weakly branch groups and groups of Higman-Thompson type are 
subclasses of the class of \emph{micro-supported groups}.  Elements of this 
class are groups $G$ acting faithfully on a topological space $X$ in such a 
way that for any nonempty open subset $U\subset X$, the rigid stabilizer 
$\RiSt_G(U)$ (also referred to as the local subgroup and denoted $G_U$), 
which consists of elements acting trivially on $X\setminus U$, is 
nontrivial.  Also, all these groups represent an even wider class of
\emph{dynamically defined groups}, the object of a fast developing new area 
on the border between theory of dynamical systems and group theory.  A 
recent book by Nekrashevych \cite{Nek2022} is an excellent source of 
information on methods and results of this direction of mathematics.

A rich source of dynamically defined groups are groups that will be called 
in this paper the \emph{ample groups}.  The idea of amplification (or 
saturation) in dynamics and group theory is quite simple.  Given a 
topological space $X$ and a group $G$ of its homeomorphisms, one can 
enlarge $G$ to a group $\F(G)$ by adding those homeomorphisms of $X$ that 
act locally as elements of $G$.  We call the group $\F(G)$ the \emph{full 
amplification} of $G$.  The group $G$ is called \emph{ample} if $\F(G)=G$. 
Note that $\F(\F(G))=\F(G)$ so that the group $\F(G)$ is always ample.  
This idea works best when $X$ is a Cantor set or, more generally, a totally
disconnected compact metrizable space.  This is because the topology on
such a space is generated by clopen (i.e., both closed and open) sets.  
Clopen sets allow to cut all homeomorphisms in $G$ into pieces, then new 
homeomorphisms can be constructed, as a jigsaw puzzle, out of those
pieces.

In what follows, all topological spaces are assumed to be totally
disconnected, compact and metrizable.  There are many situations when 
continuous group actions on such spaces arise so that the above
construction can be used.  For example, any countable group $G$ acts 
naturally by permutations on the space $A^G$, where $A$ is a finite set
with more than one element.  The actions of $G$ on closed invariant subsets 
of $A^G$ are now a popular area of studies.  The group of automorphisms of 
an infinite, locally finite tree acts naturally on the boundary of the 
tree.  Any countable group $G$ acts by conjugation on the space 
$\mathrm{Sub}(G)$ of its subgroups, where the topology on $\mathrm{Sub}(G)$ 
is induced by the product topology on $2^G$.  This list can go on.
    
The notion of an ample group was introduced by Krieger in \cite{Krieger}.  
The groups considered in \cite{Krieger} are locally finite (that is, every 
finitely generated subgroup is finite), and so this notion has seen limited
use being applied only to locally finite groups.  We would like to extend 
it to arbitrary groups of homeomorphisms replacing the common notion of a 
topological full group.  The latter originated in the theory of Cantor 
minimal systems.  A \emph{Cantor minimal system} $(X,f)$ consists of a 
Cantor set $X$ and a minimal homeomorphism $f:X\to X$.  In the paper 
\cite{GW}, E. Glasner and Weiss associated to $(X,f)$ two groups, the 
\emph{full group} $[f]$ and the \emph{finite full group} $[[f]]$ (notation 
is from respectively \cite{GPS} and \cite{Matui}).  The full group $[f]$ 
consists of all homeomorphisms of $X$ that leave invariant every orbit of 
$f$.  It is relevant to the study of orbit equivalence.  Any map $g\in[f]$ 
can be given by a formula $g(x)=f^{n(x)}(x)$, $x\in X$ for some function 
$n:X\to\bZ$.  Since the minimal homeomorphism $f$ has no periodic points, 
it follows that the function $n$ is unique and its level sets are closed.  
The finite full group $[[f]]$ consists of those $g\in[f]$ for which the 
function $n$ is continuous or, equivalently, takes only finitely many 
values.  In that case, all level sets of $n$ are clopen.  Informally, 
elements of $[f]$ are ``pointwise'' elements of the cyclic group $\langle 
f\rangle$ while elements of $[[f]]$ are ``piecewise'' elements of $\langle 
f\rangle$.  It is easy to observe that $[[f]]=\F(\langle f\rangle)$, the 
full amplification of the cyclic group generated by $f$.  The group $[[f]]$
was renamed the \emph{topological full group} (TFG) in \cite{GPS}.  An 
important property of amplification is that the ample group $\F(G)$ is 
countable whenever $G$ is countable.  In particular, the TFG $[[f]]$ is 
always countable whereas the full group $[f]$ is not.

The remarkable result proved by Giordano, Putnam and Skau \cite{GPS} is 
that the (isomorphism class of) TFG $[[f]]$ is an almost complete invariant
of dynamics of a Cantor minimal system $(X,f)$.  To be precise, if $(X,f)$ 
and $(Y,g)$ are two Cantor minimal systems such that the groups $[[f]]$ and
$[[g]]$ are isomorphic then the systems are \emph{flip conjugate}, which 
means that for some homeomorphism $\phi:X\to Y$ we have $\phi f\phi^{-1}=g$
or $g^{-1}$.  As there is a continuum of pairwise non-flip-conjugate Cantor
minimal systems (e.g., systems with different entropies), one gets a 
continuum of pairwise non-isomorphic TFGs.  The result is based on two 
fundamental facts about TFGs.  For simplicity, let us assume that $X=Y$.  
The first fact, established in \cite{GPS}, is that if the TFGs $[[f]]$ and 
$[[g]]$ are isomorphic then any isomorphism between them is implemented by
a conjugation in the group $\Homeo(X)$ of all homeomorphisms of $X$.  This 
property of TFGs seems to be characteristic for various kinds of 
micro-supported groups (see Section 2.2 in the book \cite{Nek2022}).  The 
second fact, derived from an older result of Boyle, is that if the ample 
groups $[[f]]$ and $[[g]]$ are the same then the cyclic groups $\langle 
f\rangle$ and $\langle g\rangle$ are conjugate in the group $[[f]]=[[g]]$.

The systematic study of group-theoretic properties of the TFG $[[f]]$ 
associated with a Cantor minimal system $(X,f)$ was initiated by Matui 
\cite{Matui}.  He showed that the commutator subgroup of $[[f]]$ is
simple.  In the case when $f$ is a minimal subshift over a finite alphabet,
the commutator subgroup is finitely generated (but not finitely presented).
The group $[[f]]$ admits a nontrivial homomorphism onto $\bZ$ and the 
kernel of that homomorphism is a product of two locally finite subgroups 
(the so-called factorization property).  Further, the group $[[f]]$ has the
property LEF (local embeddability into finite groups), see \cite{GrigMed}.  
Any finite group and the free abelian group of infinite rank embed into 
$[[f]]$.  Every branch group from the family $\mathcal{X}$ mentioned before
embeds into some TFG (Matte Bon \cite{MatteBon}).  The remarkable result 
obtained by Juschenko and Monod \cite{JuschMonod} (confirming a conjecture 
of Grigorchuk and Medynets \cite{GrigMed}) states that any TFG associated 
with a Cantor minimal system is amenable.  Note that these groups are not 
elementary amenable, which means that they cannot be obtained from finite 
and abelian groups using certain natural operations that preserve 
amenability.  This is the second known type of amenable but not elementary 
amenable groups, the first type being the groups of intermediate growth.
      
The goal of this paper is to initiate the systematic study of maximal 
subgroups of general ample groups.  Our results are mostly reminiscent of 
the classification of maximal subgroups in finite symmetric groups.  Let us
recall that all subgroups of a finite symmetric group are split into three 
classes: (i) \emph{intransitive subgroups} (those that leave invariant a 
nontrivial subset),  (ii) \emph{imprimitive subgroups} (transitive 
subgroups that leave invariant a nontrivial partition), and (iii) 
\emph{primitive subgroups} (the remaining ones).  Maximal subgroups in the 
first two classes are easy to describe.  Namely, intransitive maximal 
subgroups are stabilizers of certain subsets while imprimitive maximal 
subgroups are stabilizers of certain partitions.  Primitive maximal 
subgroups form a number of subclasses described by the O'Nan-Scott theorem.

When dealing with the ample groups, arbitrary subsets and partitions should 
be replaced by closed subsets and partitions into closed subsets.  
Similarly, transitivity is replaced with minimality (which is absence of 
nontrivial closed invariant subsets).  An ample group $\cG\subset\Homeo(X)$ 
acting minimally on the topological space $X$ is an analog of a finite 
symmetric group (in the case when $X$ is finite, $\cG$ is exactly the 
symmetric group).  So we split all subgroups of the group $\cG$ into three 
classes: (I) subgroups leaving invariant a nontrivial closed set, (II) 
subgroups acting minimally on $X$ but leaving invariant a nontrivial 
partition into closed sets, and (III) topologically primitive subgroups 
(the remaining ones).  Furthermore, we distinguish three subclasses in the 
first class: (I$_1$) subgroups leaving invariant a nonempty finite set, 
(I$_2$) subgroups leaving invariant an infinite, nowhere dense closed set, 
and (I$_3$) subgroups leaving invariant a nontrivial clopen set.  For 
general subgroups, these subclasses need not be disjoint or cover the 
entire class (I), but as far as the maximal subgroups are concerned, the 
subclasses (I$_1$), (I$_2$) and (I$_3$) form a partition of (I).  In the 
class (II), we distinguish one subclass (II$_0$) consisting of subgroups 
leaving invariant a nontrivial partition into clopen sets (or, 
equivalently, a finite partition into closed sets).

We proceed to the description of our main results.  The first of them 
provides a characterization of those maximal subgroups of ample groups that 
are stabilizers of closed sets.  Namely, they are exactly maximal subgroups 
in the class (I).

\begin{theorem}\label{main1}
Let $\cG\subset\Homeo(X)$ be an ample group that acts minimally on $X$.  
Suppose $H$ is a maximal subgroup of $\cG$ that does not act minimally on 
$X$.  Then $H=\St_{\cG}(Y)$, the stabilizer of some closed set $Y\subset X$ 
different from the empty set and $X$.  Moreover, the induced action of 
$\St_{\cG}(Y)$ on $Y$ is minimal.
\end{theorem}

The next theorem already gives a continuum of maximal subgroups in any 
ample group acting without finite orbits on a Cantor set.  All those 
subgroups belong to the subclass (I$_1$) of the above classification.   

\begin{theorem}\label{main2}
Let $\cG\subset\Homeo(X)$ be an ample group that has no finite orbits.  
Suppose $Y$ is a finite nonempty subset of $X$.  Then the stabilizer 
$\St_{\cG}(Y)$ is a maximal subgroup of $\cG$ if and only if $Y$ is 
contained in a single orbit of $\cG$.
\end{theorem}

For a more general and more detailed result on the stabilizers of finite 
sets, see Theorem \ref{maxsub-finite} below.  Theorem \ref{main2} imposes 
very modest conditions on the ample group $\cG$ (Theorem 
\ref{maxsub-finite} imposes no conditions at all).  To treat the 
stabilizers of infinite closed sets, we need stronger assumptions.  Namely, 
$\cG$ has to act minimally on $X$ and to possess another property that we 
call \emph{Property \ref{prop-E}} (entanglement): for any clopen sets 
$U_1,U_2\subset X$ that overlap, the local subgroup (i.e., rigid  
stabilizer) $\cG_{U_1\cup U_2}$ is generated by its subgroups $\cG_{U_1}$ 
and $\cG_{U_2}$.

\begin{theorem}\label{main3}
Let $\cG\subset\Homeo(X)$ be an ample group that acts minimally on a Cantor 
set $X$ and has Property \ref{prop-E}.  Suppose $Y\subset X$ is an infinite 
closed set that is nowhere dense in $X$.  Then the stabilizer 
$\St_{\cG}(Y)$ is a maximal subgroup of $\cG$ if and only if it acts 
minimally when restricted to $Y$.
\end{theorem}

Proposition \ref{plenty-of-examples} below shows that Theorem \ref{main3} 
provides uncountably many examples of maximal subgroups of the ample group 
$\cG$ (different from the stabilizers of finite sets).

If an ample group $\cG$ acts minimally on a Cantor set $X$ and the 
stabilizer $\St_{\cG}(Y)$ of a closed set $Y$ is a maximal subgroup of 
$\cG$, then $\St_{\cG}(Y)$ acts minimally on $Y$.  This implies that the 
set $Y$ is either finite, or infinite and nowhere dense, or clopen.  The 
first two cases are covered by Theorems \ref{main2} and \ref{main3}.  It 
remains to consider the stabilizers of clopen sets.

\begin{theorem}\label{main4}
Let $\cG\subset\Homeo(X)$ be an ample group that acts minimally on a Cantor 
set $X$ and has Property \ref{prop-E}.  Suppose $U$ is a clopen set 
different from the empty set and $X$.  Then $\St_{\cG}(U,X\setminus U)$, 
the stabilizer of the partition $X=U\sqcup(X\setminus U)$, is a maximal 
subgroup of $\cG$.  If $U$ cannot be mapped onto $X\setminus U$ by an 
element of $\cG$ then $\St_{\cG}(U)=\St_{\cG}(U,X\setminus U)$; otherwise 
$\St_{\cG}(U)$ is a subgroup of index $2$ in $\St_{\cG}(U,X\setminus U)$.
\end{theorem}

In the case when the clopen set $U$ can be mapped onto $X\setminus U$, the 
stabilizer $\St_{\cG}(U)$ is not a maximal subgroup of $\cG$.  However, 
$\St_{\cG}(U,X\setminus U)$ is the only group placed between $\St_{\cG}(U)$ 
and $\cG$ in the lattice of subgroups of $\cG$.

In addition to the stabilizers of clopen sets, Theorem \ref{main4} also
treats the stabilizers of partitions into two clopen sets.  Our next result 
covers the stabilizers of partitions into three or more clopen sets.

\begin{theorem}\label{main5}
Let $\cG\subset\Homeo(X)$ be an ample group that acts minimally on a Cantor 
set $X$ and has Property \ref{prop-E}.  Suppose $X=U_1\sqcup U_2\sqcup\dots 
\sqcup U_k$ is a partition of $X$ into at least three nonempty clopen 
sets.  Then the stabilizer $\St_{\cG}(U_1,U_2,\dots,U_k)$ of the partition  
is a maximal subgroup of $\cG$ if and only if its induced action on the set 
$\{U_1,U_2,\dots,U_k\}$ is transitive.
\end{theorem}

Our last result on maximal subgroups provides a characterization of those 
maximal subgroups of ample groups that are the stabilizers of partitions 
into clopen sets.

\begin{theorem}\label{main6}
Let $X$ be a Cantor set and $\cG\subset\Homeo(X)$ be an ample group with 
Property \ref{prop-E}.  Suppose $H$ is a maximal subgroup of $\cG$ that 
acts minimally on $X$ and contains a local subgroup $\cG_U$ for some 
nonempty clopen set $U$.  Then $H=\St_{\cG}(U_1,U_2,\dots,U_k)$ for some 
partition $X=U_1\sqcup U_2\sqcup\dots\sqcup U_k$ into nonempty clopen 
sets.  Moreover, the partition is unique, it consists of at least two sets, 
and the induced action of $H$ on the set $\{U_1,U_2,\dots,U_k\}$ is 
transitive.
\end{theorem}

Theorems \ref{main3}, \ref{main4}, \ref{main5} and \ref{main6} require the 
ample group to have Property \ref{prop-E}.  One class of ample groups with 
this property are the topological full groups associated with Cantor 
minimal systems.

\begin{theorem}\label{main7}
For any minimal homeomorphism $f$ of a Cantor set, the ample group 
$\F(\langle f\rangle)$ has Property \ref{prop-E}.
\end{theorem}

The topological full groups of Cantor minimal system are part of a larger 
class of ample groups with Property \ref{prop-E}.  For example, the cyclic 
group $\langle f\rangle$ in Theorem \ref{main7} can be replaced by any 
countable elementary amenable group acting minimally and freely on the 
Cantor set.  Another, quite different example of an ample group with 
Property \ref{prop-E} is Thompson's group $V$ as well as its various 
generalizations, for example, the Higman-Thompson groups $V_{n,r}$ (see, 
e.g., \cite{BCMNO}) and the Brin-Thompson groups $nV$ (see \cite{Brin}).  
Yet another class of ample groups with Property \ref{prop-E} are the 
(topological) \emph{full groups} of minimal, purely infinite \'etale 
groupoids considered by Matui in \cite{Matui-pureinf1,Matui-pureinf2}.  See 
Section \ref{property-E} for more details.

It would be interesting to find more examples of the ample groups with 
Property \ref{prop-E}.  More generally, it would be interesting to find out 
which of our results extend to other classes of micro-supported groups (in 
particular, the classes mentioned above).

For the ample groups acting minimally on a Cantor set and having Property 
\ref{prop-E}, our results describe all maximal subgroups in the subclasses 
(I$_1$), (I$_2$), (I$_3$) and (II$_0$).  Presumably, the other maximal 
subgroups in the class (II) are also stabilizers of partitions into closed 
sets.  In the case of finite symmetric groups, there is a large variety of 
primitive maximal subgroups.  Maximal subgroups of Thompson's group $V$ 
found in \cite{BBQS22} include a continuum of topologically primitive 
subgroups.  As for the ample groups associated with Cantor minimal systems 
(or, more generally, obtained by amplifying countable amenable groups 
acting minimally on the Cantor set), no interesting examples of such 
subgroups are found so far.  This leads us to formulate the following bold 
conjecture.

\begin{conjecture}\label{bold-conjecture}
Let $f$ be a minimal homeomorphisms of a Cantor set $X$.  Suppose $H$ is a 
maximal subgroup of the ample group $\cG=\F(\langle f\rangle)$.  Then 
exactly one of the following statements holds true.
\begin{itemize}
\item[(i)]
$H$ is the stabilizer of a closed set $Y\subset X$ different from $X$ and 
the empty set.  Moreover, the induced action of $H$ on $Y$ is minimal.  
Moreover, the set $Y$ cannot be mapped onto $X\setminus Y$ by elements of 
the group $\cG$.
\item[(ii)]
$H$ is the stabilizer of a partition 
$\mathcal{P}=\{Y_\alpha\}_{\alpha\in\mathcal{A}}$ of $X$ into closed sets 
different from the partition into points and the trivial partition 
$\{X\}$.  Moreover, the partition $\mathcal{P}$ is unique and the induced 
action of $H$ on the factor space $X/\mathcal{P}$ is minimal.
\item[(iii)]
$H$ is a normal subgroup of finite prime index in $\cG$.  Moreover, $H$ 
contains the commutator subgroup of $\cG$.
\end{itemize}
\end{conjecture}

The paper is organized as follows.  In Section \ref{sect-amplify} we 
describe the concept of amplification and define the ample groups.  In 
Section \ref{sect-stab} we define various stabilizers associated to a group 
of homeomorphisms, including the local subgroups of ample groups.  In 
Section \ref{subgroups} we introduce generalized permutations and, in 
particular, generalized $2$-cycles.  The generalized $2$-cycles are used in 
most constructions in our paper.  In Section \ref{prop} we introduce a 
number of useful properties that groups of homeomorphisms can have, 
including Property \ref{prop-E}.  Section \ref{maxsub} is devoted to the 
study of maximal subgroups of ample groups.  In that section we prove 
slightly generalized versions of Theorems \ref{main1}, \ref{main2}, 
\ref{main3}, \ref{main4}, \ref{main5} and \ref{main6} (respectively 
Theorems \ref{maxsub-nonminimal}, \ref{maxsub-finite2}, 
\ref{maxsub-not-clopen}, \ref{maxsub-clopen}, \ref{maxsub-3-or-more} and 
\ref{maxsub-2-or-more}).  In Section \ref{nowhere-dense} we describe a 
construction of closed, nowhere dense sets and groups of homeomorphisms 
acting on them that allows to provide a wealth of examples for Theorem 
\ref{main3}.  In Section \ref{property-E} we first discuss Property 
\ref{prop-E} in detail, then present examples of ample groups with that 
property.  In particular, we prove Theorem \ref{main7} (Theorem 
\ref{E-for-cyclic-amplified}).

\medskip

\emph{Acknowledgements}.
The authors are grateful to the anonymous referee whose numerous comments 
and suggestions helped to improve the exposition and sharpen the results of 
the paper substantially.

The first author is supported by the Humboldt Foundation and expresses his 
gratitude to the University of Bielefeld.  Also, he is supported by the 
Travel Support for Mathematicians grant MP-TSM-00002045 from the Simons 
Foundation.

\section{Amplification}\label{sect-amplify}

Let $X$ be a topological space.  We denote by $C(X,X)$ the set of all 
continuous maps $f:X\to X$.  This set is a semigroup (and a monoid) with 
respect to the composition of maps.  The identity function $\id_X:X\to X$ 
is the identity element.  Homeomorphisms of $X$ are invertible elements of 
$C(X,X)$.  The set $\Homeo(X)$ of all homeomorphisms is a group.

We assume that the topological space $X$ is compact, metrizable, and 
totally disconnected.  These conditions imply that the topology is 
generated by \emph{clopen} (that is, both closed and open) sets.  The main 
example is a Cantor set, but a finite set with the discrete topology is an 
example as well.  If $Y$ is a nonempty closed subset of $X$, then $Y$ as 
the topological space with the induced topology also satisfies those 
conditions.

Several properties of the topological space $X$ will be used repeatedly 
throughout the paper.  First of all, any open neighborhood of a point $x\in 
X$ contains a clopen neighborhood of the same point $x$.  Secondly, for any 
distinct points $x,y\in X$ there exists a clopen neighborhood of $x$ that 
does not contain $y$.  It easily follows by induction that for any finite 
number of distinct points in $X$ we can choose pairwise disjoint clopen 
neighborhoods.

\begin{proposition}\label{amp-local}
Given a subset $S\subset C(X,X)$ and an element $f\in C(X,X)$, the 
following conditions are equivalent:
\begin{itemize}
\item[(i)]
$f$ is \textbf{locally an element of} $S$, which means that for any point 
$x\in X$ there exists an open neighborhood $U_x$ of $x$ and an element 
$g_x\in S$ such that $f$ coincides with $g_x$ on $U_x$: 
$f|_{U_x}=g_x|_{U_x}$;

\item[(ii)]
$f$ is \textbf{piecewise an element of} $S$, which means that there exist 
clopen sets $V_1,\dots,V_k$ forming a partition of $X$ and elements 
$h_1,h_2,\dots,h_k\in S$ such that $f|_{V_i}=h_i|_{V_i}$ for $1\le i\le k$.
\end{itemize}
\end{proposition}

\begin{proof}
The implication (ii)$\implies$(i) is trivial.  Let us prove that 
(i)$\implies$(ii).  Assume condition (i) holds.  Without loss of 
generality, we may further assume that the neighborhoods $U_x$, $x\in X$ 
are clopen.  The sets $U_x$ form an open cover of the compact space $X$.  
Hence there are finitely many points $x_1,x_2,\dots,x_k\in X$ such that the 
sets $U_{x_i}$, $1\le i\le k$ form a subcover.  Now let $V_1=U_{x_1}$ and 
$V_i=U_{x_i}\setminus (U_{x_1}\cup\dots\cup U_{x_{i-1}})$ for 
$i=2,3,\dots,k$.  Then the sets $V_1,V_2,\dots,V_k$ form a partition of 
$X$.  Since each $U_{x_i}$, $1\le i\le k$ is clopen, it follows that the 
sets $V_i$, $1\le i\le k$ are clopen as well.  Let $h_i=g_{x_i}$ for $1\le 
i\le k$.  Clearly, $f$ coincides with $h_i$ on $V_i$ for $1\le i\le k$.  
Thus condition (ii) holds.
\end{proof}

\begin{definition}
Given a set $S\subset C(X,X)$, let $\hF(S)$ denote the set of all 
continuous maps in $C(X,X)$ that are locally elements of $S$ (that is, 
satisfy the two conditions in Proposition \ref{amp-local}).  Given a set 
$S\subset\Homeo(X)$, let $\F(S)=\hF(S)\cap\Homeo(X)$.
\end{definition}

Clearly, $S\subset\hF(S)$ for any $S\subset C(X,X)$.  If 
$S\subset\Homeo(X)$, then $S\subset\F(S)\subset\hF(S)$.

\begin{lemma}\label{amp-amp}
$\hF(\hF(S))=\hF(S)$ for any $S\subset C(X,X)$.  $\F(\F(S))=\F(S)$ for any 
$S\subset\Homeo(X)$.
\end{lemma}

\begin{proof}
As the inclusion $\hF(S)\subset\hF(\hF(S))$ is obvious, we need to prove 
that $\hF(\hF(S))\subset\hF(S)$.  Suppose $f\in\hF(\hF(S))$ and consider an 
arbitrary point $x\in X$.  Since $f$ is locally an element of $\hF(S)$, 
there exists an open neighborhood $U$ of $x$ and a map $g\in\hF(S)$ such 
that $f|_U=g|_U$.  Since $g$ is locally an element of $S$, there exists an 
open neighborhood $W$ of $x$ and a map $h\in S$ such that $g|_W=h|_W$.  
Then $f$ coincides with $h$ on $U\cap W$, which is an open neighborhood of 
$x$.  We conclude that $f$ is locally an element of $S$.

In the case $S\subset\Homeo(X)$, the inclusion $\F(S)\subset\F(\F(S))$ is 
obvious.  Since $\F(S)\subset\hF(S)$, it follows that 
$\hF(\F(S))\subset\hF(\hF(S))$.  By the above, $\hF(\hF(S))=\hF(S)$.  
Therefore $\F(\F(S))=\hF(\F(S))\cap\Homeo(X)\subset\hF(S)\cap\Homeo(X) 
=\F(S)$.  Thus $\F(\F(S))=\F(S)$.
\end{proof}

\begin{lemma}\label{ample-semigroup}
If $S\subset C(X,X)$ is a semigroup, then $\hF(S)$ is also a semigroup.  If
$S\subset\Homeo(X)$ is a group, then $\F(S)$ is also a group.
\end{lemma}

\begin{proof}
Suppose $S\subset C(X,X)$ is a semigroup and consider any two maps 
$f,g\in\hF(S)$.  Given a point $x\in X$, there exists an open neighborhood 
$U$ of $x$ and a map $\tilde g\in S$ such that $g|_U=\tilde g|_U$.  
Further, there exists an open neighborhood $W$ of the point $g(x)$ and a 
map $\tilde f\in S$ such that $f|_W=\tilde f|_W$.  Then the composition 
$fg$ coincides with $\tilde f\tilde g$, which belongs to $S$, on 
$g^{-1}(W)\cap U$, which is an open neighborhood of $x$.  Hence $fg$ is 
locally an element of $S$.  Thus $\hF(S)$ is a semigroup.

Now assume $S\subset\Homeo(X)$ is a group.  By the above $\hF(S)$ is a 
semigroup.  Then $\F(S)$ is also a semigroup as it is the intersection of 
two semigroups $\hF(S)$ and $\Homeo(X)$.  It remains to show that for any 
homeomorphism $f\in\F(S)$, the inverse $f^{-1}$ is also in $\F(S)$.  Take 
an arbitrary point $x\in X$.  Since $f$ is locally an element of $S$, there 
exists an open neighborhood $U$ of the point $f^{-1}(x)$ and a map $g\in S$ 
such that $f|_U=g|_U$.  Then $f^{-1}$ coincides with $g^{-1}$, which 
belongs to $S$, on $f(U)$, which is an open neighborhood of $x$.  Thus 
$f^{-1}$ is locally an element of $S$.
\end{proof}

\begin{definition}
Given subgroups $G$ and $\widetilde G$ of $\Homeo(X)$, we say that 
$\widetilde G$ \textbf{amplifies} $G$ if $G\subset\widetilde 
G\subset\F(G)$.  The group $\F(G)$ is called the \textbf{full 
amplification} of $G$.  In the case $\F(G)=G$, the group $G$ is called 
\textbf{ample} or \textbf{fully amplified}.
\end{definition}

For any group $G\subset\Homeo(X)$, it follows from Lemmas \ref{amp-amp} and 
\ref{ample-semigroup} that $\F(G)$ is an ample group.

\begin{definition}
Given a group $G\subset\Homeo(X)$, for any $x\in X$ we denote by 
$\Orb_G(x)$ the \textbf{orbit} of the point $x$ under the natural action of 
the group: $\Orb_G(x)=\{g(x)\mid g\in G\}$.  We refer to any set of the 
form $\Orb_G(x)$ as an orbit of the group $G$.
\end{definition}

An important property of the amplification is that this procedure enlarges 
a group while not enlarging orbits of the natural action.  In particular, 
the ample group $\F(G)$ has the same orbits as the group $G$.

\section{Stabilizers and local subgroups}\label{sect-stab}

Any group $G\subset\Homeo(X)$ acts naturally on the topological space $X$.  
This action induces several other actions.  In this section we define 
various stabilizers related to those actions.  Note that the stabilizers of 
any group action on any set are subgroups of the acting group.

\begin{definition}
The \textbf{stabilizer} $\St_G(x)$ of a point $x\in X$ under the action of 
the group $G$ consists of those elements of $G$ that fix $x$: 
$\St_G(x)=\{g\in G\mid g(x)=x\}$.
\end{definition}

The action of the group $G$ on the set $X$ induces an action on subsets of 
$X$.

\begin{definition}\label{def-stab-set}
The \textbf{(set) stabilizer} $\St_G(Y)$ of a set $Y\subset X$ under the 
action of the group $G$ consists of those elements of $G$ that map $Y$ onto 
itself.  The \textbf{pointwise stabilizer} $\Stp_G(Y)$ of $Y$ consists of 
those elements of $\St_G(Y)$ that fix every point of the set $Y$.  The 
\textbf{rigid stabilizer} $\RiSt_G(Y)$ of $Y$ consists of those elements of 
$\St_G(Y)$ that fix every point not in $Y$.
\end{definition}

Clearly, the stabilizer $\St_G(x)$ of a point $x\in X$ coincides with the 
stabilizer $\St_G(\{x\})$ of the one-point set $\{x\}$ as well as with 
$\Stp_G(\{x\})$.  For any subset $Y\subset X$, the individual stabilizer 
$\Stp_G(Y)$ is the intersection of all point stabilizers $\St_G(y)$, $y\in 
Y$.

The pointwise stabilizer can be interpreted as the usual stabilizer of an 
induced action of $G$ on ordered subsets of $X$.  Consider the set of all 
pairs $(Y,\prec)$, where $Y\subset X$ and $\prec$ is an order on $Y$.  The 
group $G$ acts naturally on this set.  Clearly, the stabilizer of a pair 
$(Y,\prec)$ always contains $\Stp_G(Y)$.  In the case $\prec$ is a 
well-ordering on the subset $Y$, it is easy to observe that the stabilizer 
of $(Y,\prec)$ is exactly $\Stp_G(Y)$.

For any invertible map $f:X\to X$ we have $f(Y)=Y$ if and only if 
$f(X\setminus Y)=X\setminus Y$.  Hence $\St_G(Y)=\St_G(X\setminus Y)$.  It 
follows that $\RiSt_G(Y)=\Stp_G(X\setminus Y)$.

The set stabilizer $\St_G(Y)$ acts naturally on the set $Y$.  The pointwise 
stabilizer $\Stp_G(Y)$ is the kernel of that action.  Therefore it is 
a normal subgroup of $\St_G(Y)$.  Similarly, $\RiSt_G(Y)$ is a normal 
subgroup of $\St_G(X\setminus Y)=\St_G(Y)$. 

\begin{lemma}\label{ample-stab-set}
Suppose $G\subset\Homeo(X)$ is an ample group.  Then for any subset 
$Y\subset X$ the stabilizers $\St_G(Y)$, $\Stp_G(Y)$ and $\RiSt_G(Y)$ are 
also ample groups.
\end{lemma}

\begin{proof}
Since the set stabilizer $\St_G(Y)$ is a subgroup of $G$, its full 
amplification $\F(\St_G(Y))$ is a subgroup of $\F(G)=G$.  Take any map
$f\in\F(\St_G(Y))$.  Then $f$ is piecewise an element of $\St_G(Y)$, which 
implies that $f(Y)\subset Y$.  Since $\St_G(Y)=\St_G(X\setminus Y)$, it 
also follows that $f(X\setminus Y)\subset X\setminus Y$.  As $f$ is 
invertible, we obtain $f(Y)=Y$ so that $f\in\St_G(Y)$.  Thus the group 
$\St_G(Y)$ is ample.

If $Y$ is the empty set or $X$ then one of the groups $\Stp_G(Y)$ and 
$\RiSt_G(Y)$ is trivial while the other coincides with $G$, and both are 
ample.  Now assume $Y$ is different from $\emptyset$ and $X$.  By the above 
for any point $x\in X$ the stabilizer $\St_G(x)=\St_G(\{x\})$ is an ample 
group.  Therefore the point stabilizer $\Stp_G(Y)$ is the intersection of 
ample groups $\St_G(x)$, $x\in Y$ while the rigid stabilizer $\RiSt_G(Y)$ 
is the intersection of ample groups $\St_G(x)$, $x\notin Y$.  It is easy to 
observe that for any collection of ample subgroups of $\Homeo(X)$, their 
intersection is ample as well.
\end{proof}

The action of the group $G$ on subsets of $X$ induces an action on 
collections of such subsets (that is, on the set of subsets of the set of 
subsets of $X$).

\begin{definition}\label{def-stab-coll}
The \textbf{(collective) stabilizer} $\St_G(Y_1,Y_2,\dots,Y_k)$ of a 
collection of distinct sets $Y_i\subset X$, $1\le i\le k$ under the action 
of the group $G$ consists of those elements of $G$ that map each set in the 
collection onto (the same or another) set in the collection.  The 
\textbf{individual stabilizer} $\Sti_G(Y_1,Y_2,\dots,Y_k)$ of the 
collection consists of those elements of $\St_G(Y_1,Y_2,\dots,Y_k)$ that 
map each set $Y_i$, $1\le i\le k$ onto itself.
\end{definition}

Clearly, the individual stabilizer $\Sti_G(Y_1,Y_2,\dots,Y_k)$ is the 
intersection of the set stabilizers $\St_G(Y_i)$, $1\le i\le k$.  In the 
case when the collection consists of a single set $Y$, the notation 
$\St_G(Y)$ is not ambiguous since the collective stabilizer of the 
collection $\{Y\}$ coincides with the set stabilizer of the set $Y$.

The collective stabilizer $\St_G(Y_1,Y_2,\dots,Y_k)$ acts naturally on the 
set $\{Y_1,Y_2,\dots,Y_k\}$.  The individual stabilizer 
$\Sti_G(Y_1,Y_2,\dots,Y_k)$ is the kernel of that action.  Therefore it is 
a normal subgroup of $\St_G(Y_1,Y_2,\dots,Y_k)$.

\begin{definition}
Let $f\in\Homeo(X)$.  A point $x\in X$ is called a \textbf{fixed point} of 
$f$ if $f(x)=x$.  The set of all fixed points of $f$ is denoted $\Fix(f)$.  
Further, a point $x\in X$ is called a \textbf{support point} of $f$ if the 
homeomorphism $f$ does not coincide with the identity map in any 
neighborhood of $x$.  The set of all support points of $f$ is called the 
\textbf{support} of $f$ and denoted $\supp(f)$.
\end{definition}

By definition, the support of $f$ is the complement of the interior of 
$\Fix(f)$.  Equivalently, $\supp(f)=\overline{X\setminus\Fix(f)}$, the 
closure of the complement of $\Fix(f)$.

If two homeomorphisms $f$ and $g$ have disjoint supports then, clearly, 
they commute: $fg=gf$.  In fact, $f$ and $g$ commute whenever $\supp(f)$ 
and $\supp(g)$ share no interior point.

\begin{definition}
Let $G$ be a subgroup of $\Homeo(X)$ and $U$ be a clopen subset of $X$.  
The \textbf{local subgroup} of the ample group $\F(G)$ associated to $U$, 
denoted $\F_U(G)$, consists of all maps $f\in\F(G)$ such that 
$\supp(f)\subset U$.  In the case the group $G$ is already ample, we may 
use alternative notation $G_U$.
\end{definition}

For any clopen set $U\subset X$ and homeomorphism $f:X\to X$, the condition 
$\supp(f)\subset U$ is equivalent to the condition $f(x)=x$ for all 
$x\notin U$.  Hence the local subgroup $\F_U(G)$ coincides with the rigid 
stabilizer $\RiSt_{\F(G)}(U)$.  By Lemma \ref{ample-stab-set}, $\F_U(G)$ is 
an ample group.

\section{Generalized permutations}\label{subgroups}

For any integer $n\ge1$ let $\sS_n$ denote the \emph{symmetric group} of 
all permutations on the set $\{1,2,\dots,n\}$.

\begin{definition}\label{def-gen-perm}
Let $U$ be a clopen subset of a topological space $X$.  Suppose 
$f_1,f_2,\dots,f_n$ are homeomorphisms of $X$ such that the images 
$f_1(U),f_2(U),\dots,f_n(U)$ are disjoint.  Then for any permutation 
$\pi\in\sS_n$ we define a \textbf{generalized permutation} 
$\mu[U;f_1,f_2,\dots,f_n;\pi]:X\to X$ by
\[
\mu[U;f_1,f_2,\dots,f_n;\pi](x)=
\begin{cases}
f_{\pi(i)}(f_i^{-1}(x)) & \text{if $x\in f_i(U)$, $1\le i\le n$,}\\
x & \text{otherwise.}
\end{cases}
\]
\end{definition}

The generalized permutation $\mu[U;f_1,f_2,\dots,f_n;\pi]$ is a 
homeomorphism of $X$ that permutes disjoint clopen sets 
$f_1(U),f_2(U),\dots,f_n(U)$ while fixing the rest of the space.  For any 
$x\in U$ and $i\in\{1,2,\dots,n\}$ the point $f_i(x)$ is mapped to 
$f_{\pi(i)}(x)$.  The generalized permutation has finite order in the group 
$\Homeo(X)$, which is the same as the order of the permutation $\pi$.  If 
the sets $f_1(U),f_2(U),\dots,f_n(U)$ are not disjoint then 
$\mu[U;f_1,f_2,\dots,f_n;\pi]$ is not defined.

\begin{lemma}\label{gen-perm-homo}
Suppose $U\subset X$ is a clopen set and $f_1,f_2,\dots,f_n$ are 
homeomorphisms of $X$ such that the images $f_1(U),f_2(U),\dots,f_n(U)$ are 
disjoint.  Then the map $\Phi:\sS_n\to\Homeo(X)$ given by 
$\Phi(\pi)=\mu[U;f_1,f_2,\dots,f_n;\pi]$ is a group homomorphism.
\end{lemma}

The proof of Lemma \ref{gen-perm-homo} is straightforward and we omit it.

Let $G$ be a subgroup of $\Homeo(X)$ and $\F(G)$ be its full amplification.
If homeomorphisms $f_1,f_2,\dots,f_n$ belong to the group $G$, then any 
generalized permutation of the form $\mu[U;f_1,f_2,\dots,f_n;\pi]$ belongs 
to $\F(G)$.

\begin{definition}
The \textbf{generalized symmetric group} $\sS(G)$ over $G$ is the subgroup 
of $\F(G)$ generated by all generalized permutations in $\F(G)$.
\end{definition}

The group $\sS(G)$ was introduced (in a more general context of \'{e}tale 
groupoids) by Nekrashevych \cite{Nek16}.

\begin{lemma}\label{gen-perm-normal-subgroup}
$\sS(G)$ is a normal subgroup of the group $\F(G)$.
\end{lemma}

\begin{proof}
It is enough to show that any conjugate of a generalized permutation in the 
group $\Homeo(X)$ is itself a generalized permutation.  After that the 
proof of the lemma is straightforward.  Suppose a generalized permutation 
$g=\mu[U;f_1,f_2,\dots,f_n;\pi]$ is defined.  Then 
$f_1(U),f_2(U),\dots,f_n(U)$ are disjoint clopen sets.  The 
support of $g$ is the union of those $f_i(U)$ for which $\pi(i)\ne i$.  For 
any $x\in U$ and $i\in\{1,2,\dots,n\}$ we have $g(f_i(x))=f_{\pi(i)}(x)$.  
Now consider an arbitrary homeomorphism $h\in\Homeo(X)$.  It is easy to 
observe that $\supp(hgh^{-1})=h\bigl(\supp(g)\bigr)$.  Hence the support of 
$hgh^{-1}$ is contained in the union of sets 
$hf_1(U),hf_2(U),\dots,hf_n(U)$, which are disjoint clopen sets.  
Furthermore, for any $x\in U$ and $i\in\{1,2,\dots,n\}$ we have 
$hgh^{-1}(hf_i(x))=hg(f_i(x))=hf_{\pi(i)}(x)$.  It follows that 
$hgh^{-1}=\mu[U;hf_1,hf_2,\dots,hf_n;\pi]$.  Thus $hgh^{-1}$ is a 
generalized permutation.
\end{proof}

\begin{lemma}\label{gen-sym}
For any group $G\subset\Homeo(X)$ the generalized symmetric group $\sS(G)$ 
is generated by all maps of the form $\mu[U;f_1,f_2,\dots,f_n;\pi]$, where 
$f_1,f_2,\dots,f_n\in\F(G)$.
\end{lemma}

\begin{proof}
Let $P$ be the set of all generalized permutations of the form 
$\mu[U;f_1,f_2,\dots,f_n;\pi]$, where $f_1,f_2,\dots,f_n\in\F(G)$.  It 
follows from the definition that any element of $P$ is piecewise an element 
of the group $\F(G)$.  Therefore $P$ is contained in the group $\F(\F(G))$, 
which coincides with $\F(G)$ due to Lemma \ref{amp-amp}.

To prove the lemma, it is enough to show that any generalized permutation 
$g\in\F(G)$ is a product of elements of the set $P$.  The map $g$ is 
represented as $\mu[U;f_1,f_2,\dots,f_n;\pi]$, where maps 
$f_1,f_2,\dots,f_n$ need not be in $\F(G)$.  Let $\pi=\si_1\si_2\dots\si_k$ 
be the decomposition of the permutation $\pi$ as a product of disjoint 
cycles.  Lemma \ref{gen-perm-homo} implies that $g=g_1g_2\dots g_k$, where 
$g_j=\mu[U;f_1,f_2,\dots,f_n;\si_j]$, $1\le j\le k$.  We are going to show 
that each $g_j$ belongs to $P$. 

First let us show that $g_j\in\F(G)$.  For any $i\in\{1,2,\dots,n\}$ we 
have either $\si_j(i)=\pi(i)$ or $\si_j(i)=i$.  In the former case, $g_j$ 
coincides with $g$ on the clopen set $f_i(U)$.  In the latter case, $g_j$ 
coincides with the identity map on $f_i(U)$.  On the complement of the 
union of sets $f_i(U)$, $1\le i\le n$, both $g$ and $g_j$ coincide with the 
identity map.  Hence $g_j$ is piecewise an element of the set 
$\{\id_X,g\}$.  Since $g\in\F(G)$, it follows that $g_j\in\F(\F(G))=\F(G)$.

Recall that the permutation $\si_j$ is a cycle: $\si_j=(i_1\,i_2\dots 
i_m)$, where $i_1,i_2,\dots,i_m$ are distinct elements of 
$\{1,2,\dots,n\}$.  Let $W_s=f_{i_s}(U)$, $1\le s\le m$.  Then 
$W_1,W_2,\dots,W_m$ are disjoint clopen sets and $g_j$ coincides with the 
identity map away from them.  We have $g_j(W_s)=W_{s+1}$ for $1\le s\le 
m-1$ and $g_j(W_m)=W_1$.  It follows by induction that $W_s=g_j^{s-1}(W_1)$ 
for $1\le s\le m$.  Since the cycle $\si_j$ has order $m$ in the group 
$\sS_n$, Lemma \ref{gen-perm-homo} implies that $g_j^m=\id_X$.  As a 
consequence, $(g_j^{m-1})^{-1}$ coincides with $g_j$ on $W_m$ (as well as 
anywhere else on $X$).  All this leads to an alternative representation of 
$g_j$ as a generalized permutation:
\[
g_j=\mu\bigl[W_1;\id_X,g_j,g_j^2,\dots,g_j^{m-1};(1\,2\dots m)\bigr].
\]
Since $g_j$ and all its powers belong to the group $\F(G)$, we conclude 
that $g_j\in P$.
\end{proof}

\begin{definition}\label{def-2-cycle}
Let $U$ be a clopen subset of $X$.  For any homeomorphism $f:X\to X$ such 
that the image $f(U)$ is disjoint from $U$, we define a \textbf{generalized 
2-cycle} $\de_{U;f}:X\to X$ by
\[
\de_{U;f}(x)=
\begin{cases}
f(x) & \text{if $x\in U$,}\\
f^{-1}(x) & \text{if $x\in f(U)$,}\\
x & \text{otherwise.}
\end{cases}
\]
\end{definition}

The generalized $2$-cycles are a special case of the generalized 
permutations.  Indeed, $\de_{U;f}=\mu[U;\id_X,f;(1\,2)]$.  As a 
consequence, $\de_{U;f}\in\sS(G)$ whenever $f\in G$.

Recall that the symmetric group $\sS_n$ is generated by all $2$-cycles 
$(i\,j)$ in $\sS_n$.

\begin{lemma}\label{generated-by-2-cycles}
For any group $G\subset\Homeo(X)$, a generalized $2$-cycle belongs to 
$\F(G)$ if and only if it is of the form $\de_{U;f}$, where 
$f\in\F(G)$.  The generalized symmetric group $\sS(G)$ is generated by all 
generalized $2$-cycles in $\F(G)$.
\end{lemma}

\begin{proof}
If a generalized $2$-cycle $g=\de_{U;f}$ is defined and $f\in\F(G)$, then 
$g$ is piecewise an element of the group $\F(G)$.  Hence it belongs to the 
group $\F(\F(G))$, which coincides with $\F(G)$ due to Lemma 
\ref{amp-amp}.  Conversely, if $g=\de_{U;f}$ is defined and $g\in\F(G)$, we 
observe that the generalized $2$-cycle $\de_{U;g}$ is also defined and 
$g=\de_{U;g}$.

By Lemma \ref{gen-sym}, the group $\sS(G)$ is generated by generalized 
permutations of the form $\mu[U;f_1,f_2,\dots,f_n;\pi]$, where 
$f_1,f_2,\dots,f_n\in\F(G)$.  Those include all maps of the form 
$\de_{U;f}$, $f\in\F(G)$.  Whenever $\mu[U;f_1,f_2,\dots,f_n;\pi]$ is
defined, all maps $\mu[U;f_1,f_2,\dots,f_n;\si]$, $\si\in\sS_n$ are defined 
as well.  Assuming $f_1,f_2,\dots,f_n\in\F(G)$, they all belong to the
group $\F(\F(G))=\F(G)$.  Moreover, 
$\sS_n\ni\si\mapsto\mu[U;f_1,f_2,\dots,f_n;\si]$ is a homomorphism of the 
group $\sS_n$ to $\F(G)$ due to Lemma \ref{gen-perm-homo}.  Since the group 
$\sS_n$ is generated by $2$-cycles $(i\,j)$, it follows that the group 
$\sS(G)$ is generated by generalized permutations of the form 
$\mu[U;f_1,f_2,\dots,f_n;(i\,j)]$.  It remains to show that each 
$\mu[U;f_1,f_2,\dots,f_n;(i\,j)]$ is a generalized $2$-cycle of prescribed 
form.  First we observe that 
$\mu[U;f_1,f_2,\dots,f_n;(i\,j)]=\mu[U;f_i,f_j;(1\,2)]$.  Then we observe 
that $\mu[U;f_i,f_j;(1\,2)]=\de_{f_i(U);f_jf_i^{-1}}$.  Finally, we note 
that $f_jf_i^{-1}\in\F(G)$.
\end{proof}

\section{Regularity properties}\label{prop}

In this section we formulate several useful properties that a subgroup $G$
of $\Homeo(X)$ may or may not have, and establish relations between them.  
When proving statements on the ample group $\F(G)$, we will usually need 
some of those properties.  Let us begin with a very well known property 
that will be required most of the time.

\begin{definition}\label{def-minimality}
We say that a group $G\subset\Homeo(X)$ acts \textbf{minimally} on $X$ (or 
that the natural action of $G$ on $X$ is \textbf{minimal}) if there are no 
closed sets invariant under the action other than the empty set and $X$ 
itself.
\end{definition}

A set $Y\subset X$ is \emph{invariant} under a homeomorphism $f:X\to X$ if 
$f(Y)\subset Y$.  It is invariant under the natural action of a group 
$G\subset\Homeo(X)$ if $g(Y)\subset Y$ for all $g\in G$.  For any $g\in G$, 
the inverse map $g^{-1}$ is in $G$ as well.  Assuming invariance of the set 
$Y$, we have $g(Y)\subset Y$ and $g^{-1}(Y)\subset Y$, which implies that, 
in fact, $g(Y)=Y$ for all $g\in G$.

For any point $x\in X$ the closure $\overline{\Orb_G(x)}$ of its orbit is a 
closed set invariant under the action of the group $G$.  It easily follows 
that the action is minimal if and only if every orbit is dense in $X$.

\begin{lemma}\label{finite-or-Cantor}
Suppose a topological space $X$ is totally disconnected, compact and 
metrizable.  If there is a group $G\subset\Homeo(X)$ that acts minimally on 
$X$, then $X$ is either finite or a Cantor set.
\end{lemma}

\begin{proof}
If a point $x\in X$ is isolated, then the set $\{x\}$ is clopen.  It 
follows that the set $Y$ of all non-isolated points of $X$ is closed.  Any 
homeomorphism of $X$ maps an isolated point to an isolated point.  Hence 
the set $Y$ is invariant under the action of any group 
$G\subset\Homeo(X)$.  If the action is minimal, then either $Y$ is empty or 
$Y=X$.  In the case $Y=X$, there are no isolated points.  Then $X$ is a 
Cantor set due to Brouwer's characterization of Cantor sets.  In the case 
$Y$ is empty, every point of $X$ is isolated.  Hence the partition of $X$ 
into points is a partition into clopen sets.  Then compactness of $X$ 
implies that it is finite.
\end{proof}

The following technical statement will be used several times in our paper.

\begin{lemma}\label{map-piecewise}
Let $U,U'\subset X$ be nonempty clopen sets.  If a group 
$G\subset\Homeo(X)$ acts minimally on $X$, then for some $k\ge1$ there 
exist disjoint clopen sets $V_1,V_2,\dots,V_k$ and maps 
$g_1,g_2,\dots,g_k\in G$ such that $U=V_1\sqcup V_2\sqcup\dots\sqcup V_k$ 
and $g_i(V_i)\subset U'$ for all $i$.
\end{lemma}

\begin{proof}
Since the group $G$ acts minimally, the orbit of any point $x\in U$ is 
dense in $X$.  In particular, the orbit has a point in the nonempty clopen 
set $U'$.  Hence $g_x(x)\in U'$ for some map $g_x\in G$.  Let $U_x=U\cap 
g_x^{-1}(U')$.  Then $U_x$ is a clopen subset of $U$ containing the point 
$x$, and $g_x(U_x)\subset U'$.  The clopen sets $U_x$, $x\in U$ cover $U$.  
Since $U$ is compact, there are finitely many points $x_1,x_2,\dots,x_k\in 
U$ such that the sets $U_{x_i}$, $1\le i\le k$ form a subcover.  Let 
$V_1=U_{x_1}$ and $V_i=U_{x_i}\setminus (U_{x_1}\cup\dots\cup U_{x_{i-1}})$ 
for $i=2,3,\dots,k$.  Then $V_1,V_2,\dots,V_k$ are clopen sets that 
partition $U$.  By construction, $g_{x_i}(V_i)\subset 
g_{x_i}(U_{x_i})\subset U'$ for all $i$.
\end{proof}

The other properties formulated in this section are less (if at all) known, 
and we will introduce an initialism for each of them.

\begin{property}{UR}[universal recurrence]\label{prop-UR}
For any open set $U\subset X$, any point $x\in U$ and any element $g\in G$ 
there exists an integer $k\ge1$ such that $g^k(x)\in U$.
\end{property}

A point $x\in X$ is called \emph{recurrent} relative to a homeomorphism 
$f:X\to X$ if the sequence $f(x),f^2(x),f^3(x),\dots$ visits every 
neighborhood of $x$.  Property \ref{prop-UR} means that every point of $X$ 
is recurrent relative to every element of the group $G$.

\begin{property}{NC}[no contraction]\label{prop-NC}
If $g(U)\subset U$ for some element $g\in G$ and open set $U\subset X$ 
then, in fact, $g(U)=U$.
\end{property}

Open sets could be replaced by closed sets in the formulation of Property 
\ref{prop-NC}.  Indeed, a homeomorphism $f\in\Homeo(X)$ contracts an open 
set $U\subset X$ (that is, maps it onto a proper subset of $U$) if and only 
if the inverse map $f^{-1}$ contracts the closed set $f(X\setminus U)$.

Let us mention two obvious cases in which no contraction of any sets, open 
or not, is possible.  The first case is when the topological space $X$ is 
finite.  The second case is when every element of the group $G$ has finite 
order (this actually includes the first case).

\begin{lemma}\label{NC-equiv-UR}
Property \ref{prop-NC} is equivalent to Property \ref{prop-UR}.
\end{lemma}

\begin{proof}
Assume that the group $G\subset\Homeo(X)$ has Property \ref{prop-NC}.  Take
any open set $U\subset X$ and element $g\in G$.  Since $g$ is a
homeomorphism, the sets $g^{-1}(U),g^{-2}(U),\dots$ are open as well, and
so is the union $W=U\cup g^{-1}(U)\cup g^{-2}(U)\cup\dots$.  Note that
$g^{-1}(W)=g^{-1}(U)\cup g^{-2}(U)\cup g^{-3}(U)\cup\dots$, which implies
that $g^{-1}(W)\subset W$.  Since $g^{-1}\in G$ and the group $G$ has
Property \ref{prop-NC}, we have $g^{-1}(W)=W$.  As a consequence, any point
$x\in U$ belongs to $g^{-k}(U)$ for some $k\ge1$.  Then $g^k(x)\in U$.
Thus $G$ has Property \ref{prop-UR}.

Conversely, assume that the group $G$ does not have Property \ref{prop-NC}.
Then there exist an open set $U\subset X$ and an element $g\in G$ such that
$g(U)$ is a proper subset of $U$.  The inverse map $h=g^{-1}$ is also an
element of $G$.  Take any point $x\in U$ not in $g(U)$.  The condition
$g(U)\subset U$ implies that $h(X\setminus g(U))\subset X\setminus U$ and
$h(X\setminus U)\subset X\setminus U$.  It follows that the sequence
$h(x),h^2(x),h^3(x),\dots$ never visits $U$.  We conclude that the group
$G$ does not have Property \ref{prop-UR}.
\end{proof}

\begin{lemma}\label{contraction-by-a-point}
Suppose that the group $G$ does not have Property \ref{prop-NC}.  Then
there exist an open set $U\subset X$ and an element $g\in G$ such that
$g(U)\subset U$ and the set difference $U\setminus g(U)$ consists of a
single point.
\end{lemma}

\begin{proof}
Lemma \ref{NC-equiv-UR} implies that the group $G$ does not have Property
\ref{prop-UR}.  Hence there exist an open set $W\subset X$, a point
$x\in W$ and a map $h\in G$ such that the sequence
$h(x),h^2(x),h^3(x),\dots$ never visits $W$.  The inverse map $g=h^{-1}$ is
also an element of $G$.  Let $S=\{h^k(x)\mid k\ge1\}$ and
$U=X\setminus\overline{S}$, the complement of the closure of $S$.  Since
$g$ is a homeomorphism, it follows that $g(\overline{S})=\overline{g(S)}$.
It is easy to see that $g(S)=S\cup\{x\}$.  Hence
$\overline{g(S)}=\overline{S}\cup\{x\}$.  Consequently,
$g(U)=X\setminus g(\overline{S})=X\setminus(\overline{S}\cup\{x\})
=U\setminus\{x\}$.  By construction, $U$ is an open set.  Besides,
$W\subset U$ since $W$ is an open set disjoint from $S$.  In particular,
$x\in U$.  Thus $g(U)\subset U$ and $U\setminus g(U)$ consists of a single
point.
\end{proof}

Note that replacing open sets with clopen sets in the formulation of
Property \ref{prop-UR}, we obtain an equivalent property.  On the other
hand, Lemma \ref{contraction-by-a-point} suggests that having no 
contraction of clopen sets may not be enough to have Property \ref{prop-NC}.

\begin{property}{SC}[strong contraction]\label{prop-SC}
For any clopen set $U\subset X$ different from the empty set and $X$ there 
exist maps $f_1,f_2\in G$ such that the images $f_1(U)$ and $f_2(U)$ are 
disjoint subsets of $U$.
\end{property}

In the case the group $G$ is ample, Property \ref{prop-SC} coincides with a 
property introduced by Matui \cite{Matui-pureinf1} in a more general 
context of \'{e}tale groupoids.  Matui calls groupoids with this property 
\emph{purely infinite} (see Definition 4.9 in \cite{Matui-pureinf1}).  Just 
like in \cite{Matui-pureinf1}, we introduce another, closely related 
property.

\begin{property}{UC}[universal contraction]\label{prop-UC}
For any clopen sets $A,B\subset X$ different from the empty set and $X$ 
there exists a map $f\in G$ such that $f(A)\subset B$. 
\end{property}

Properties \ref{prop-SC} and \ref{prop-UC} vacuously hold if the 
topological space $X$ consists of a single point.  Property \ref{prop-UC} 
also holds if $X$ consists of two points and the group $G$ is not trivial.  
Otherwise either of the two properties implies absence of isolated points 
in $X$, which requires $X$ to be a Cantor set.

\begin{lemma}\label{UC-equiv-M-and-SC}
Suppose $X$ is a Cantor set.  Then for any group $G\subset\Homeo(X)$, 
Property \ref{prop-UC} implies minimality and Property \ref{prop-SC}.  If 
the group $G$ is ample then, conversely, minimality and Property 
\ref{prop-SC} imply Property \ref{prop-UC}.
\end{lemma}

\begin{proof}
Assume a group $G\subset\Homeo(X)$ has Property \ref{prop-UC}.  Since the 
topology on $X$ is generated by clopen sets, Property \ref{prop-UC} clearly 
implies that every orbit of the group $G$ is dense.  Hence $G$ acts 
minimally on $X$.  Since $X$ is a Cantor set, any clopen set $U\subset X$ 
different from $\emptyset$ and $X$ consists of more than one point.  
Therefore it can be represented as a disjoint union of two nonempty clopen 
sets $U_1$ and $U_2$.  Property \ref{prop-UC} implies that $f_1(U)\subset 
U_1$ and $f_2(U)\subset U_2$ for some maps $f_1,f_2\in G$.  Then $f_1(U)$ 
and $f_2(U)$ are disjoint subsets of $U$.  Thus the group $G$ has Property 
\ref{prop-SC}.

Conversely, assume the group $G$ is ample, acts minimally on $X$ and has 
Property \ref{prop-SC}.  For any clopen set $W\subset X$ different from 
$\emptyset$ and $X$, there exist maps $f_1,f_2\in G$ such that the images 
$f_1(W)$ and $f_2(W)$ are disjoint subsets of $W$.  Then $f_1(W)$ is a 
proper subset of $W$.  It follows that $f_1^{j+1}(W)$ is a proper subset of 
$f_1^j(W)$ for all $j\ge0$.  Let $D=W\setminus f_1(W)$.  Then 
$f_1^j(D)=f_1^j(W)\setminus f_1^{j+1}(W)$ for all $j\ge0$, which implies 
that $D,f_1(D),f_1^2(D),\dots$ are disjoint subsets of $W$.

Let us show that any clopen set $U\subset X$ different from $\emptyset$ and 
$X$ can be mapped to its complement by an element of $G$.  By the above 
there exist a nonempty clopen set $D$ and a map $f\in G$ such that 
$D,f(D),f^2(D),\dots$ are disjoint subsets of $X\setminus U$.  By Lemma 
\ref{map-piecewise}, for some $k\ge1$ there exist disjoint clopen sets 
$V_1,V_2,\dots,V_k$ and maps $g_1,g_2,\dots,g_k\in G$ such that 
$U=V_1\sqcup V_2\sqcup\dots\sqcup V_k$ and $g_i(V_i)\subset D$ for all 
$i$.  For any $i$, $1\le i\le k$, let $h_i=f^{i-1}g_i$.  Then $h_i\in G$ 
and $h_i(V_i)\subset f^{i-1}(D)$.  Therefore 
$h_1(V_1),h_2(V_2),\dots,h_k(V_k)$ are disjoint subsets of $X\setminus U$.  
Note that the generalized $2$-cycles $\de_{V_i;h_i}$, $1\le i\le k$ are 
well defined.  Since the group $G$ is ample, it contains all of them as 
well as their product $h=\de_{V_1;h_1}\de_{V_2;h_2}\ldots\de_{V_k;h_k}$.  
As the sets $V_1,V_2,\dots,V_k$ and $h_1(V_1),h_2(V_2),\dots,h_k(V_k)$ are 
disjoint from one another, $h(V_i)=\de_{V_i;h_i}(V_i)=h_i(V_i)$ for all $i$ 
so that $h(V_i)\subset X\setminus U$ for all $i$.  Then $h(U)\subset 
X\setminus U$ as well.

Now we can establish Property \ref{prop-UC}.  Take any clopen sets 
$A,B\subset X$ different from the empty set and $X$.  We need to show that 
$A$ can be mapped to $B$ by an element of the group $G$.  First consider 
the case when $A\cup B\ne X$.  Then $C=X\setminus(A\cup B)$ is a nonempty 
clopen set.  By the above, $h_1(A\cup B)\subset C$ and $h_2(X\setminus 
B)\subset B$ for some maps $h_1,h_2\in G$.  Since $C\subset X\setminus B$, 
we obtain that $h_2h_1(A)\subset h_2h_1(A\cup B)\subset h_2(C)\subset 
h_2(X\setminus B)\subset B$.  As for the case when $A\cup B=X$, we have 
$h(A)\subset X\setminus A$ for some $h\in G$.  If $A\cup B=X$ then the 
complement $X\setminus A$ is contained in $B$.
\end{proof}

\begin{property}{CI}[cancellation of intersection]\label{prop-CI}
For any nonempty clopen sets $A,B,C\subset X$ such that $C$ is disjoint 
from $A\cup B$ and the unions $A\cup C$ and $B\cup C$ are different from 
$X$, we have $f(A\cup C)=B\cup C$ for some $f\in G$ if and only if $h(A)=B$
for some $h\in G$.
\end{property}

\begin{definition}\label{def-absorbable}
Given nonempty clopen sets $D,E\subset X$, we say that $D$ can be 
\textbf{absorbed} by $E$ under the action of a group $G\subset\Homeo(X)$ if
$E$ is disjoint from $D$ and $g(D\cup E)=E$ for some $g\in G$.  Further, a 
nonempty clopen set is called \textbf{absorbable} under the action of $G$
if it can be absorbed by another clopen set.
\end{definition}

\begin{lemma}\label{absorbed-by-any}
Suppose a nonempty clopen set $D\subset X$ is absorbable under the action 
of a group $G\subset\Homeo(X)$ that has Property \ref{prop-CI}.  Then $D$ 
can be absorbed by any clopen set $E\subset X\setminus D$ different from 
the empty set and $X\setminus D$.
\end{lemma}

\begin{proof}
Since the set $D$ is absorbable under the action of the group $G$, we have 
$f(D\cup E_0)=E_0$ for some $f\in G$ and clopen set $E_0$ disjoint from 
$D$.  Note that $E_0$ is different from $\emptyset$ and $X\setminus D$.  
Consider an arbitrary clopen set $E\subset X\setminus D$ different from 
$\emptyset$ and $X\setminus D$.  In the case when $E\cap E_0\ne\emptyset$, 
Property \ref{prop-CI} first implies that $g(D\cup(E\cap E_0))=E\cap E_0$ 
for some $g\in G$, then implies that $h(D\cup E)=E$ for some $h\in G$.  
Similarly, in the case when $E\cup E_0\ne X\setminus D$, Property 
\ref{prop-CI} first implies that $g(D\cup(E\cup E_0))=E\cup E_0$ for some 
$g\in G$, then implies that $h(D\cup E)=E$ for some $h\in G$.  Finally, in 
the case when $E\cap E_0=\emptyset$ and $E\cup E_0=X\setminus D$, we have 
$f(E)=D\cup E$ so that $f^{-1}(D\cup E)=E$.  Thus $D$ can be absorbed by 
$E$.
\end{proof}

\begin{lemma}\label{CI-equiv}
Any group $G\subset\Homeo(X)$ with Property \ref{prop-CI} satisfies the 
following two conditions.
\begin{itemize}
\item[(CI-1)]
For any disjoint nonempty clopen sets $A,B,C\subset X$, if $f(A\cup C)= 
B\cup C$ for some $f\in G$ then $h(A)=B$ for some $h\in G$.
\item[(CI-2)]
For any disjoint nonempty clopen sets $D,E_0,E\subset X$, if $f(D\cup E_0)= 
E_0$ for some $f\in G$ then $h(D\cup E)=E$ for some $h\in G$.
\end{itemize}
If the group $G$ is ample then, conversely, conditions CI-1 and CI-2 imply 
Property \ref{prop-CI}.
\end{lemma}

\begin{proof}
Suppose a group $G\subset\Homeo(X)$ has Property \ref{prop-CI}.  For any 
disjoint nonempty sets $A,B,C\subset X$, the unions $A\cup C$ and $B\cup C$ 
are both different from $X$.  Hence condition CI-1 follows directly from 
Property \ref{prop-CI}.  As for condition CI-2, it easily follows from 
Lemma \ref{absorbed-by-any}.

Now let us prove that any ample group $G\subset\Homeo(X)$ satisfying 
conditions CI-1 and CI-2 has Property \ref{prop-CI}.  Given nonempty clopen 
sets $A,B,C\subset X$, where $C$ is disjoint from both $A$ and $B$, $A\cup 
C\ne X$ and $B\cup C\ne X$, we need to show that the set $A\cup C$ can be 
mapped onto $B\cup C$ under the action of $G$ if and only if the set $A$ 
can be mapped onto $B$.  This is trivial when $A=B$.  Consider the case 
when both $A\setminus B$ and $B\setminus A$ are nonempty sets.  If $f(A\cup 
C)=B\cup C$ or $f(A)=B$ for some $f\in G$, condition CI-1 implies that 
$g(A\setminus B)=B\setminus A$ for some $g\in G$.  Since the clopen sets 
$A\setminus B$ and $B\setminus A$ are disjoint, a generalized $2$-cycle 
$h=\de_{A\setminus B;g}$ is well defined.  It belongs to the group $G$ as 
$G$ is ample.  Since $h$ coincides with the identity map on $A\cap B$ and 
on $C$, we obtain that $h(A)=B$ and $h(A\cup C)=B\cup C$.

It remains to consider the case when one of the sets $A$ and $B$ is a 
proper subset of the other.  Without loss of generality, we may assume that 
$B$ is a proper subset of $A$.  Note that the set $B\cup C$ is disjoint 
from $A\setminus B$ and $(A\setminus B)\cup (B\cup C)=A\cup C\ne X$.  Hence 
$X$ is a disjoint union of four nonempty clopen sets $A\setminus B$, $B$, 
$C$ and $X\setminus(A\cup C)$.  Condition CI-2 first implies that the set 
$A\setminus B$ can be absorbed by $B\cup C$ if and only if it can be 
absorbed by $X\setminus(A\cup C)$, then implies that $A\setminus B$ can be 
absorbed by $X\setminus(A\cup C)$ if and only if it can be absorbed by 
$B$.  Consequently, $f(A\cup C)=B\cup C$ for some $f\in G$ if and only if 
$h(A)=B$ for some $h\in G$.
\end{proof}

\begin{lemma}\label{NC-implies-CI}
For any ample group, Property \ref{prop-NC} implies Property \ref{prop-CI}.
\end{lemma}

\begin{proof}
Suppose a group $G\subset\Homeo(X)$ is ample and has Property 
\ref{prop-NC}.  To establish Property \ref{prop-CI}, we are going to show 
that $G$ satisfies conditions CI-1 and CI-2 in Lemma \ref{CI-equiv}.  Since 
the group $G$ allows no contraction of open sets, it also allows no 
absorbable sets.  Hence condition CI-2 holds vacuously.  Let us verify 
condition CI-1.  Consider any disjoint nonempty clopen sets $A,B,C\subset 
X$ such that $f(A\cup C)=B\cup C$ for some $f\in G$.  We need to find an 
element of $G$ that maps $A$ onto $B$.  By Lemma 
\ref{NC-equiv-UR}, Property \ref{prop-NC} is equivalent to Property 
\ref{prop-UR}.  Therefore the group $G$ has Property \ref{prop-UR}.  As a 
consequence, for any point $x\in A$ the sequence $f(x),f^2(x),f^3(x),\dots$ 
visits the set $A$.  In particular, this sequence is not contained in $C$.  
Let $p(x)$ denote the least positive integer $k$ such that $f^k(x)\notin 
C$.  Since $f(A\cup C)=B\cup C$, it follows that $f^{p(x)}(x)\in B$.  
Likewise, Property \ref{prop-UR} implies that for any point $y\in B$ the
sequence $f^{-1}(y),f^{-2}(y),f^{-3}(y),\dots$ visits $B$ and hence is not 
confined to $C$.  Let $m(y)$ denote the least positive integer $k$ such 
that $f^{-k}(y)\notin C$.  Then $f^{-m(y)}(y)\in A$.  For any integer 
$k\ge1$, let $A_k=\{x\in A\mid p(x)=k\}$ and $B_k=\{y\in B\mid m(y)=k\}$.  
The sets $A_1,A_2,A_3,\dots$ form a partition of $A$ while the sets 
$B_1,B_2,B_3,\dots$ form a partition of $B$.  By construction, 
$f^k(A_k)\subset B_k$ and $f^{-k}(B_k)\subset A_k$ for all $k$, which 
implies that $f^k(A_k)=B_k$ for all $k$.

Now we define a map $h:X\to X$ by $h(x)=f^{p(x)}(x)$ if $x\in A$, 
$h(x)=f^{-m(x)}(x)$ if $x\in B$, and $h(x)=x$ otherwise.  For any $k\ge1$ 
the map $h$ coincides with $f^k$ on $A_k$ and with $f^{-k}$ on $B_k$.  
Since $f^k(A_k)=B_k$ for all $k$, it follows that $h$ is an involution.  
Besides, $h(A)=B$.

Notice that all sets of the form $A_k$ and $B_k$ are clopen.  Indeed, $A_k$ 
is the intersection of clopen sets $A$, $f^{-k}(B)$ and $f^{-i}(C)$, $1\le 
i\le k-1$.  Likewise, $B_k$ is the intersection of clopen sets $B$, 
$f^k(A)$ and $f^i(C)$, $1\le i\le k-1$.  Since the set $A$, which is 
compact, is the disjoint union of clopen sets $A_1,A_2,A_3,\dots$, all but 
finitely many of those sets have to be empty.  Likewise, all but finitely 
many of the sets $B_1,B_2,B_3,\dots$ are empty.  As a consequence, the map 
$h$ is piecewise an element of the group $G$.  Since $h$ is invertible, we 
conclude that $h$ belongs to the group $\F(G)$, which coincides with $G$ as 
$G$ is ample.
\end{proof}

\begin{lemma}\label{UC-implies-CI}
For any ample group, Property \ref{prop-UC} implies Property \ref{prop-CI}.
\end{lemma}

\begin{proof}
Suppose a group $G\subset\Homeo(X)$ is ample and has Property 
\ref{prop-UC}.  To establish Property \ref{prop-CI}, we are going to show 
that the group $G$ satisfies conditions CI-1 and CI-2 in Lemma 
\ref{CI-equiv}.  We begin with condition CI-1.  Take any disjoint nonempty 
clopen sets $A,B,C\subset X$ such that $f(A\cup C)=B\cup C$ for some $f\in 
G$.  First consider the case when the image $f(A)$ intersects $B$.  Then 
$A\cap f^{-1}(B)$ is a nonempty clopen set.  Property \ref{prop-UC} implies 
that $g(C)\subset A\cap f^{-1}(B)$ for some $g\in G$.  Let $O_1=g(C)$ and 
$O_2=f(O_1)$.  Then $O_1$ and $O_2$ are nonempty clopen subsets of 
respectively $A$ and $B$.  In particular, $O_1$ and $O_2$ are both disjoint 
from $C$.  Therefore the generalized $2$-cycles $\de_{C;g}$ and 
$\de_{C;fg}$ are well defined.  Since the group $G$ is ample, it contains 
$\de_{C;g}$ and $\de_{C;fg}$ as well as the map $h=\de_{C;fg}f\de_{C;g}$.  
Let us find $h(A)$.  The map $\de_{C;g}$ fixes all points in $A\setminus 
O_1$ while $O_1$ is mapped onto $C$.  Hence $\de_{C;g}$ maps $A$ onto  
$(A\setminus O_1)\cup C=(A\cup C)\setminus O_1$.  Further, $f$ maps $(A\cup 
C)\setminus O_1$ onto $(B\cup C)\setminus O_2$.  Finally, $\de_{C;fg}$ 
coincides with the identity map on $B\setminus O_2$ while mapping $C$ onto 
$O_2$.  We conclude that $h(A)=B$.

Now consider the case when $f(A)$ is disjoint from $B$.  In this case, 
$f(A)\subset C$ and also $f^{-1}(B)\subset C$.  Property \ref{prop-UC} 
implies that $g_1(C)\subset A$ and $g_2(A)\subset B$ for some maps 
$g_1,g_2\in G$.  Let $O_1=g_1(C)$, $O_2=g_2(O_1)$ and $O_3=f^{-1}(O_2)$.  
Then $O_1$, $O_2$ and $O_3$ are nonempty clopen subsets of respectively 
$A$, $B$ and $C$.  In particular, $O_1$ and $O_2$ are both disjoint from 
$C$ and $O_3$.  Therefore the generalized $2$-cycles $\de_{C;g_1}$, 
$\de_{C;g_2g_1}$ and $\de_{O_1;f^{-1}g_2}$ are well defined.  The ample 
group $G$ contains all these generalized $2$-cycles as well as the map 
$h=\de_{C;g_2g_1}f\de_{O_1;f^{-1}g_2}\de_{C;g_1}$.  Let us find $h(A)$.  
The map $\de_{C;g_1}$ fixes all points in $A\setminus O_1$ while $O_1$ is 
mapped onto $C$.  Hence $\de_{C;g_1}$ maps $A$ onto $(A\cup C)\setminus 
O_1$.  Then $\de_{O_1;f^{-1}g_2}$ maps $(A\cup C)\setminus O_1$ onto 
$(A\cup C)\setminus O_3$.  Further, $f$ maps $(A\cup C)\setminus O_3$ onto 
$(B\cup C)\setminus O_2$.  Finally, $\de_{C;g_2g_1}$ coincides with the 
identity map on $B\setminus O_2$ while mapping $C$ onto $O_2$.  Thus 
$h(A)=B$.

We proceed to verification of condition CI-2.  Consider any disjoint 
nonempty clopen sets $D,E_0,E\subset X$ such that $f(D\cup E_0)=E_0$ for 
some $f\in G$.  Let $E_1=f(D)$ and $E_2=f(E_0)$.  Then $E_1$ and $E_2$ are 
nonempty clopen sets that partition $E_0$.  Property \ref{prop-UC} implies 
that $g(E_0)\subset E$ for some $g\in G$.  Let $U=g(E_0)$, $U_1=g(E_1)$ and 
$U_2=g(E_2)$.  Then $U_1$ and $U_2$ are nonempty clopen sets that partition 
$U$.  Now we define a continuous map $h:X\to X$ that is piecewise an 
element of the group $G$.  The map $h$ coincides with $gf$ on $D$, with 
$f^{-1}$ on $E_0$, with $gfg^{-1}$ on $U$, and with the identity map 
everywhere else.  We have $h(D)=g(E_1)=U_1$, $h(E_0)=D\cup E_0$ and 
$h(U)=gf(E_0)=g(E_2)=U_2$.  Since the sets $h(D)$, $h(E_0)$ and $h(U)$ form 
a partition of $D\cup E_0\cup U$, it follows that the map $h$ is invertible 
and hence a homeomorphism.  By construction, $h$ belongs to the group 
$\F(G)$, which coincides with $G$ as $G$ is ample.  Note that $h(D\cup 
U)=U_1\cup U_2=U$.  Besides, $U\subset E$ and $h$ fixes every point in 
$E\setminus U$.  Therefore $h(D\cup E)=E$.  Thus the set $D$ can be 
absorbed by $E$.
\end{proof}

\begin{lemma}\label{two-absorbable-sets}
If a nonempty clopen set $U\subset X$ is absorbable under the action of a 
group $G\subset\Homeo(X)$, then for any $g\in G$ the set $g(U)$ is 
absorbable as well.  If the group $G$ has Property \ref{prop-CI} then, 
conversely, any two absorbable sets can be mapped onto each other by 
elements of $G$.
\end{lemma}

\begin{proof}
If a clopen set $U$ can be absorbed by a clopen set $E$ under the action of 
a group $G\subset\Homeo(X)$, then for any $g\in G$ the set $g(U)$ can be 
absorbed by $g(E)$.  Indeed, $E$ is disjoint from $U$ and $f(U\cup E)=E$ 
for some $f\in G$.  Then for any $g\in G$ the clopen set $g(E)$ is disjoint 
from $g(U)$, the map $h=gfg^{-1}$ belongs to $G$, and $h\bigl(g(U)\cup 
g(E)\bigr)=hg(U\cup E)=gf(U\cup E)=g(E)$.

Now suppose the group $G$ has Property \ref{prop-CI} and $U_1,U_2\subset X$ 
are two absorbable sets.  We need to show that $h(U_1)=U_2$ for some $h\in 
G$.  First consider the case when both $U_1$ and $U_2$ can be absorbed by 
the same clopen set $E$.  In this case, $f_1(U_1\cup E)=E$ and $f_2(U_2\cup 
E)=E$ for some $f_1,f_2\in G$.  Then $f_2^{-1}f_1\in G$ and 
$f_2^{-1}f_1(U_1\cup E)=U_2\cup E$.  Since $E$ is disjoint from both $U_1$ 
and $U_2$, Property \ref{prop-CI} implies that $h(U_1)=U_2$ for some $h\in 
G$.

If one of the sets $U_1$ and $U_2$ contains the other, it follows from 
Lemma \ref{absorbed-by-any} that the smaller set can be absorbed by any 
clopen set that can absorb the larger one.  Next consider the case when 
neither of the sets $U_1$ and $U_2$ contains the other, and also $U_1\cup 
U_2\ne X$.  Then the clopen set $E=X\setminus(U_1\cup U_2)$ is nonempty and 
the unions $U_1\cup E$ and $U_2\cup E$ are both different from $X$.  
By Lemma \ref{absorbed-by-any}, both $U_1$ and $U_2$ can be absorbed by $E$.

It remains to consider the case when $U_1\cup U_2=X$.  Take any clopen set 
$E$ that can absorb $U_1$.  Then $E$ is disjoint from $U_1$ and $f(U_1\cup 
E)=E$ for some $f\in G$.  The set $U=f(U_1)$ is also absorbable.  Indeed, 
$f\bigl(U\cup(E\setminus U)\bigr)=f(E)=E\setminus U$ so that $U$ can be 
absorbed by $E\setminus U$.  Since $U$ is disjoint from $U_1$ and $U_1\cup 
U_2=X$, the set $U_2$ contains $U$.  By the above, $h(U)=U_2$ for some 
$h\in G$.  Then $hf\in G$ and $hf(U_1)=U_2$.
\end{proof}

\begin{lemma}\label{absorbable-and-CI-implies-UC}
For any ample group that acts minimally and admits an absorbable set, 
Property \ref{prop-CI} implies Property \ref{prop-UC}.
\end{lemma}

\begin{proof}
Suppose $G\in\Homeo(X)$ is an ample group that acts minimally on $X$, 
admits an absorbable set, and has Property \ref{prop-CI}.  We are going to 
show that the group $G$ has the following property, which is a weakened 
form of Property \ref{prop-SC}: for any nonempty clopen set $W\subset X$ 
there exist a clopen set $W_0\subset W$ and a map $f\in G$ such that 
$f(W_0)$ is a proper subset of $W_0$.  After that Property \ref {prop-UC} 
can be established in the same way as it was done in the proof of Lemma 
\ref{UC-equiv-M-and-SC}.

Take any nonempty clopen set $D\subset X$ absorbable under the action of 
the group $G$.  Let $E$ be any clopen set that can absorb $D$.  Then $E$ is 
disjoint from $D$ and $g(D\cup E)=E$ for some $g\in G$.  The set 
$D_1=g(D)$, which is a proper subset of $E$, is absorbable due to Lemma 
\ref{two-absorbable-sets}.  Since the group $G$ has Property \ref{prop-CI}, 
it follows from Lemma \ref{absorbed-by-any} that every nonempty clopen 
subset of $D$ can absorb $D_1$, every nonempty clopen subset of $D_1$ can 
absorb $D$, and every nonempty clopen subset of $X\setminus(D\cup D_1)$ can 
absorb both $D$ and $D_1$.  Any nonempty clopen set $W\subset X$ intersects 
at least one of the sets $D$, $D_1$ and $X\setminus(D\cup D_1)$.  Hence $W$ 
has a clopen subset $W_0$ that can absorb an absorbable set $D_0$, where 
$D_0=D$ or $D_1$.  Then $W_0$ is disjoint from $D_0$ and $f(D_0\cup 
W_0)=W_0$ for some $f\in G$.  Since $W_0$ is a proper subset of $D_0\cup 
W_0$, it follows that $f(W_0)$ is a proper subset of $W_0$.
\end{proof}

Combining Lemmas \ref{UC-implies-CI} and 
\ref{absorbable-and-CI-implies-UC}, we obtain that an ample group has 
Property \ref{prop-UC} if and only if it acts minimally, admits an 
absorbable set (or, equivalently, admits contraction of a clopen set) and 
has Property \ref{prop-CI}.

\begin{question}\label{question}
Let $G$ be an ample group that acts minimally on a Cantor set.
\begin{itemize}
\item
Suppose the group $G$ admits contraction of a clopen set.  Is it possible 
that $G$ does not have Property \ref{prop-UC}?
\item
Suppose the group $G$ does not allow contraction of any clopen set.  Is it 
possible that $G$ does not have Property \ref{prop-NC}?
\end{itemize}
\end{question}

Let $C(X,\bZ)$ denote the set of all continuous functions $\phi:X\to\bZ$.  
We regard $C(X,\bZ)$ as an additive group.  For any clopen set $U\subset X$ 
the characteristic function $\chi_U$, defined by $\chi_U(x)=1$ if $x\in U$ 
and $\chi_U(x)=0$ if $x\in X\setminus U$, belongs to $C(X,\bZ)$.  Any 
function $\phi\in C(X,\bZ)$ takes only finitely many values, and nonempty 
level sets of $\phi$ form a partition of $X$ into clopen sets.  It follows 
that the group $C(X,\bZ)$ is generated by characteristic functions of 
clopen sets.

\begin{definition}\label{def-homology}
Given a group $G\subset\Homeo(X)$, let $K_G$ denote the subgroup of 
$C(X,\bZ)$ generated by all functions of the form $\phi-\phi\circ g$, where 
$\phi\in C(X,\bZ)$ and $g\in G$.  The factor group $C(X,\bZ)/K_G$ is 
denoted $H_0(G)$ and called the zeroth \textbf{homology group} of $G$.
\end{definition}

It is not hard to show that the homology group $H_0(G)$ is preserved when 
the group $G$ is amplified.  In particular, the homology group $H_0(\F(G))$ 
of the ample group $\F(G)$ is the same as $H_0(G)$.

\begin{property}{HS}[homologous sets can be mapped onto each other] 
\label{prop-HS}
Whenever the characteristic functions of two clopen sets $A,B\subset X$ 
different from the empty set and $X$ represent the same homology class in 
$H_0(G)$, we have $g(A)=B$ for some $g\in G$.
\end{property}

If $g(A)=B$ for some clopen sets $A,B\subset X$ and $g\in G$, then 
$\chi_A-\chi_B=\chi_A-\chi_A\circ g^{-1}\in K_G$ so that $\chi_A$ and 
$\chi_B$ represent the same homology class in $H_0(G)$.  Property 
\ref{prop-HS} means that the action of the group $G$ on clopen subsets of 
$X$ has as much transitivity as the homology group $H_0(G)$ allows.

An analogue of Property \ref{prop-HS} was established by Matui for two 
classes of \'{e}tale groupoids, for \emph{almost finite} groupoids in 
\cite{Matui-homology} and for purely infinite groupoids in 
\cite{Matui-pureinf2}.  The authors are grateful to the anonymous referee 
for drawing our attention to these results.  In \cite{Matui-pureinf2} this 
property is called \emph{cancellation} (see Definition 2.11 in 
\cite{Matui-pureinf2}).  Our naming of Property \ref{prop-CI} is a tribute 
to that.

\begin{lemma}\label{HS-implies-CI}
Property \ref{prop-HS} implies Property \ref{prop-CI}.
\end{lemma}

\begin{proof}
If $A$, $B$ and $C$ are clopen subsets of $X$ such that $C$ is disjoint 
from both $A$ and $B$, then $\chi_{A\cup C}-\chi_{B\cup C}=\chi_A-\chi_B$.  
Hence for any group $G\subset\Homeo(X)$ the characteristic functions 
$\chi_{A\cup C}$ and $\chi_{B\cup C}$ represent the same homology class in 
$H_0(G)$ if and only if the characteristic functions $\chi_A$ and $\chi_B$ 
represent the same homology class in $H_0(G)$.  It follows that Property 
\ref{prop-HS} implies Property \ref{prop-CI}.
\end{proof}

It seems plausible that for an ample group, Properties \ref{prop-CI} and 
\ref{prop-HS} are equivalent.  However, the authors are not able to prove 
this yet.

\begin{property}{E}[entanglement]\label{prop-E}
Whenever clopen sets $U_1$ and $U_2$ intersect, the local subgroup 
$\F_{U_1\cup U_2}(G)$ is generated by the union of local subgroups 
$\F_{U_1}(G)\cup\F_{U_2}(G)$.
\end{property}

Along with minimality, Property \ref{prop-E} will be very important for us 
in Section \ref{maxsub} below.   The other properties introduced in this 
section are relevant only because they can help to establish Property 
\ref{prop-E}.  We defer a detailed discussion of Property \ref{prop-E} to 
Section \ref{property-E} below.

\section{Maximal subgroups of ample groups}\label{maxsub}

In this section we study maximal subgroups of an ample group 
$\cG\subset\Homeo(X)$, where the topological space $X$ is compact, 
metrizable, and totally disconnected.  We give examples of such subgroups 
and obtain some partial results on their classification.  Most of the time 
we require the group $\cG$ to act minimally on $X$.  Many results also 
require Property \ref{prop-E} (see Section \ref{prop}).

\begin{definition}
A subgroup $H$ of a group $G$ is called \textbf{maximal} if $H\ne G$ and 
there is no subgroup $K$ such that $H\subset K\subset G$ while $K\ne H$ and 
$K\ne G$. 
\end{definition}

The first class of subgroups to look for maximal subgroups are the 
stabilizers of closed sets (see Definition \ref{def-stab-set}).  Since 
$\St_{\cG}(Y)=\St_{\cG}(X\setminus Y)$ for any set $Y\subset X$, they are 
also the stabilizers of open sets.

\begin{lemma}\label{stab-closed1}
Let $\cG\subset\Homeo(X)$ be an ample group that acts minimally on $X$.  
Suppose $Y$ is a closed subset of $X$.  Then $\St_{\cG}(Y)\ne\cG$ unless 
$Y$ is the empty set or $Y=X$.
\end{lemma}

\begin{proof}
Clearly, the entire group $\cG$ is the stabilizer of $X$ and of the empty 
set.  If $Y$ is not empty, we can find a point $x\in Y$.  Since the group 
$\cG$ acts minimally on $X$, the orbit $\Orb_{\cG}(x)$ is dense in $X$.  If 
$Y\ne X$ then the open set $X\setminus Y$ is not empty.  Hence 
$\Orb_{\cG}(x)$ has a point in $X\setminus Y$.  This means that $g(x)\notin 
Y$ for some $g\in\cG$.  Then $g\notin\St_{\cG}(Y)$ so that 
$\St_{\cG}(Y)\ne\cG$.
\end{proof}

The following lemma and its proof describe a general construction that will 
be used repeatedly throughout this section.

\begin{lemma}\label{transposition-mimic}
Let $\cG\subset\Homeo(X)$ be an ample group.  Suppose $x$ and $y$ are two 
different points in the same orbit of $\cG$ and $Z\subset X$ is a closed 
set that contains neither $x$ nor $y$.  Then there exists a map $f\in\cG$ 
such that $f(x)=y$, $f(y)=x$, and $f(z)=z$ for all $z\in Z$.
\end{lemma}

\begin{proof}
For any point $z\in Z$ we can find a clopen neighborhood $U_z$ that 
contains neither $x$ nor $y$.  Since the set $Z$ is closed and hence 
compact, it can be covered by finitely many clopen sets of the form $U_z$, 
$z\in Z$.  Taking the union of all sets in that finite cover, we obtain a 
clopen set $U$ containing $Z$ and disjoint from $\{x,y\}$.  Next, we can 
find disjoint clopen sets $V_x,V_y\subset X\setminus U$ such that $x\in 
V_x$ and $y\in V_y$.  Since the points $x$ and $y$ are in the same orbit of 
the group $\cG$, we have $y=g(x)$ for some $g\in\cG$.  Then $W=V_x\cap 
g^{-1}(V_y)$ is a clopen neighborhood of $x$.  The set $W$ is disjoint from 
$g(W)$ as $W\subset V_x$ while $g(W)\subset V_y$.  Therefore the 
generalized $2$-cycle $f=\de_{W;g}$ is defined.  Since $g\in\cG$, the map 
$f$ belongs to the group $\F(\cG)=\cG$.  By construction, $f(x)=g(x)=y$ and 
$f(y)=g^{-1}(y)=x$.  All points in the complement of the set $W\cup g(W)$ 
are fixed by $f$.  Since $W\cup g(W)\subset V_x\cup V_y\subset X\setminus 
U$ and $Z\subset U$, it follows that $f(z)=z$ for all $z\in Z$.
\end{proof}

\begin{lemma}\label{stab-closed2}
Let $\cG\subset\Homeo(X)$ be an ample group that acts minimally on $X$.  
Suppose $Y_1$ and $Y_2$ are distinct closed subsets of $X$.  Then 
$\St_{\cG}(Y_1)=\St_{\cG}(Y_2)$ if and only if $Y_1\cup Y_2=X$ and $Y_1\cap 
Y_2=\partial Y_1\cap\partial Y_2$.  Equivalent conditions are that 
$Y_1=\overline{X\setminus Y_2}$ and $Y_2=\overline{X\setminus Y_1}$.
\end{lemma}

\begin{proof}
Recall that $\St_{\cG}(X\setminus E)=\St_{\cG}(E)$ for any set $E\subset 
X$.  If $f:X\to X$ is a homeomorphism, then 
$f\bigl(\overline{E}\bigr)=\overline{f(E)}$.  Hence 
$f\bigl(\overline{E}\bigr)=\overline{E}$ whenever $f(E)=E$.  As a 
consequence, $\St_{\cG}(E)\subset\St_{\cG}\bigl(\overline{E}\bigr)$.

Assume that sets $Y_1,Y_2\subset X$ satisfy conditions 
$Y_1=\overline{X\setminus Y_2}$ and $Y_2=\overline{X\setminus Y_1}$.  By 
the above, $\St_{\cG}(Y_2)=\St_{\cG}(X\setminus Y_2)\subset 
\St_{\cG}\bigl(\overline{X\setminus Y_2}\bigr)=\St_{\cG}(Y_1)$.  Likewise,
$\St_{\cG}(Y_1)\subset\St_{\cG}(Y_2)$.  Hence 
$\St_{\cG}(Y_1)=\St_{\cG}(Y_2)$.  Further, either of the two assumed 
conditions implies that $Y_1\cup Y_2=X$.  The first condition also implies 
that no interior point of $Y_2$ belongs to $Y_1$.  Likewise, the 
second of the two implies that no interior point of $Y_1$ belongs to 
$Y_2$.  Therefore any common point of $Y_1$ and $Y_2$ is their common 
boundary point: $Y_1\cap Y_2\subset\partial Y_1\cap\partial Y_2$.  Since 
the sets $Y_1$ and $Y_2$ are clearly closed, they contain their own 
boundaries so that $Y_1\cap Y_2=\partial Y_1\cap\partial Y_2$.  Conversely, 
assume $Y_1,Y_2\subset X$ are closed sets such that $Y_1\cup Y_2=X$ and 
$Y_1\cap Y_2=\partial Y_1\cap\partial Y_2$.  Then $X$ is a disjoint union 
of three sets $Y_1\setminus Y_2$, $Y_2\setminus Y_1$ and $Y_1\cap Y_2$.  
Moreover, the sets $Y_1\setminus Y_2=X\setminus Y_2$ and $Y_2\setminus Y_1= 
X\setminus Y_1$ are open.  It follows that $Y_1\setminus Y_2$ and 
$Y_2\setminus Y_1$ are disjoint from both $\partial Y_1$ and $\partial 
Y_2$.  Hence $\partial Y_1\cup\partial Y_2\subset Y_1\cap Y_2$, which 
implies that $\partial Y_1=\partial Y_2=Y_1\cap Y_2$.  Since $\partial 
Y_1=\partial(X\setminus Y_1)=\partial(Y_2\setminus Y_1)$, we obtain that 
$\overline{X\setminus Y_1}=\overline{Y_2\setminus Y_1}=(Y_2\setminus 
Y_1)\cup\partial(Y_2\setminus Y_1)=(Y_2\setminus Y_1)\cup(Y_1\cap 
Y_2)=Y_2$.  Likewise, $\overline{X\setminus Y_2}=Y_1$.

It remains to show that $\St_{\cG}(Y_1)\ne\St_{\cG}(Y_2)$ for any distinct 
closed sets $Y_1,Y_2\subset X$ that do not satisfy at least one of the 
conditions $Y_1\cup Y_2=X$ and $Y_1\cap Y_2=\partial Y_1\cap\partial Y_2$.  
First we consider the case when $Y_1\cup Y_2\ne X$.  Then $U=X\setminus 
(Y_1\cup Y_2)$ is a nonempty open set.  Since $Y_1\ne Y_2$, there is a 
point $x\in X$ that belongs to one of these sets but not to the other.  We 
may assume without loss of generality that $x\in Y_1$ and $x\notin Y_2$.  
Since the group $\cG$ acts minimally on $X$, the orbit $\Orb_{\cG}(x)$ is 
dense in $X$.  In particular, it has a point $y$ in $U$.  Clearly, $y\notin 
Y_1$ and $y\notin Y_2$.  By Lemma \ref{transposition-mimic}, there exists a 
map $f\in\cG$ such that $f(x)=y$ and $f(z)=z$ for all $z\in Y_2$.  Then 
$f\in\Stp_{\cG}(Y_2)\subset\St_{\cG}(Y_2)$.  On the other hand, 
$f\notin\St_{\cG}(Y_1)$ as $x\in Y_1$ while $y\notin Y_1$.  Therefore 
$\St_{\cG}(Y_1)\ne\St_{\cG}(Y_2)$.

Now consider the case when $Y_1\cup Y_2=X$ but there exists a point $x\in 
Y_1\cap Y_2$ that is not a boundary point for at least one of the sets 
$Y_1$ and $Y_2$.  We may assume without loss of generality that $x$ is an 
interior point of $Y_2$.  Furthermore, we may assume that $Y_1\ne X$ as 
otherwise $\St_{\cG}(Y_2)\ne\cG=\St_{\cG}(Y_1)$ due to Lemma 
\ref{stab-closed1}.  Then $X\setminus Y_1$ is a nonempty open set.  Just as 
above, minimality of the action of the group $\cG$ on $X$ implies that the 
orbit $\Orb_{\cG}(x)$ has a point $y$ in $X\setminus Y_1$.  Note that 
$X\setminus Y_1\subset Y_2$ since $Y_1\cup Y_2=X$.  Therefore every point 
of $X\setminus Y_1$ (including $y$) is an interior point of $Y_2$.  Hence 
the closed set $\overline{X\setminus Y_2}$ contains neither $x$ nor $y$.  
By Lemma \ref{transposition-mimic}, there exists a map $f\in\cG$ such that 
$f(x)=y$ and $f(z)=z$ for all $z\in\overline{X\setminus Y_2}$.  Then 
$f\in\St_{\cG}(X\setminus Y_2)=\St_{\cG}(Y_2)$.  On the other hand, 
$f\notin\St_{\cG}(Y_1)$ as $x\in Y_1$ while $y\notin Y_1$.  Thus 
$\St_{\cG}(Y_1)\ne\St_{\cG}(Y_2)$.
\end{proof}

\begin{lemma}\label{stab-closed3}
Let $\cG\subset\Homeo(X)$ be an ample group that acts minimally on $X$.  
Suppose $Y_1$ and $Y_2$ are nonempty closed subsets of $X$ such that 
$Y_1\subset Y_2$ and $Y_1\ne Y_2$.  Then $\St_{\cG}(Y_1)\ne\St_{\cG}(Y_2)$.
\end{lemma}

\begin{proof}
If $Y_2=X$ then $\St_{\cG}(Y_1)\ne\cG=\St_{\cG}(Y_2)$ due to Lemma 
\ref{stab-closed1}.  Otherwise $Y_1\cup Y_2=Y_2\ne X$.  Then 
$\St_{\cG}(Y_1)\ne\St_{\cG}(Y_2)$ due to Lemma \ref{stab-closed2}.
\end{proof}

Now we are ready to establish our first result on the classification of 
maximal subgroups of ample groups.

\begin{theorem}\label{maxsub-nonminimal}
Let $\cG\subset\Homeo(X)$ be an ample group that acts minimally on $X$.  
Suppose $H$ is a maximal subgroup of $\cG$ that does not act minimally on 
$X$.  Then $H=\St_{\cG}(Y)$ for some closed set $Y\subset X$ different from 
the empty set and $X$.  Furthermore, if the stabilizer $\St_{\cG}(Y)$ of a 
closed set $Y\subset X$ is indeed a maximal subgroup of $\cG$, then the 
induced action of $\St_{\cG}(Y)$ on $Y$ is minimal.
\end{theorem}

\begin{proof}
Since the group $H$ does not act minimally on $X$, it admits a closed 
invariant set $Y$ different from the empty set and $X$.  We have $h(Y)=Y$ 
for any $h\in H$.  Hence $H\subset\St_{\cG}(Y)$.  Lemma \ref{stab-closed1} 
implies that $\St_{\cG}(Y)\ne\cG$.  As $H$ is a maximal subgroup of $\cG$, 
it follows that $H=\St_{\cG}(Y)$.  

Now assume that a group $H=\St_{\cG}(Y)$, where $Y$ is a closed subset of 
$X$, is indeed a maximal subgroup of $\cG$.  To prove that the induced 
action of $H$ on $Y$ is minimal, we need to show that any closed set 
$Y_0\subset Y$ invariant under it coincides with the empty set or $Y$.  We 
have $h(Y_0)=Y_0$ for any $h\in H$.  Hence $H\subset\St_{\cG}(Y_0)$.  Note 
that $Y\ne X$ since $H\ne\cG$.  Then $Y_0\ne X$ as well.  Assume that 
$Y_0\ne\emptyset$.  Just as above, Lemma \ref{stab-closed1} implies that 
$\St_{\cG}(Y_0)\ne\cG$ and then maximality of $H$ implies that 
$H=\St_{\cG}(Y_0)$.  Since $\St_{\cG}(Y_0)=\St_{\cG}(Y)$ and $Y_0\subset 
Y$, it follows from Lemma \ref{stab-closed3} that $Y_0=Y$.
\end{proof}

The assumption that the stabilizer $\St_{\cG}(Y)$ of a closed set $Y$ acts 
minimally when restricted to $Y$ has implications for the set $Y$ and for 
the group $\St_{\cG}(Y)$.

\begin{lemma}\label{minim-nowhere-dense}
Let $G$ be a subgroup of $\Homeo(X)$.  Suppose $Y$ is a closed subset of 
$X$ such that the stabilizer $\St_G(Y)$ acts minimally when restricted to 
$Y$.  Then the set $Y$ is either clopen or nowhere dense in $X$.
\end{lemma}

\begin{proof}
We have $h(Y)=Y$ for any $h\in\St_G(Y)$.  Since $h$ is a homeomorphism of 
the ambient space $X$, it maps interior points of the set $Y$ to its 
interior points and boundary points of $Y$ to its boundary points.  
Therefore the boundary $\partial Y$, which is a closed subset of $Y$, is 
invariant under the induced action of the group $\St_G(Y)$ on $Y$.  As the 
latter action is minimal, the boundary $\partial Y$ has to coincide with 
the empty set or $Y$.  If $\partial Y=\emptyset$ then the set $Y$ is 
clopen.  If $\partial Y=Y$ then $Y$ is nowhere dense in $X$.
\end{proof}

\begin{lemma}\label{minim-clopen}
Let $\cG\subset\Homeo(X)$ be an ample group.  Suppose $U$ is a nonempty 
clopen subset of $X$.  Then any orbit of the induced action of the 
stabilizer $\St_{\cG}(U)$ on $U$ is the intersection of an orbit of $\cG$ 
with $U$.  As a consequence, $\St_{\cG}(U)$ acts minimally on $U$ whenever 
$\cG$ acts minimally on $X$.
\end{lemma}

\begin{proof}
The induced action of the group $H=\St_{\cG}(U)$ on $U$ is the restriction 
of its action on $X$.  Therefore every orbit of the induced action is of 
the form $\Orb_H(x)$, where $x\in U$.  We need to show that 
$\Orb_H(x)=\Orb_{\cG}(x)\cap U$.  The inclusion $\Orb_H(x)\subset 
\Orb_{\cG}(x)\cap U$ is obvious.  Conversely, take any point 
$y\in\Orb_{\cG}(x)\cap U$ different from $x$.  By Lemma 
\ref{transposition-mimic}, there exists a map $f\in\cG$ such that $f(x)=y$ 
and $f$ coincides with the identity map on the clopen set $X\setminus U$.  
Then $f\in\RiSt_{\cG}(U)\subset H$.  Thus $y\in\Orb_H(x)$.

Assuming the group $\cG$ acts minimally on $X$, every orbit of $\cG$ is 
dense in $X$.  Since $U$ is a clopen set, it follows that the intersection 
of any orbit of $\cG$ with $U$ is dense in $U$.  By the above this means 
that every orbit of the induced action of the group $H$ on $U$ is dense in 
$U$.  Hence the latter action is minimal.
\end{proof}

\begin{lemma}\label{same-orbits}
Let $\cG\subset\Homeo(X)$ be an ample group.  Suppose $Y$ is a closed 
subset of $X$ such that the stabilizer $\St_{\cG}(Y)$ acts minimally when 
restricted to $Y$.  Then for any $g\in\cG\setminus\St_{\cG}(Y)$ the group 
$\langle\St_{\cG}(Y)\cup\{g\}\rangle$ generated by $\St_{\cG}(Y)$ and $g$ 
has the same orbits as $\cG$.
\end{lemma}

\begin{proof}
Let $H=\St_{\cG}(Y)$.  For any $g\in\cG\setminus\St_{\cG}(Y)$, let  
$H_g=\langle\St_{\cG}(Y)\cup\{g\}\rangle$.  First let us show that no orbit 
of the group $H_g$ is contained in $Y$.  Since $g$ does not map the set $Y$ 
onto itself, there exists a point $x_0\in Y$ such that $g(x_0)\notin Y$ or 
$g^{-1}(x_0)\notin Y$.  In either case, we have an element $g_0\in H_g$ 
such that $g_0(x_0)\notin Y$.  Note that $X\setminus Y$ is an open 
neighborhood of the point $g_0(x_0)$.  Hence the set $U=g_0^{-1}(X\setminus 
Y)$ is an open neighborhood of $x_0$.  Take any $x\in Y$.  Since the group 
$H$ acts minimally when restricted to $Y$, the orbit $\Orb_H(x)$ is dense 
in $Y$.  In particular, it has a point in $U\cap Y$.  That is, $h(x)\in U$ 
for some $h\in H$.  Then $g_0h\in H_g$ and $g_0h(x)\in g_0(U)=X\setminus 
Y$.  Therefore the orbit $\Orb_{H_g}(x)$ has a point outside $Y$.

Since $H_g$ is a subgroup of $\cG$, for any $x\in X$ the orbit 
$\Orb_{H_g}(x)$ is contained in $\Orb_{\cG}(x)$.  We need to show that, 
conversely, $\Orb_{\cG}(x)\subset\Orb_{H_g}(x)$.  Take any point $y\in 
\Orb_{\cG}(x)$.  By the above the orbits of the points $x$ and $y$ under 
the action of the group $H_g$ are not contained in $Y$.  Hence there exist 
$h_1,h_2\in H_g$ such that the points $x_1=h_1(x)$ and $y_1=h_2(y)$ are 
outside the set $Y$.  If $x_1=y_1$ then $y=h_2^{-1}h_1(x)$ so that 
$y\in\Orb_{H_g}(x)$.  Now assume $x_1\ne y_1$.  Both $x_1$ and $y_1$ do 
belong to $\Orb_{\cG}(x)$.  By Lemma \ref{transposition-mimic}, there 
exists a map $f\in\cG$ such that $f(x_1)=y_1$ and $f(z)=z$ for all $z\in 
Y$.  Then $h_2^{-1}fh_1(x)=y$.  Note that $f\in\Stp_{\cG}(Y)\subset H$, 
which implies that $h_2^{-1}fh_1\in H_g$.  Thus $y\in\Orb_{H_g}(x)$.
\end{proof}

The simplest example of a closed subset of $X$ is a finite set.  To treat 
the stabilizers of finite sets, we do not even need to assume that the 
ample group acts minimally on $X$.

\begin{lemma}\label{perm-mimic}
Let $\cG\subset\Homeo(X)$ be an ample group.  Suppose $Y\subset X$ is a 
finite set contained in a single orbit of $\cG$ and $Z\subset X$ is a 
closed (e.g., finite) set disjoint from $Y$.  Then for any permutation 
$\pi:Y\to Y$ there exists a map $f\in\cG$ such that $f(x)=\pi(x)$ for all 
$x\in Y$ and $f(x)=x$ for all $x\in Z$.
\end{lemma}

\begin{proof}
First we consider a particular case when the permutation $\pi$ is a 
transposition.  That is, $\pi=(x\,y)$, where $x$ and $y$ are distinct 
elements of $Y$.  Let $Z_+=(Y\cup Z)\setminus\{x,y\}$.  Then $Z_+$ is a 
closed set containing neither $x$ nor $y$.  By Lemma 
\ref{transposition-mimic}, there exists a map $f\in\cG$ such that $f(x)=y$, 
$f(y)=x$, and $f(z)=z$ for all $z\in Z_+$.  The map $f$ coincides with 
$\pi$ on the set $Y$ and with the identity map on the set $Z$.

Now consider the general case.  If $\pi$ is the identity map, we can take 
$f=\id_X$, which belongs to $\cG$.  Otherwise the permutation $\pi$ can be 
decomposed as a product, $\pi=\pi_1\pi_2\dots\pi_k$, where each $\pi_i$ is 
a transposition exchanging two elements of $Y$.  By the above for any $i$, 
$1\le i\le k$ there exists a map $f_i\in\cG$ that coincides with $\pi_i$ on 
the set $Y$ and fixes all points of the set $Z$.  Then $f=f_1f_2\dots f_k$ 
is an element of $\cG$ that coincides with the permutation 
$\pi_1\pi_2\dots\pi_k=\pi$ on $Y$ and fixes all points of $Z$. 
\end{proof}

The next theorem is our main result on the stabilizers of finite sets.

\begin{theorem}\label{maxsub-finite}
Let $\cG\subset\Homeo(X)$ be an ample group.  Suppose $Y$ is a finite 
nonempty subset of $X$.  Then the following statements hold true.
\begin{itemize}
\item[(i)]
If $Y$ is a proper subset of a single orbit of $\cG$ and, moreover, $Y$ 
does not contain exactly half of points in the orbit, then 
the stabilizer $\St_{\cG}(Y)$ is a maximal subgroup of $\cG$.
\item[(ii)]
If $Y$ is a subset of an orbit $\Orb_{\cG}(x)$ that contains exactly half 
of its points, then $\St_{\cG}(Y)$ is a subgroup of index $2$ in the group 
$\St_{\cG}(Y,\Orb_{\cG}(x)\setminus Y)$.  The latter is a maximal subgroup 
of $\cG$ unless $Y$ consists of a single point, in which case 
$\St_{\cG}(Y,\Orb_{\cG}(x)\setminus Y)=\cG$.
\item[(iii)]
If there are at least two orbits of $\cG$ that intersect $Y$ but are not 
contained in $Y$, then $\St_{\cG}(Y)$ is a proper subgroup of $\cG$ that is 
not maximal.
\item[(iv)]
If $Y$ is a union of orbits of $\cG$, then $\St_{\cG}(Y)=\cG$.
\item[(v)]
If $Y=Y_1\sqcup Y_2$, where $Y_2$ is a union of orbits of $\cG$, then 
$\St_{\cG}(Y)=\St_{\cG}(Y_1)$.
\end{itemize}
\end{theorem}

\begin{proof}
Since $g(h(x))=(gh)(x)$ for all maps $g,h\in\cG$ and points $x\in X$, it 
follows that $g(\Orb_{\cG}(x))\subset\Orb_{\cG}(x)$.  For any $g\in\cG$,
the inverse map $g^{-1}$ is also in $\cG$.  Hence we also have 
$g^{-1}(\Orb_{\cG}(x))\subset\Orb_{\cG}(x)$, which implies that, in fact, 
$g(\Orb_{\cG}(x))=\Orb_{\cG}(x)$.  It further follows that $g(Y)=Y$ for any
set $Y\subset X$ that is a union of orbits of the group $\cG$.  
Consequently, $\St_{\cG}(Y)=\cG$ for any such set $Y$, finite or not.  
Further, if a set $Y\subset X$ is represented as a disjoint union of two 
sets, $Y=Y_1\sqcup Y_2$, then $g(Y)=g(Y_1)\sqcup g(Y_2)$ for any
$g\in\cG$.  Assuming $Y_2$ is a union of orbits of $\cG$, we have 
$g(Y_2)=Y_2$.  Hence $g(Y)=Y$ if and only if $g(Y_1)=Y_1$.  As a result, 
$\St_{\cG}(Y)=\St_{\cG}(Y_1)$.

We have proved statements (iv) and (v).  Next we prove statement (iii).  
Suppose $\Orb_{\cG}(x)$ and $\Orb_{\cG}(y)$ are two distinct orbits of
$\cG$ that intersect a finite set $Y\subset X$ but are not contained in 
$Y$.  Then there exist points $x_1,x_2\in\Orb_{\cG}(x)$ and 
$y_1,y_2\in\Orb_{\cG}(y)$ such that $x_1$ and $y_1$ belong to $Y$ while 
$x_2$ and $y_2$ do not.  By Lemma \ref{perm-mimic}, the group $\cG$ 
contains a map $f$ that interchanges $y_1$ and $y_2$ while fixing all
points of $Y$ different from $y_1$.  Note that 
$f(Y\cap\Orb_{\cG}(x))=Y\cap\Orb_{\cG}(x)$ while 
$f(Y\cap\Orb_{\cG}(y))\ne Y\cap\Orb_{\cG}(y)$.  Likewise, there exists 
$h\in\cG$ that interchanges $x_1$ and $x_2$ while fixing all points of $Y$ 
different from $x_1$.  Then $h(Y\cap\Orb_{\cG}(y))=Y\cap\Orb_{\cG}(y)$ 
while $h(Y\cap\Orb_{\cG}(x))\ne Y\cap\Orb_{\cG}(x)$.  It follows that 
neither of the stabilizers $\St_{\cG}(Y\cap\Orb_{\cG}(x))$ and
$\St_{\cG}(Y\cap\Orb_{\cG}(y))$ contains the other.  Therefore both 
stabilizers are proper subgroups of $\cG$ and their intersection is a 
proper subgroup of both of them.  It remains to observe that the said 
intersection contains $\St_{\cG}(Y)$.  Indeed, any $g\in\cG$ maps each 
orbit of $\cG$ onto itself.  Hence $g(Y\cap\Orb_{\cG}(x))= 
g(Y)\cap\Orb_{\cG}(x)$ and $g(Y\cap\Orb_{\cG}(y))=g(Y)\cap\Orb_{\cG}(y)$.  
Consequently, the equality $g(Y)=Y$ implies that 
$g(Y\cap\Orb_{\cG}(x))=Y\cap\Orb_{\cG}(x)$ and 
$g(Y\cap\Orb_{\cG}(y))=Y\cap\Orb_{\cG}(y)$.

We proceed to the main part of the theorem, namely, statements (i) and 
(ii).  Suppose that a finite nonempty set $Y$ is a proper subset of a 
single orbit $\Orb_{\cG}(x)$.  Since $g(\Orb_{\cG}(x))=\Orb_{\cG}(x)$ for 
any $g\in\cG$, we have $g(Y)=Y$ if and only if $g(\Orb_{\cG}(x)\setminus 
Y)=\Orb_{\cG}(x)\setminus Y$.  It follows that $\St_{\cG}(Y)$ coincides 
with another set stabilizer $\St_{\cG}(\Orb_{\cG}(x)\setminus Y)$ as well
as with the individual stabilizer $\Sti_{\cG}(Y,\Orb_{\cG}(x)\setminus
Y)$.  The collective stabilizer $\St_{\cG}(Y,\Orb_{\cG}(x)\setminus Y)$
acts on the set $\{Y,\Orb_{\cG}(x)\setminus Y\}$ and 
$\Sti_{\cG}(Y,\Orb_{\cG}(x)\setminus Y)$ is the kernel of this action.  If 
the set $Y$ does not contain exactly half of points in $\Orb_{\cG}(x)$, it 
cannot be mapped onto $\Orb_{\cG}(x)\setminus Y$ by any invertible 
transformation of $X$.  Then the action is trivial so that 
$\St_{\cG}(Y,\Orb_{\cG}(x)\setminus Y)=
\Sti_{\cG}(Y,\Orb_{\cG}(x)\setminus Y)=\St_{\cG}(Y)$.  In the case when $Y$ 
contains exactly half of points in $\Orb_{\cG}(x)$, the orbit is finite and 
the set $\Orb_{\cG}(x)\setminus Y$ contains the same number of points as 
$Y$.  Then there exists a permutation $\pi$ on $\Orb_{\cG}(x)$ such that 
$\pi(Y)=\Orb_{\cG}(x)\setminus Y$ and $\pi(\Orb_{\cG}(x)\setminus Y)=Y$.  
Lemma \ref{perm-mimic} implies that there is also a homeomorphism 
$f\in\cG$ such that $f(Y)=\Orb_{\cG}(x)\setminus Y$ and 
$f(\Orb_{\cG}(x)\setminus Y)=Y$.  Hence the action of the group 
$\St_{\cG}(Y,\Orb_{\cG}(x)\setminus Y)$ on the set 
$\{Y,\Orb_{\cG}(x)\setminus Y\}$ is nontrivial, which implies that 
$\St_{\cG}(Y)=\Sti_{\cG}(Y,\Orb_{\cG}(x)\setminus Y)$ is a subgroup of
index $2$ in $\St_{\cG}(Y,\Orb_{\cG}(x)\setminus Y)$.

Both $Y$ and $\Orb_{\cG}(x)\setminus Y$ are nonempty sets.  If each of them
consists of a single point, then $\St_{\cG}(Y,\Orb_{\cG}(x)\setminus Y)= 
\St_{\cG}(\Orb_{\cG}(x))=\cG$.  Otherwise we can choose two distinct points 
$x_1$ and $x_2$ in one of the two sets and another point $x_3$ in the other
set.  By Lemma \ref{perm-mimic}, the ample group $\cG$ contains a map $f$ 
such that $f(x_1)=x_1$, $f(x_2)=x_3$ and $f(x_3)=x_2$.  Then one of the 
points $f(x_1)$ and $f(x_2)$ is in $Y$ while the other one is in 
$\Orb_{\cG}(x)\setminus Y$.  It follows that $f\notin 
\St_{\cG}(Y,\Orb_{\cG}(x)\setminus Y)$.  Hence 
$\St_{\cG}(Y,\Orb_{\cG}(x)\setminus Y)\ne\cG$.

It remains to prove that the stabilizer $\St_{\cG}(Y,\Orb_{\cG}(x)\setminus
Y)$ is a maximal subgroup of $\cG$ provided that at least one of the sets 
$Y$ and $\Orb_{\cG}(x)\setminus Y$ consists of more than one point.  We 
have already shown that under this condition, 
$\St_{\cG}(Y,\Orb_{\cG}(x)\setminus Y)$ is a proper subgroup of $\cG$.  For
convenience, let us assume that the number of points in $Y$ is less than or
equal to the number of points in $\Orb_{\cG}(x)\setminus Y$ (note that the 
latter can be infinite).  There is no loss of generality here as otherwise 
we could replace $Y$ by $\Orb_{\cG}(x)\setminus Y$.  Given a map $g\in\cG$ 
that does not belong to $\St_{\cG}(Y,\Orb_{\cG}(x)\setminus Y)$, let $H_g$ 
denote the subgroup of $\cG$ generated by 
$\St_{\cG}(Y,\Orb_{\cG}(x)\setminus Y)$ and $g$.  We need to show that 
$H_g=\cG$.  This requires some preparation.

Since $g(\Orb_{\cG}(x))=\Orb_{\cG}(x)$, the condition 
$g\notin\St_{\cG}(Y,\Orb_{\cG}(x)\setminus Y)$ is equivalent to a pair of 
conditions $g(\Orb_{\cG}(x)\setminus Y)\ne Y$ and
$g(\Orb_{\cG}(x)\setminus Y)\ne\Orb_{\cG}(x)\setminus Y$.  Note that the 
image $g(\Orb_{\cG}(x)\setminus Y)$ cannot be a proper subset of $Y$ as any 
proper subset of $Y$ has strictly fewer points than $\Orb_{\cG}(x)\setminus 
Y$.  Besides, $g(\Orb_{\cG}(x)\setminus Y)$ cannot be a proper subset
of $\Orb_{\cG}(x)\setminus Y$ as that would imply that $Y$ is a proper 
subset of $g(Y)$, which would contradict the fact that the set $Y$ is 
finite.  We conclude that $g(\Orb_{\cG}(x)\setminus Y)$ intersects both $Y$ 
and $\Orb_{\cG}(x)\setminus Y$.  Hence we can choose points $x_+,x_-\in 
\Orb_{\cG}(x)\setminus Y$ such that $g(x_+)\in Y$ while $g(x_-)\in 
\Orb_{\cG}(x)\setminus Y$.  By Lemma \ref{perm-mimic}, there exists a map 
$h\in\cG$ such that the restriction of $h$ to the set $Z=Y\cup g^{-1}(Y) 
\cup\{x_+,x_-\}$ coincides with the transposition $(x_+\,\,x_-)$.  Then the 
restriction of the map $\tilde g=ghg^{-1}$ to the set $g(Z)=Y\cup g(Y)\cup 
\{g(x_+),g(x_-)\}$ coincides with the transposition 
$\bigl(g(x_+)\,g(x_-)\bigr)$.  Since the map $h$ fixes all points in $Y$, 
we have $h\in\St_{\cG}(Y)\subset\St_{\cG}(Y,\Orb_{\cG}(x)\setminus Y)$.  It 
follows that $\tilde g\in H_g$.  Let $y_+=g(x_+)$ and $y_-=g(x_-)$.  By 
construction, $y_+\in Y$, $y_-\in \Orb_{\cG}(x)\setminus Y$, $\tilde 
g(y_+)=y_-$, $\tilde g(y_-)=y_+$, and $\tilde g(y)=y$ for all points $y\in 
Y$ different from $y_+$.

Next we prove that for any finite set $S\subset \Orb_{\cG}(x)\setminus Y$ 
containing the point $y_-$, there exists a map $h_S\in H_g$ such that the 
restriction of $h_S$ to $Y\cup S$ coincides with the transposition 
$(y_+\,\,y_-)$.  The proof is by induction on the number $n$ of points in 
$S$.  If $n=1$ then $S=\{y_-\}$ and we can let $h_S=\tilde g$.  In the case
$n\ge2$, we assume the claim holds for all sets of $n-1$ points.  Take any 
point $z_0\in S$ different from $y_-$.  Then the set
$S_0=S\setminus\{z_0\}$ consists of $n-1$ points including $y_-$.  By the 
inductive assumption, there exists a map $h_{S_0}\in H_g$ such that the 
restriction of $h_{S_0}$ to the set $Y\cup S_0$ coincides with the 
transposition $(y_+\,\,y_-)$.  If the point $\tilde{z}_0=h_{S_0}(z_0)$ 
coincides with $z_0$, then we can clearly let $h_S=h_{S_0}$.  Otherwise 
Lemma \ref{perm-mimic} implies existence of a map $\phi\in\cG$ such that 
its restriction to the set $Y\cup S\cup\{\tilde{z}_0\}$ coincides with the 
transposition $(z_0\,\tilde{z}_0)$.  Since $\phi$ fixes all points in $Y$, 
it belongs to the stabilizer $\St_{\cG}(Y)$, which is contained in $H_g$.  
Then the map $\phi h_{S_0}$ is in $H_g$ as well.  By construction, $\phi 
h_{S_0}(z_0)=z_0$.  Note that $\tilde{z}_0\notin Y\cup S_0$ as $Y\cup S_0= 
h_{S_0}(Y\cup S_0)$.  Hence $\phi h_{S_0}$ coincides with $h_{S_0}$ on 
$Y\cup S_0$.  Therefore we can let $h_S=\phi h_{S_0}$ and complete the 
inductive step.

Next we prove that for any pair of sets $S'\subset Y$ and $S''\subset 
\Orb_{\cG}(x)\setminus Y$ of the same cardinality, there exists a map 
$h_{S',S''}\in H_g$ such that $h_{S',S''}(S')=S''$, $h_{S',S''}^2(y)=y$ 
for all $y\in S'$, and $h_{S',S''}(y)=y$ for all $y\in Y\setminus S'$.  The 
proof is by induction on the number $n$ of points in $S'$.  If $n=0$ then 
both sets are empty and we can let $h_{S',S''}=\id_X$.  In the case 
$n\ge1$, we assume the claim holds whenever the sets consist of $n-1$ 
points.  Take any points $y_0\in S'$ and $z_0\in S''$.  Then the sets 
$S'_0=S'\setminus\{y_0\}$ and $S''_0=S''\setminus\{z_0\}$ contain $n-1$ 
points each.  By the inductive assumption, there exists a map 
$h_{S'_0,S''_0}\in H_g$ such that $h_{S'_0,S''_0}(S'_0)=S''_0$, 
$h_{S'_0,S''_0}^2(y)=y$ for all $y\in S'_0$, and $h_{S'_0,S''_0}(y)=y$ for 
all $y\in Y\setminus S'_0$.  Note that $h_{S'_0,S''_0}(S''_0)=S'_0$ and 
hence $h_{S'_0,S''_0}(Y\cup S''_0)=Y\cup S''_0$.  As a consequence, the 
point $\tilde{z}_0=h_{S'_0,S''_0}(z_0)$ is in $\Orb_{\cG}(x)\setminus Y$ 
but not in $S''_0$.  Consider a set $Z=S''_0\cup\{y_-\}\cup\{z_0\}\cup
\{\tilde{z}_0\}$, which is a finite subset of $\Orb_{\cG}(x)\setminus Y$.  
By the above there exists a map $h_Z\in H_g$ such that its restriction to 
$Y\cup Z$ coincides with the transposition $(y_+\,\,y_-)$.  Next, by Lemma 
\ref{perm-mimic}, for any permutation on the set $Y\cup Z$ there exists an 
element of the group $\cG$ that coincides with this permutation on $Y\cup 
Z$.  In particular, there exists a map $\phi\in\cG$ such that 
$\phi(z_0)=\tilde{z}_0$, $\phi(\tilde{z}_0)=z_0$, and $\phi(y)=y$ for all 
other points $y\in Y\cup Z$.  The restriction of $\phi$ to $Y\cup Z$ 
coincides with the transposition $(z_0\,\tilde{z}_0)$ if 
$z_0\ne\tilde{z}_0$, and with the identity map if $z_0=\tilde{z}_0$.  
Furthermore, there exists a map $\psi\in\cG$ such that $\psi(y_+)=y_0$, 
$\psi(y_0)=y_+$, $\psi(y_-)=z_0$, $\psi(z_0)=y_-$, and $\psi(y)=y$ for all 
other points $y\in Y\cup Z$.  The restriction of $\psi$ to the set $Y\cup 
Z$ coincides with the permutation $(y_+\,y_0)(y_-\,z_0)$ if $y_+\ne y_0$ 
and $y_-\ne z_0$, with the transposition $(y_+\,y_0)$ if $y_+\ne y_0$ and 
$y_-=z_0$, with the transposition $(y_-\,z_0)$ if $y_+=y_0$ and $y_-\ne 
z_0$, and with the identity map if $y_+=y_0$ and $y_-=z_0$.  It follows 
that the restriction of the map $\psi h_Z\psi^{-1}$ to $Y\cup Z$ coincides 
with the transposition $\bigl(\psi(y_-)\,\psi(y_+)\bigr)=(y_0\,z_0)$.  
Finally, let $h_{S',S''}=\psi h_Z\psi^{-1}\phi h_{S'_0,S''_0}$.  By 
construction, $\phi(Y)=\psi(Y)=Y$ so that $\phi,\psi\in\St_{\cG}(Y)\subset 
H_g$.  Then $h_{S',S''}\in H_g$ as well.  The map $h_{S',S''}$ coincides 
with $h_{S'_0,S''_0}$ on the set $(Y\setminus\{y_0\})\cup S''_0$.  Besides, 
$h_{S',S''}(y_0)=\psi h_Z\psi^{-1}\phi(y_0)=\psi h_Z\psi^{-1}(y_0)=z_0$ and
$h_{S',S''}(z_0)=\psi h_Z\psi^{-1}\phi(\tilde{z}_0)=\psi 
h_Z\psi^{-1}(z_0)=y_0$.  The inductive step is complete.

Now we are ready to show that $H_g=\cG$.  Take any map $f\in\cG$.  We 
define three sets $S_{++}=Y\cap f^{-1}(Y)$, $S_{+-}=Y\cap 
f^{-1}(\Orb_{\cG}(x)\setminus Y)$ and $S_{-+}=(\Orb_{\cG}(x)\setminus 
Y)\cap f^{-1}(Y)$.  Since $f(\Orb_{\cG}(x))=\Orb_{\cG}(x)$, it follows that 
$Y=S_{++}\sqcup S_{+-}=f(S_{++})\sqcup f(S_{-+})$.  Since $f$ is a 
one-to-one map, the set $f(S_{++})$ is of the same cardinality as $S_{++}$ 
while the set $f(S_{-+})$ is of the same cardinality as $S_{-+}$.  Since 
the set $Y$ is finite, it follows that $S_{-+}$ is of the same cardinality 
as $S_{+-}$.  By the above there exists a map $h_f\in H_g$ such that 
$h_f(S_{+-})=S_{-+}$ and $h_f(y)=y$ for all $y\in Y\setminus 
S_{+-}=S_{++}$.  We obtain that $fh_f(S_{+-})=f(h_f(S_{+-}))=f(S_{-+})$ and 
$fh_f(S_{++})=f(h_f(S_{++}))=f(S_{++})$, which implies that $fh_f(Y)=Y$.  
Hence the map $fh_f$ belongs to the stabilizer $\St_{\cG}(Y)$, which is 
contained in $H_g$.  Then $f=(fh_f)h_f^{-1}$ is in $H_g$ as well.  Thus 
$H_g=\cG$.
\end{proof}

\begin{remark}\label{maxsub-finite-remark}
Theorem \ref{maxsub-finite} allows to determine for any ample group 
$\cG\subset\Homeo(X)$ and any finite set $Y\subset X$ whether the 
stabilizer $\St_{\cG}(Y)$ is a maximal subgroup of $\cG$ and whether 
$\St_{\cG}(Y)=\cG$.  Namely, $\St_{\cG}(Y)=\cG$ if and only if $Y$ is a 
union of orbits of $\cG$ (since the set $Y$ is finite, this means the union 
of a finite number of finite orbits).  The stabilizer $\St_{\cG}(Y)$ is a 
maximal subgroup of $\cG$ if and only if $Y=Y_1\sqcup Y_2$, where $Y_2$ is
a union of orbits of $\cG$ (it may be empty) and $Y_1$ is a proper nonempty
subset of a single orbit such that either $Y$ does not contain exactly half 
of points in the orbit or else $Y$ is a one-point subset of a two-point 
orbit.
\end{remark}

The following concise corollary of Theorem \ref{maxsub-finite} is much less
general but it will be enough, e.g., when the topological space $X$ is a 
Cantor set and the ample group $\cG$ acts minimally on $X$.

\begin{theorem}\label{maxsub-finite2}
Let $\cG\subset\Homeo(X)$ be an ample group that has no finite orbits.  
Suppose $Y$ is a finite nonempty subset of $X$.  Then the stabilizer 
$\St_{\cG}(Y)$ is a maximal subgroup of $\cG$ if and only if $Y$ is 
contained in a single orbit of $\cG$.
\end{theorem}

\begin{proof}
Since all orbits of the group $\cG$ are infinite, the finite set $Y$ cannot 
contain an entire orbit of $\cG$ or exactly half of elements in an orbit.  
If $Y$ intersects at least two different orbits of $\cG$, then the 
stabilizer $\St_{\cG}(Y)$ is not a maximal subgroup of $\cG$ due to 
statement (iii) of Theorem \ref{maxsub-finite}.  Otherwise $Y$ is 
contained in a single orbit of $\cG$.  Then $\St_{\cG}(Y)$ is a maximal 
subgroup of the group $\cG$ due to statement (i) of Theorem 
\ref{maxsub-finite}.
\end{proof}

To treat the stabilizers of infinite closed sets, we need to develop a 
different approach.  This approach will also apply to another class of 
subgroups that contains many maximal subgroups, namely, the stabilizers of 
partitions of $X$ into clopen sets (see Definition \ref{def-stab-coll}).

\begin{lemma}\label{direct-product-of-local-subgroups}
Suppose $X=U_1\sqcup U_2\sqcup\dots\sqcup U_k$ is a partition of $X$ into 
clopen sets.  Then for any ample group $\cG\subset\Homeo(X)$ the individual 
stabilizer $\Sti_{\cG}(U_1,U_2,\dots,U_k)$ is the internal direct product 
of the local subgroups $\cG_{U_i}$, $1\le i\le k$.
\end{lemma}

\begin{proof}
The group $H=\Sti_{\cG}(U_1,U_2,\dots,U_k)$ is the intersection of the set 
stabilizers $\St_{\cG}(U_i)$, $1\le i\le k$.  For any $i$ the local 
subgroup $\cG_{U_i}=\RiSt_{\cG}(U_i)$ is a subgroup of $\St_{\cG}(U_i)$.
Besides, $\cG_{U_i}\subset\Stp_{\cG}(U_j)\subset\St_{\cG}(U_j)$ whenever 
$i\ne j$ as then $U_i\cap U_j=\emptyset$.  Hence $\cG_{U_i}\subset H$ for 
all $i$.

For each $i$, $1\le i\le k$ choose a map $f_i\in\cG_{U_i}$.  Then the maps 
$f_1,f_2,\dots,f_k$ commute with one another as they have disjoint 
supports.  Besides, the product $f_1f_2\ldots f_k$ coincides with $f_i$ on 
$U_i$ for all $i$, which implies that $f_1f_2\ldots f_k=\id_X$ if and only 
if each $f_i$ is the identity map.  It follows that the local subgroups 
$\cG_{U_i}$, $1\le i\le k$ form an internal direct product in $H$.

Now take any $f\in H$.  For each $i$, $1\le i\le k$, consider a map 
$f_i:X\to X$ that coincides with $f$ on $U_i$ and with the identity map on 
$X\setminus U_i$.  Since $f$ is invertible and $f(U_i)=U_i$, it follows 
that the map $f_i$ is invertible as well.  By construction, $f_i$ is 
piecewise an element of the ample group $\cG$ and $\supp(f_i)\subset U_i$.  
Hence $f_i\in\cG_{U_i}$.  Also by construction, $f=f_1f_2\ldots f_k$.  We 
conclude that the product of the local subgroups 
$\cG_{U_1},\cG_{U_2},\dots,\cG_{U_k}$ equals $H$.
\end{proof}

\begin{lemma}\label{local-subgroups-in-stabilizers}
Let $X=U_1\sqcup U_2\sqcup\dots\sqcup U_k$ be a partition of $X$ into 
clopen sets at most one of which consists of a single point.  Let 
$\cG\subset\Homeo(X)$ be an ample group that acts minimally on $X$.  
Suppose $H$ is a subgroup of $\cG$ such that 
$\Sti_{\cG}(U_1,U_2,\dots,U_k)\subset H\subset 
\St_{\cG}(U_1,U_2,\dots,U_k)$.  Then $H$ contains a local subgroup $\cG_U$ 
if and only if $U\subset U_i$ for some $i$, $1\le i\le k$.
\end{lemma}

\begin{proof}
Lemma \ref{direct-product-of-local-subgroups} implies that the individual 
stabilizer $\Sti_{\cG}(U_1,U_2,\dots,U_k)$ contains local subgroups 
$\cG_{U_1},\cG_{U_2},\dots,\cG_{U_k}$.  If a clopen set $U$ is contained in 
$U_i$ for some $i$, $1\le i\le k$ then, clearly, $\cG_U\subset\cG_{U_i}$.  
Consequently, $\cG_U\subset H$.  Now assume that a clopen set $U$ is not 
contained in a single element of the partition $X=U_1\sqcup 
U_2\sqcup\dots\sqcup U_k$.  Then there are two different elements $V_1$ and 
$V_2$ of the partition such that the intersections $U\cap V_1$ and $U\cap 
V_2$ are not empty.  By the assumption of the lemma, at least one of the 
sets $V_1$ and $V_2$ contains more than one point.  We may assume without 
loss of generality that $V_1$ contains more than one point.  Take any point 
$x\in U\cap V_1$ and let $z$ be any point of $V_1$ different from $x$.  
Since the group $\cG$ acts minimally on $X$, the orbit $\Orb_{\cG}(x)$ is 
dense in $X$.  In particular, this orbit has a point $y$ in the nonempty 
clopen set $U\cap V_2$.  By Lemma \ref{transposition-mimic}, there exists a 
map $f\in\cG$ such that $f(x)=y$ and $f$ coincides with the identity map on 
the closed set $(X\setminus U)\cup\{z\}$.  Then $f\in\cG_U$.  Since 
$f(x)=y$ and $f(z)=z$, the image $f(V_1)$ intersects both $V_1$ and $V_2$ 
so that it cannot be an element of the partition $X=U_1\sqcup 
U_2\sqcup\dots\sqcup U_k$.  As a consequence, $f$ does not belong to the 
collective stabilizer $\St_{\cG}(U_1,U_2,\dots,U_k)$.  In particular, 
$f\notin H$.  Thus the local subgroup $\cG_U$ is not contained in $H$.
\end{proof}

Notice that the assumption of Lemma \ref{local-subgroups-in-stabilizers} 
that at most one element of the partition of $X$ consists of a single point 
is essential.  A one-point set $\{x\}$ is clopen if $x$ is an isolated 
point of $X$.  If $\{x\}$ and $\{y\}$ are two distinct elements of the 
partition, then the local subgroup $\cG_{\{x,y\}}$ is contained in the 
stabilizer of the partition.

\begin{lemma}\label{partition-stab}
Let $\cG\subset\Homeo(X)$ be an ample group that acts minimally on $X$.  
Suppose that $X=U_1\sqcup U_2\sqcup\dots\sqcup U_k$ is a partition of $X$ 
into nonempty clopen sets.  Then $\St_{\cG}(U_1,U_2,\dots,U_k)\ne\cG$ 
unless this is the partition into points or the trivial partition 
consisting of only one set.
\end{lemma}

\begin{proof}
The partition of $X$ into points and the trivial partition are preserved by 
any invertible transformation of $X$.  In particular, they are preserved by 
all elements of the group $\cG$.  Conversely, assume that 
$\St_{\cG}(U_1,U_2,\dots,U_k)=\cG$.  Then the group $\cG$ acts on the set
$\{U_1,U_2,\dots,U_k\}$.  That action is transitive since the action of 
$\cG$ on $X$ is minimal.  As a consequence, all elements of the partition 
are of the same cardinality.  Hence this is either the partition into 
points, or else a partition in which every element consists of more than 
one point.  In the latter case, Lemma \ref{local-subgroups-in-stabilizers} 
implies that any local subgroup $\cG_U$ is contained in the stabilizer 
$\St_{\cG}(U_1,U_2,\dots,U_k)$ if and only if the clopen set $U$ is 
contained in a single element of the partition.  Since $\cG$ can be 
represented as the local subgroup $\cG_X$, it follows that $X$ is an 
element of the partition.  As all elements of the partition are nonempty 
sets, we obtain that $X$ is the only element so that the partition is 
trivial.
\end{proof}

\begin{proposition}\label{maxsub-partition}
Let $\cG\subset\Homeo(X)$ be an ample group with Property \ref{prop-E}.  
Suppose $H$ is a subgroup of $\cG$ that acts minimally on $X$ and contains 
a local subgroup $\cG_U$ for some clopen set $U$ consisting of more than 
one point.  Then there exists a partition of $X$ into clopen sets, 
$X=U_1\sqcup U_2\sqcup\dots\sqcup U_k$, such that each $U_i$ consists of 
more than one point and $\Sti_{\cG}(U_1,U_2,\dots,U_k)\subset 
H\subset\St_{\cG}(U_1,U_2,\dots,U_k)$.  Moreover, the partition is unique 
and the induced action of $H$ on the set $\{U_1,U_2,\dots,U_k\}$ is 
transitive. 
\end{proposition}

\begin{proof}
Let $\cL_H$ be the set of all clopen sets $V\subset X$ such that the 
subgroup $H$ contains the local subgroup $\cG_V$.  Note that 
$\cG_{V_1}\subset\cG_{V_2}$ whenever $V_1\subset V_2$.  Hence any clopen 
set contained in an element of $\cL_H$ is itself an element of $\cL_H$.  
Since the group $\cG$ has Property \ref{prop-E}, for any intersecting 
clopen sets $V_1$ and $V_2$ the local subgroup $\cG_{V_1\cup V_2}$ is 
generated by the union of $\cG_{V_1}$ and $\cG_{V_2}$.  Therefore $V_1\cup 
V_2\in\cL_H$ whenever $V_1,V_2\in\cL_H$ and $V_1\cap V_2\ne\emptyset$.  
Furthermore, $\cG_{g(V)}=g\cG_Vg^{-1}$ for any clopen set $V$ and any 
$g\in\cG$.  This follows from the fact that 
$\supp(gfg^{-1})=g\bigl(\supp(f)\bigr)$ for any homeomorphism 
$f\in\Homeo(X)$.  As a consequence, $h(V)\in\cL_H$ for all $h\in H$ 
whenever $V\in\cL_H$.

By assumption, some clopen set $U$ consisting of more than one point is an 
element of $\cL_H$.  Minimality of the action of the group $H$ on $X$ 
implies that any orbit $\Orb_H(x)$ of $H$ is dense in $X$.  In particular, 
it has a point in $U$.  That is, $h(x)\in U$ for some $h\in H$.  Then 
$h^{-1}(U)$ is an element of $\cL_H$ containing the point $x$.  We conclude 
that elements of $\cL_H$ cover $X$.  Since the topological space $X$ is 
compact, there is a finite subcover.  Let $\{U_1,U_2,\dots,U_k\}$ be a 
finite subcover with the least possible number of elements.  This choice 
implies that all elements of the subcover are nonempty sets and the union 
of any two or more of them is not an element of $\cL_H$.  Then it follows 
from the above that any two elements of the subcover are disjoint.  Hence 
our subcover is actually a partition: $X=U_1\sqcup U_2\sqcup\dots\sqcup 
U_k$.  Observe that this is not the partition into points as otherwise $U$ 
would be the union of two or more elements of the partition.  We claim that 
any set $V\in\cL_H$ is contained in a single element of the partition.  
Assume the contrary, that is, $V$ intersects $U_i$ and $U_j$, where $i\ne 
j$.  Then the unions $V\cup U_i$ and $V\cup U_j$ are elements of $\cL_H$.  
These unions are not disjoint as they both contain $V$.  Therefore their 
union $V\cup U_i\cup U_j$ belongs to $\cL_H$ as well.  Since $U_i\cup U_j$ 
is a clopen subset of $V\cup U_i\cup U_j$, we obtain that $U_i\cup 
U_j\in\cL_H$, which yields a contradiction.

Take any $h\in H$.  By the above the sets $h(U_1),h(U_2),\dots,h(U_k)$ 
belong to $\cL_H$.  Therefore each of them is contained in a single element 
of the partition $\{U_1,U_2,\dots,U_k\}$.  Since $h$ is an invertible map, 
the sets $h(U_1),h(U_2),\dots,h(U_k)$ themselves form a partition of $X$.
Hence we have a finite partition of $X$ into nonempty sets each element of 
which is contained in a single element of another partition of $X$ into 
nonempty sets.  Since both partitions consist of the same number of sets, 
this is possible only if both partitions are the same.  In other words, $h$ 
maps the sets $U_1,U_2,\dots,U_k$ onto one another.  We conclude that the 
subgroup $H$ is contained in $\St_{\cG}(U_1,U_2,\dots,U_k)$, the stabilizer 
of the partition.  Besides, $H$ contains the individual stabilizer 
$\Sti_{\cG}(U_1,U_2,\dots,U_k)$.  Indeed, it follows from Lemma 
\ref{direct-product-of-local-subgroups} that 
$\Sti_{\cG}(U_1,U_2,\dots,U_k)$ is generated by the local subgroups 
$\cG_{U_i}$, $1\le i\le k$, which are contained in $H$.

Since $H\subset\St_{\cG}(U_1,U_2,\dots,U_k)$, the group $H$ acts on the set
$\{U_1,U_2,\dots,U_k\}$.  Take any orbit $\mathcal{O}$ of this action.  
Then the union of all sets in $\mathcal{O}$ is invariant under the action 
of $H$ on $X$.  Note that this union is a nonempty clopen set.  Minimality 
of the latter action implies that the union coincides with $X$.  Hence 
$\mathcal{O}=\{U_1,U_2,\dots,U_k\}$ so that the action of $H$ on 
$\{U_1,U_2,\dots,U_k\}$ is transitive.  In other words, all elements of the 
partition can be mapped onto one another by elements of $H$.  As a 
consequence, all elements of the partition are of the same cardinality.  
Since $\{U_1,U_2,\dots,U_k\}$ is not the partition into points, it follows 
that each $U_i$ consists of more than one point.

Let $\{V_1,V_2,\dots,V_\ell\}$ be an arbitrary partition of $X$ into clopen 
sets such that each $V_j$ consists of more than one point and 
$\Sti_{\cG}(V_1,V_2,\dots,V_\ell)\subset 
H\subset\St_{\cG}(V_1,V_2,\dots,V_\ell)$.  Lemma 
\ref{direct-product-of-local-subgroups} implies that the sets 
$V_1,V_2,\dots,V_\ell$ belong to $\cL_H$, just as the sets 
$U_1,U_2,\dots,U_k$ do.  Note that the group $\cG$ acts minimally on $X$ 
since its subgroup $H$ acts minimally.  Hence we can apply Lemma 
\ref{local-subgroups-in-stabilizers}.  It follows from the lemma that each 
element of either of the partitions $\{U_1,U_2,\dots,U_k\}$ and 
$\{V_1,V_2,\dots,V_\ell\}$ is contained in a single element of the other 
partition.  Since both partitions consist of nonempty sets, this is 
possible only if both partitions are the same.
\end{proof}

Let us discuss the formulation of Proposition \ref{maxsub-partition}, 
specifically, why all clopen sets involved should consist of more than one 
point.  If $x$ is an isolated point of the topological space $X$, then 
$\{x\}$ is a nonempty clopen set but the associated local subgroup 
$\cG_{\{x\}}$ is trivial.  Hence the condition that a clopen set $U$ 
contains more than one point is necessary for the local subgroup $\cG_U$ to 
be nontrivial.  In the case when the ample group $\cG$ acts minimally on 
$X$, one can show that this condition is also sufficient.  Further, if $X$ 
is a finite set and $X=U_1\sqcup U_2\sqcup\dots\sqcup U_k$ is the partition 
into points, then $\Sti_{\cG}(U_1,U_2,\dots,U_k)$ is a trivial group while 
$\St_{\cG}(U_1,U_2,\dots,U_k)=\cG$.  The requirement that each $U_i$ 
consists of more than one point rules out this degenerate case.  Without it 
we would not be able to show uniqueness of the partition.  Besides, the 
requirement allows to use Proposition \ref{maxsub-partition} in conjunction 
with Lemma \ref{local-subgroups-in-stabilizers}.

Proposition \ref{maxsub-partition} is going to be crucial for the proofs of 
all subsequent results on maximal subgroups.  The downside of this is that 
all those results will require the ample group to have Property 
\ref{prop-E}.  We shall discuss this property and provide examples of ample
groups with it in Section \ref{property-E} below.

Our next result concerns the stabilizers of closed sets that are not open.

\begin{theorem}\label{maxsub-not-clopen}
Let $\cG\subset\Homeo(X)$ be an ample group that acts minimally on $X$ and 
has Property \ref{prop-E}.  Suppose $Y\subset X$ is a closed set that is 
not open.  Then the stabilizer $\St_{\cG}(Y)$ is a maximal subgroup of 
$\cG$ if and only if it acts minimally when restricted to $Y$, in which 
case $Y$ is necessarily nowhere dense in $X$.
\end{theorem}

\begin{proof}
If the stabilizer $\St_{\cG}(Y)$ is a maximal subgroup of $\cG$, then the 
induced action of $\St_{\cG}(Y)$ on $Y$ is minimal due to Theorem 
\ref{maxsub-nonminimal}.  Since the set $Y$ is not clopen, Lemma 
\ref{minim-nowhere-dense} implies that it is nowhere dense in $X$.

Now assume that $\St_{\cG}(Y)$ acts minimally when restricted to $Y$.  Just 
as above, this implies that $Y$ is nowhere dense in $X$.  Also, since the 
closed set $Y$ is not open, it is different from the empty set and $X$.  
Then $\St_{\cG}(Y)\ne\cG$ due to Lemma \ref{stab-closed1}.  To prove that 
the stabilizer $\St_{\cG}(Y)$ is a maximal subgroup of $\cG$, we are going 
to show that for any $g\in\cG\setminus\St_{\cG}(Y)$ the group 
$H_g=\langle\St_{\cG}(Y)\cup\{g\}\rangle$ generated by $\St_{\cG}(Y)$ and 
$g$ coincides with $\cG$.  By Lemma \ref{same-orbits}, the group $H_g$ has 
the same orbits as $\cG$.  As a consequence, $H_g$ acts minimally on $X$.  
To apply Proposition \ref{maxsub-partition} to the group $H_g$, we need it 
to contain a local subgroup $\cG_U$ for some clopen set $U$ containing more 
than one point.  If $U\subset X\setminus Y$ then any map in $\cG_U$ fixes 
all points of $Y$ so that $\cG_U\subset\Stp_{\cG}(Y)\subset\St_{\cG}(Y) 
\subset H_g$.  Note that the set $X\setminus Y$ is open but not closed.  
Therefore it is infinite.  Any point $x\in X\setminus Y$ has a clopen 
neighborhood $V_x\subset X\setminus Y$.  Take any distinct points $x,y\in 
X\setminus Y$.  Then $V_x\cup V_y$ is a clopen set consisting of more than 
one point.  Besides, $V_x\cup V_y\subset X\setminus Y$ so that 
$\cG_{V_x\cup V_y}\subset H_g$.  Now Proposition \ref{maxsub-partition} 
implies that $\Sti_{\cG}(U_1,U_2,\dots,U_k)\subset 
H_g\subset\St_{\cG}(U_1,U_2,\dots,U_k)$, where $\cP=\{U_1,U_2,\dots,U_k\}$ 
is a partition of $X$ into clopen sets, each consisting of more than one 
point.  Since the set $Y$ is nowhere dense in $X$, its complement 
intersects any nonempty open set.  In particular, for any $i$, $1\le i\le 
k$ we can find a point $x_i\in U_i$ that does not belong to $Y$.  Let 
$V=V_{x_1}\cup V_{x_2}\cup\dots\cup V_{x_k}$.  Then $V$ is a clopen subset 
of $X\setminus Y$, which implies that $\cG_V\subset H_g$.  Lemma 
\ref{local-subgroups-in-stabilizers} further implies that $V$ is contained 
in a single element of the partition $\cP$.  However, $V$ was constructed 
so as to intersect every element of $\cP$.  We conclude that $\cP$ consists 
of a single element, $\cP=\{X\}$.  Since $\Sti_{\cG}(X)=\St_{\cG}(X)=\cG$, 
we obtain that $H_g=\cG$.
\end{proof}

If a closed set $Y\subset X$ is clopen then for any ample group
$\cG\subset\Homeo(X)$ that acts minimally on $X$, the stabilizer 
$\St_{\cG}(Y)$ acts minimally when restricted to $Y$ (due to 
Lemma \ref{minim-clopen}).  If $Y$ is finite then $\St_{\cG}(Y)$ acts 
minimally on it if and only if $Y$ is contained in a single orbit of
$\cG$.  If the closed set $Y$ is neither clopen nor finite, it is not clear
whether minimality of the action of $\St_{\cG}(Y)$ on $Y$ is at all 
possible.  In other words, it is not clear whether Theorem 
\ref{maxsub-not-clopen} yields any new examples of maximal subgroups 
compared with Theorem \ref{maxsub-finite}.   We address this question in 
Section \ref{nowhere-dense} below.  In view of Lemma
\ref{finite-or-Cantor}, the assumptions of Theorem \ref{maxsub-not-clopen} 
can be satisfied only in the case when $X$ is a Cantor set.  As it turns 
out, there are plenty of new examples in that case (see Proposition 
\ref{plenty-of-examples} below).

The next three theorems contain our main results on the stabilizers of 
partitions of $X$ into clopen sets.  Theorems \ref{maxsub-clopen} and 
\ref{maxsub-3-or-more} tell when such stabilizers are maximal subgroups of 
the corresponding ample groups (Theorem \ref{maxsub-clopen} for partitions 
into two sets, Theorem \ref{maxsub-3-or-more} for partitions into three or 
more sets).  Theorem \ref{maxsub-2-or-more} provides a characterization of 
those maximal subgroups of ample groups that are the stabilizers of 
partitions into clopen sets (just as Theorem \ref{maxsub-nonminimal} 
provides a characterization of those maximal subgroups that are the 
stabilizers of closed sets).  Theorem \ref{maxsub-clopen} also treats the 
stabilizers of clopen sets.

\begin{theorem}\label{maxsub-clopen}
Let $\cG\subset\Homeo(X)$ be an ample group that acts minimally on $X$ and 
has Property \ref{prop-E}.  Suppose $U$ is a clopen set different from the 
empty set and $X$.  Then $\St_{\cG}(U,X\setminus U)$ is a maximal subgroup 
of $\cG$ unless both $U$ and $X\setminus U$ consist of a single point, in 
which case $\St_{\cG}(U,X\setminus U)=\cG$.  If $U$ cannot be mapped onto 
$X\setminus U$ by an element of $\cG$ then 
$\St_{\cG}(U)=\St_{\cG}(U,X\setminus U)$; otherwise $\St_{\cG}(U)$ is a 
subgroup of index $2$ in $\St_{\cG}(U,X\setminus U)$.
\end{theorem}

\begin{proof}
For any invertible map $f:X\to X$ we have $f(U)=U$ if and only if 
$f(X\setminus U)=X\setminus U$.  It follows that $\St_{\cG}(U)$ coincides 
with another set stabilizer $\St_{\cG}(X\setminus U)$ as well as with the 
individual stabilizer $\Sti_{\cG}(U,X\setminus U)$.  The stabilizer 
$\St_{\cG}(U,X\setminus U)$ of the partition $X=U\sqcup(X\setminus U)$ acts 
on the set $\{U,X\setminus U\}$ and $\Sti_{\cG}(U,X\setminus U)$ is the 
kernel of this action.  If $U$ cannot be mapped onto $X\setminus U$ by an 
element of $\cG$, then the action is trivial so that 
$\St_{\cG}(U,X\setminus U)=\Sti_{\cG}(U,X\setminus U)=\St_{\cG}(U)$.  
Otherwise $f(U)=X\setminus U$ for some $f\in\cG$.  Then also $f(X\setminus 
U)=X\setminus f(U)=U$ so that $f\in\St_{\cG}(U,X\setminus U)$.  Hence the 
action of the group $\St_{\cG}(U,X\setminus U)$ on $\{U,X\setminus U\}$ is 
nontrivial, which implies that $\St_{\cG}(U)=\Sti_{\cG}(U,X\setminus U)$ is 
a subgroup of index $2$ in $\St_{\cG}(U,X\setminus U)$.

If both $U$ and $X\setminus U$ consist of a single point, then 
$X=U\sqcup(X\setminus U)$ is the partition into points, which implies that 
$\St_{\cG}(U,X\setminus U)=\cG$.  Otherwise $\St_{\cG}(U,X\setminus U)\ne 
\cG$ due to Lemma \ref{partition-stab}.  To prove that 
$\St_{\cG}(U,X\setminus U)$ is a maximal subgroup of $\cG$ in the latter 
case, we are going to show that for any 
$g\in\cG\setminus\St_{\cG}(U,X\setminus U)$ the group 
$H_g=\langle\St_{\cG}(U,X\setminus U)\cup\{g\}\rangle$ generated by 
$\St_{\cG}(U,X\setminus U)$ and $g$ coincides with $\cG$.  Since the group 
$\cG$ acts minimally on $X$, it follows from Lemma \ref{minim-clopen} that 
the stabilizer $\St_{\cG}(U)$ acts minimally when restricted to $U$.  Then 
Lemma \ref{same-orbits} implies that the group 
$\langle\St_{\cG}(U)\cup\{g\}\rangle$ generated by $\St_{\cG}(U)$ and $g$ 
has the same orbits as $\cG$.  Since the group $H_g$ contains 
$\langle\St_{\cG}(U)\cup\{g\}\rangle$, it also has the same orbits as 
$\cG$.  As a consequence, $H_g$ acts minimally on $X$.  Further, Lemma 
\ref{direct-product-of-local-subgroups} implies that the stabilizer 
$\St_{\cG}(U,X\setminus U)$ contains the local subgroups $\cG_U$ and 
$\cG_{X\setminus U}$.  Hence $H_g$ contains them as well.  Recall that we 
consider the case when at least one of the sets $U$ and $X\setminus U$ 
consists of more than one point.  Now Proposition \ref{maxsub-partition} 
implies that $\Sti_{\cG}(U_1,U_2,\dots,U_k)\subset 
H_g\subset\St_{\cG}(U_1,U_2,\dots,U_k)$, where $\cP=\{U_1,U_2,\dots,U_k\}$ 
is a partition of $X$ into clopen sets, each consisting of more than one 
point.  By Lemma \ref{local-subgroups-in-stabilizers}, any local subgroup 
$\cG_V$ is a subgroup of $H_g$ if and only if $V\subset U_i$ for some $i$, 
$1\le i\le k$.  In particular, each of the sets $U$ and $X\setminus U$ is 
contained in a single element of the partition $\cP$.  This leaves only two 
possibilities, $\cP=\{U,X\setminus U\}$ or $\cP=\{X\}$.  Since $H_g$ is not 
contained in $\St_{\cG}(U,X\setminus U)$, we conclude that $\cP=\{X\}$.  As 
$\Sti_{\cG}(X)=\St_{\cG}(X)=\cG$, we obtain that $H_g=\cG$.
\end{proof}

\begin{theorem}\label{maxsub-3-or-more}
Let $\cG\subset\Homeo(X)$ be an ample group that acts minimally on $X$ and 
has Property \ref{prop-E}.  Suppose $X=U_1\sqcup U_2\sqcup\dots\sqcup U_k$ 
is a partition of $X$ into at least three clopen sets, each consisting of 
more than one point.  Then $\St_{\cG}(U_1,U_2,\dots,U_k)$, the stabilizer 
of the partition, is a maximal subgroup of $\cG$ if and only if its induced 
action on the set $\{U_1,U_2,\dots,U_k\}$ is transitive.
\end{theorem}

\begin{proof}
Let $H=\St_{\cG}(U_1,U_2,\dots,U_k)$.  Since the partition 
$\cP=\{U_1,U_2,\dots,U_k\}$ is neither trivial nor the partition into 
points, Lemma \ref{partition-stab} implies that $H\ne\cG$.  First we 
consider the case when the induced action of the group $H$ on $\cP$ is not 
transitive.  Take any orbit $\mathcal{O}$ of this action and let $U$ be the 
union of all sets in $\mathcal{O}$.  Then $U$ is a clopen set invariant 
under the action of $H$ on $X$.  It follows that $H\subset\St_{\cG}(U)$.  
Note that the orbit $\mathcal{O}$ is different from the empty set and 
$\cP$.  Therefore $U$ is different from the empty set and $X$.  Then 
$\St_{\cG}(U)\ne\cG$ due to Lemma \ref{stab-closed1}.  Recall that the 
pointwise stabilizer $\Stp_{\cG}(U)$ and the rigid stabilizer 
$\RiSt_{\cG}(U)$ are subgroups of $\St_{\cG}(U)$.  It is easy to see that 
$\Stp_{\cG}(U)=\cG_{X\setminus U}$ and $\RiSt_{\cG}(U)=\cG_U$.  On the 
other hand, the local subgroups $\cG_U$ and $\cG_{X\setminus U}$ cannot 
both be subgroups of $H$.  Indeed, otherwise it would follow from Lemma 
\ref{local-subgroups-in-stabilizers} that each of the sets $U$ and 
$X\setminus U$ is contained in a single element of the partition $\cP$, 
which is not possible as $k\ge3$.  We conclude that $H\ne\St_{\cG}(U)$.  
Hence $H$ is not a maximal subgroup of $\cG$.

Now consider the case when the induced action of $H$ on the partition $\cP$ 
is transitive.  Let us show that in this case the group $H$ has the same 
orbits as $\cG$.  Any orbit of $H$ has a point in every element of $\cP$.  
Given $x\in X$, for any $i$, $1\le i\le k$ choose a point 
$y_i\in\Orb_H(x)\cap U_i$.  By Lemma \ref{minim-clopen}, the intersection 
of the orbit $\Orb_{\cG}(y_i)=\Orb_{\cG}(x)$ with $U_i$ coincides with the 
orbit of $y_i$ under the action of the stabilizer $\St_{\cG}(U_i)$.  Note 
that $\St_{\cG}(U_i)=\St_{\cG}(X\setminus U_i)=\Sti_{\cG}(U_i,X\setminus 
U_i)$.  By Lemma \ref{direct-product-of-local-subgroups}, the group 
$\Sti_{\cG}(U_i,X\setminus U_i)$ is the internal direct product of local 
subgroups $\cG_{U_i}$ and $\cG_{X\setminus U_i}$.  Since $\cG_{X\setminus 
U_i}$ acts trivially on $U_i$, it follows that the orbit of the point $y_i$ 
under the action of $\St_{\cG}(U_i)$ coincides with its orbit under the 
action of $\cG_{U_i}$.  As each $\cG_{U_i}$ is obviously contained in $H$, 
we obtain that $\Orb_{\cG}(x)\cap U_i\subset\Orb_H(y_i)=\Orb_H(x)$ for all 
$i$, which implies that $\Orb_H(x)=\Orb_{\cG}(x)$.

To prove that $H$ is a maximal subgroup of $\cG$, we need to show that for 
any $g\in\cG\setminus H$ the group $H_g=\langle H\cup\{g\}\rangle$ 
generated by $H$ and $g$ coincides with $\cG$.  Since the group $H$ has the 
same orbits as $\cG$, the same is true for $H_g$.  As a consequence, $H_g$ 
acts minimally on $X$.  Since the group $H$ contains the local subgroups 
$\cG_{U_i}$, $1\le i\le k$, so does $H_g$.  Now Proposition 
\ref{maxsub-partition} implies that 
$\Sti_{\cG}(V_1,V_2,\dots,V_\ell)\subset 
H_g\subset\St_{\cG}(V_1,V_2,\dots,V_\ell)$, where 
$\cP_g=\{V_1,V_2,\dots,V_\ell\}$ is a partition of $X$ into clopen sets, 
each consisting of more than one point.  By Lemma 
\ref{local-subgroups-in-stabilizers}, any local subgroup $\cG_U$ is 
contained in $H_g$ if and only if $U\subset V_j$ for some $j$, $1\le 
j\le\ell$.  In particular, each element of the partition $\cP$ is contained 
in a single element of the partition $\cP_g$.  It follows that each element 
of $\cP_g$ is the union of one or more elements of $\cP$.  Note that 
$\cP_g\ne\cP$ since the group $H_g$ is not contained in $H$.  Hence there 
exists $j$, $1\le j\le\ell$ such that the set $V_j$ contains sets $U_{i_1}$ 
and $U_{i_2}$, where $i_1\ne i_2$.  Take any $i$, $1\le i\le k$, different 
from $i_1$ and $i_2$ (as $k\ge3$, at least one such $i$ exists).  Since the 
group $H$ acts transitively on $\cP$, we have $h(U_{i_2})=U_i$ for some 
$h\in H$.  Note that $U_{i_2}$ and $U_i$ are disjoint clopen sets.  Hence 
the generalized $2$-cycle $f=\de_{U_{i_2};h}$ is defined.  Since $h\in\cG$, 
it follows that $f\in\F(\cG)=\cG$.  We have $f(U_{i_2})=U_i$, 
$f(U_i)=U_{i_2}$, and $f(x)=x$ for any $x$ not in $U_{i_2}$ or $U_i$.  
Therefore $f$ preserves the partition $\cP$ so that $f\in H$.  Then $f\in 
H_g$ so that $f$ preserves the partition $\cP_g$ as well.  In particular, 
$f(V_j)\in\cP_g$.  Since $f(U_{i_1})=U_{i_1}$ and $f(U_{i_2})=U_i$, it 
follows that $f(V_j)=V_j$ and $U_i\subset V_j$.  As $i$ was chosen 
arbitrarily, we conclude that $V_j=X$.  Thus the partition $\cP_g$ is 
trivial, $\cP_g=\{X\}$.  Since $\Sti_{\cG}(X)=\St_{\cG}(X)=\cG$, we obtain 
that $H_g=\cG$.
\end{proof}

\begin{theorem}\label{maxsub-2-or-more}
Let $\cG\subset\Homeo(X)$ be an ample group with Property \ref{prop-E}.  
Suppose $H$ is a maximal subgroup of $\cG$ that acts minimally on $X$ and 
contains a local subgroup $\cG_U$ for some clopen set $U$ containing more 
than one point.  Then $H=\St_{\cG}(U_1,U_2,\dots,U_k)$, where $X=U_1\sqcup 
U_2\sqcup\dots\sqcup U_k$ is a partition of $X$ into clopen sets, each 
containing more than one point.  Moreover, the partition is unique, it 
consists of at least two sets, and the induced action of $H$ on the set 
$\{U_1,U_2,\dots,U_k\}$ is transitive.
\end{theorem}

\begin{proof}
By Proposition \ref{maxsub-partition}, there exists a partition 
$X=U_1\sqcup U_2\sqcup\dots\sqcup U_k$ such that each $U_i$ is a clopen set 
consisting of more than one point and $\Sti_{\cG}(U_1,U_2,\dots,U_k)\subset 
H\subset\St_{\cG}(U_1,U_2,\dots,U_k)$.  Moreover, the partition is unique 
and the induced action of $H$ on the set $\{U_1,U_2,\dots,U_k\}$ is 
transitive.  The partition cannot be trivial (with $k=1$ and $U_1=X$) as 
$\Sti_{\cG}(X)=\St_{\cG}(X)=\cG$ while $H\ne\cG$.  Hence $k\ge2$.  Besides, 
this is clearly not the partition into points.  Then 
$\St_{\cG}(U_1,U_2,\dots,U_k)\ne\cG$ due to Lemma \ref{partition-stab}.  
Since $H\subset\St_{\cG}(U_1,U_2,\dots,U_k)$ and $H$ is a maximal subgroup 
of $\cG$, it follows that $H=\St_{\cG}(U_1,U_2,\dots,U_k)$.
\end{proof}

\section{Nowhere dense closed sets}\label{nowhere-dense}

In this section we present a construction of a closed, nowhere dense subset 
$Y$ of a topological space $X$ and a group $H$ of homeomorphisms of $X$ 
that leave $Y$ invariant.  The closed set $Y$ is a Cantor set and the group 
$H$ acts minimally when restricted to $Y$.  In the case when $X$ is a 
Cantor set, the construction helps, given an ample group 
$\cG\subset\Homeo(X)$ that acts minimally on $X$, to produce uncountably 
many infinite, nowhere dense closed sets $Y$ such that the stabilizer 
$\St_{\cG}(Y)$ acts minimally on $Y$.  By Theorem \ref{maxsub-not-clopen}, 
all those stabilizers are maximal subgroups of $\cG$ provided that the 
group $\cG$ has Property \ref{prop-E}.  Besides, it follows from Lemma 
\ref{stab-closed2} that all those stabilizers are different from one
another and from the stabilizers of other closed sets.

The construction is general and we do not assume the topological space $X$ 
to be a Cantor set.  Suppose $U_0\supset U_1\supset U_2\supset\dots$ is a 
nested sequence of nonempty clopen subsets of $X$ and $x_0$ is a common 
point of all these sets.  Further suppose that for any $n\ge1$ we have a 
homeomorphism $g_n:X\to X$ such that the set $g_n(U_n)$ is disjoint from 
$U_n$ and, moreover, the union $U_n\cup g_n(U_n)$ is a proper subset of 
$U_{n-1}$.  Let $H$ be a subgroup of $\Homeo(X)$ generated by the 
generalized $2$-cycles $\de_{U_n;g_n}$, $n=1,2,\dots$  Let 
$Y=\overline{\Orb_H(x_0)}$, the closure of the orbit of the point $x_0$ 
under the natural action of the group $H$ on $X$.  Then $Y$ is a closed set 
invariant under the action of $H$.  Assuming all maps $g_n$, $n\ge1$ are 
taken from a group $G\subset\Homeo(X)$, the group $H$ is a subgroup of the 
ample group $\cG=\F(G)$ so that $H\subset\St_{\cG}(Y)$.  For any $n\ge1$, 
we let $f^{(0)}_n=\id_X$ and $f^{(1)}_n=\de_{U_n;g_n}$.  Note that 
$f^{(0)}_n(U_n)=U_n$ and $f^{(1)}_n(U_n)=g_n(U_n)$ are subsets of 
$U_{n-1}$.  Hence for any infinite string $\xi_1\xi_2\xi_3\dots$ of 0s and 
1s we have a nested sequence of clopen sets $U_0\supset 
f^{(\xi_1)}_1(U_1)\supset f^{(\xi_1)}_1f^{(\xi_2)}_2(U_2) 
\supset\dots$.  The intersection of these sets is nonempty.

\begin{lemma}\label{example-nowhere-dense}
In terms of the above construction, suppose that for any infinite string 
$\xi_1\xi_2\xi_3\dots$ of 0s and 1s the intersection of the nested sets 
$U_0\supset f^{(\xi_1)}_1(U_1)\supset f^{(\xi_1)}_1f^{(\xi_2)}_2(U_2) 
\supset\dots$ consists of a single point.  Then $Y$ is a Cantor set that is 
nowhere dense in $X$ and the group $H$ acts minimally when restricted to 
$Y$.
\end{lemma}

\begin{proof}
We keep the same notation as used in the construction.  First let us 
describe an abstract model behind the construction.  Let 
$\Xi=\{0,1\}^{\bN}$.  We regard any element $\xi\in\Xi$ as an infinite 
string $\xi=\xi_1\xi_2\xi_3\dots$, where each $\xi_i\in\{0,1\}$.  We endow 
$\Xi$ with the product topology (assuming the discrete topology on 
$\{0,1\}$).  Then $\Xi$ is a Cantor set.  Let $\{0,1\}^*$ be the set of all 
finite strings of 0s and 1s (including the empty string $\varnothing$).  
For any $w\in\{0,1\}^*$ we denote by $C_w$ the set of all infinite strings 
in $\Xi$ that begin with $w$.  Sets of the form $C_w$ are called 
\emph{cylinders}, they are clopen and form a base of the topology on $\Xi$.

For any $n\ge1$ let $\gamma_n$ be a transformation of $\Xi$ that changes 
only the $n$-th character in any infinite string that begins with (at 
least) $n-1$ consecutive zeros, and fixes all the other elements of $\Xi$.  
The transformation $\gamma_n$ is clearly an involution.  Besides, it is 
continuous.  Hence $\gamma_n\in\Homeo(\Xi)$.  Let $\Gamma$ be the subgroup 
of $\Homeo(\Xi)$ generated by $\gamma_1,\gamma_2,\gamma_3,\dots$.

For any $\xi=\xi_1\xi_2\xi_3\ldots\in\Xi$ and integer $n\ge0$ let 
$Z(\xi,n)$ denote a string $\zeta=\zeta_1\zeta_2\zeta_3\dots$ that begins 
with $n$ zeros and coincides with the string $\xi$ afterwards: $\zeta_k=0$ 
if $k\le n$ and $\zeta_k=\xi_k$ if $k>n$.  Clearly, $Z(\xi,0)=\xi$.  For 
any $n\ge1$ we have $Z(\xi,n)=Z(\xi,n-1)$ if $\xi_n=0$ and 
$Z(\xi,n)=\gamma_n(Z(\xi,n-1))$ if $\xi_n=1$.  It follows that the strings 
$Z(\xi,0),Z(\xi,1),Z(\xi,2),\dots$ all belong to the orbit 
$\Orb_\Gamma(\xi)$.  If two strings $\xi,\xi'\in\Xi$ coincide up to 
finitely many characters, then $Z(\xi,n)=Z(\xi',n)$ for all sufficiently 
large $n$, which implies that $\xi$ and $\xi'$ are in the same orbit of the 
group $\Gamma$.  We conclude that any orbit of $\Gamma$ has an element in 
every cylinder $C_w$, $w\in\{0,1\}^*$.  As a consequence, any orbit of 
$\Gamma$ is dense in $\Xi$.  Therefore the group $\Gamma$ acts minimally on 
$\Xi$.

Now we are going to show that the actual construction of the set $Y$ and 
the group $H$ agrees with our abstract model (respectively $\Xi$ and 
$\Gamma$).  First we define a map $\alpha:\Xi\to Y$.  For any finite string 
$w=\xi_1\xi_2\ldots\xi_k$ of 0s and 1s consisting of $k\ge1$ characters, 
let $f^{(w)}=f^{(\xi_1)}_1f^{(\xi_2)}_2\ldots f^{(\xi_k)}_k$ and 
$V^{(w)}=f^{(w)}(U_k)$.  Also, we let $f^{(\varnothing)}=\id_X$ and 
$V^{(\varnothing)}=U_0$.  Given an infinite string 
$\xi=\xi_1\xi_2\xi_3\ldots\in\Xi$, consider a nested sequence of clopen 
sets $V^{\varnothing}\supset V^{(\xi_1)}\supset 
V^{(\xi_1\xi_2)}\supset\ldots$.  By assumption, the intersection of all 
these sets consists of a single point, and we define $\alpha(\xi)$ to be 
that point.  Let us show that $\alpha(\xi)\in Y$.  Since the point $x_0$ 
belongs to each of the sets $U_0,U_1,U_2,\dots$, it follows that 
$f^{(\xi_1\xi_2\ldots\xi_k)}(x_0)\in V^{(\xi_1\xi_2\ldots\xi_k)}$ for 
$k=1,2,\dots$.  This implies that every limit point of the sequence 
$x_0,f^{(\xi_1)}_1(x_0),f^{(\xi_1\xi_2)}_2(x_0),\dots$ belongs to the 
intersection of the nested clopen sets 
$V^{\varnothing},V^{(\xi_1)},V^{(\xi_1\xi_2)},\dots$.  As $\alpha(\xi)$ is 
the only point in the intersection, we conclude that the sequence converges 
to $\alpha(\xi)$.  Since all points in the sequence belong to the orbit 
$\Orb_H(x_0)$, we obtain that $\alpha(\xi)\in\overline{\Orb_H(x_0)}=Y$.

By construction, $\alpha(C_w)\subset V^{(w)}$ for all $w\in\{0,1\}^*$.  For 
any strings $w,u\in\{0,1\}^*$ let $wu$ denote their concatenation.  Notice 
that $C_{wu}\subset C_w$ and $V^{(wu)}\subset V^{(w)}$.  Let $k$ be the 
number of characters in the string $w$.  Then $V^{(w)}=f^{(w)}(U_k)$, 
$V^{(w0)}=f^{(w)}(U_{k+1})$ and  
$V^{(w1)}=f^{(w)}\bigl(f^{(1)}_{k+1}(U_{n+1})\bigr)$.  By construction, 
$U_{k+1}$ and $f^{(1)}_{k+1}(U_{k+1})$ are disjoint subsets of $U_k$ and 
their union is different from $U_k$.  It follows that $V^{(w)}$ is the 
disjoint union of three nonempty clopen sets $V^{(w0)}$, $V^{(w1)}$ and 
$V^{(w)}\setminus \bigl(V^{(w0)}\cup V^{(w1)}\bigr)$.  Note that any two 
strings in $\{0,1\}^*$ neither of which is a beginning of the other can be 
written as $w0u$ and $w1u'$, where $w,u,u'\in\{0,1\}^*$.  Since 
$V^{(w0u)}\subset V^{(w0)}$ and $V^{(w1u')}\subset V^{(w1)}$, we obtain 
that the set $V^{(w0u)}$ is disjoint from $V^{(w1u')}$.  In particular, the 
sets $V^{(w)}$ and $V^{(w')}$ are disjoint if $w$ and $w'$ are two 
different strings with the same number of characters.

Next we prove that the map $\alpha$ is one-to-one.  Given two infinite 
strings $\xi,\xi'\in\Xi$, let $w$ be their longest common beginning.  If 
$\xi\ne\xi'$ then $w$ is a finite string.  Moreover, one of the strings 
$\xi$ and $\xi'$ belongs to $C_{w0}$ while the other is in $C_{w1}$.  Hence 
one of the points $\alpha(\xi)$ and $\alpha(\xi')$ belongs to $V^{(w0)}$ 
while the other is in $V^{(w1)}$.  By the above the sets $V^{(w0)}$ and 
$V^{(w1)}$ are disjoint.  As a consequence, $\alpha(\xi)\ne\alpha(\xi')$.

Given $\xi=\xi_1\xi_2\xi_3\ldots\in\Xi$, let us show that every open 
neighborhood of the point $\alpha(\xi)$ contains a clopen set of the form 
$V^{(w)}$, where the finite string $w\in\{0,1\}^*$ is a beginning of $\xi$.
By construction, $V^{(\xi_1)}\supset V^{(\xi_1\xi_2)}\supset 
V^{(\xi_1\xi_2\xi_3)}\supset\ldots$ and the intersection of these clopen 
sets is $\{\alpha(\xi)\}$.  Hence $X\setminus V^{(\xi_1)}\subset X\setminus 
V^{(\xi_1\xi_2)}\subset X\setminus V^{(\xi_1\xi_2\xi_3)}\subset\ldots$ and 
the union of these complements is $X\setminus\{\alpha(\xi)\}$.  Adding an 
arbitrary open neighborhood $U$ of $\alpha(\xi)$ to the complements, we 
obtain an open cover of $X$.  Since $X$ is compact, there is a finite 
subcover.  It further follows that $X$ is covered by two sets $U$ and 
$X\setminus V^{(\xi_1\xi_2\ldots\xi_n)}$ provided that $n$ is large 
enough.  Equivalently, $V^{(\xi_1\xi_2\ldots\xi_n)}\subset U$ if $n$ is 
large enough.

Next we prove that the map $\alpha$ is continuous.  Let $\xi\in\Xi$ and $U$ 
be an open neighborhood of the point $\alpha(\xi)$.  We need to show that 
the preimage $\alpha^{-1}(U)$ contains an open neighborhood of $\xi$.  By 
the above $U$ contains a set $V^{(w)}$ for some string $w$ that $\xi$ 
begins with.  Then the cylinder $C_w$ is a clopen neighborhood of $\xi$.  
It is contained in $\alpha^{-1}(U)$ since $\alpha(C_w)\subset V^{(w)}$.

Next we prove that $f^{(1)}_n\alpha=\alpha\gamma_n$ on $\Xi$ for all 
$n\ge1$.  Take any $\xi=\xi_1\xi_2\xi_3\ldots\in\Xi$ and let 
$\gamma_n(\xi)=\eta_1\eta_2\eta_3\ldots$.  By the above, 
$f^{(\xi_1\xi_2\ldots\xi_k)}(x_0)\to\alpha(\xi)$ and 
$f^{(\eta_1\eta_2\ldots\eta_k)}(x_0)\to\alpha(\gamma_n(\xi))$ as 
$k\to\infty$.  Then $f^{(1)}_nf^{(\xi_1\xi_2\ldots\xi_k)}(x_0)\to 
f^{(1)}_n(\alpha(\xi))$ as $k\to\infty$.  We are going to show that 
$f^{(1)}_nf^{(\xi_1\xi_2\ldots\xi_k)}(x_0)= 
f^{(\eta_1\eta_2\ldots\eta_k)}(x_0)$ for $k\ge n$.  This will imply that 
$f^{(1)}_n(\alpha(\xi))=\alpha(\gamma_n(\xi))$.  First consider the case 
when $\xi_i=0$ for all $i<n$.  Then $\eta_n\ne\xi_n$ and 
$\eta_i=\xi_i$ for all $i\ne n$.  Since $f^{(1)}_n$ is an involution and 
$f^{(0)}_n=\id_X$, we obtain that $f^{(1)}_nf^{(\xi_n)}_n=f^{(\eta_n)}_n$.  
As $f^{(\xi_i)}_i=\id_X$ for $i<n$, it follows that 
$f^{(1)}_nf^{(\xi_1\xi_2\ldots\xi_k)}=f^{(\eta_1\eta_2\ldots\eta_k)}$ for 
all $k\ge n$.  Now consider the case when $\xi_i=1$ for some $i<n$.  In 
this case, $\gamma_n(\xi)=\xi$.  Since $x_0\in U_k$ for any $k\ge n$, it 
follows that the point $y=f^{(\xi_{i+1})}_{i+1}f^{(\xi_{i+2})}_{i+2}\ldots 
f^{(\xi_k)}_k(x_0)$ belongs to $U_i$.  Then the point $f^{(1)}_i(y)$ does 
not belong to $U_i$, and neither does $f^{(\xi_1\xi_2\ldots\xi_k)}(x_0)$.  
Since $\supp\bigl(f^{(1)}_n\bigr)\subset U_{n-1}\subset U_i$, we conclude 
that $f^{(1)}_nf^{(\xi_1\xi_2\ldots\xi_k)}(x_0)= 
f^{(\xi_1\xi_2\ldots\xi_k)}(x_0)= f^{(\eta_1\eta_2\ldots\eta_k)}(x_0)$.

Next we prove that the map $\alpha$ is onto.  For any $n\ge1$, the equality 
$f^{(1)}_n\alpha=\alpha\gamma_n$ implies that $f^{(1)}_n(\alpha(\Xi))= 
\alpha(\gamma_n(\Xi))=\alpha(\Xi)$.  Since the group $H$ is generated by 
the maps $f^{(1)}_1,f^{(1)}_2,f^{(1)}_3,\dots$, it follows that 
$h(\alpha(\Xi))=\alpha(\Xi)$ for all $h\in H$.  Notice that 
$x_0\in\alpha(\Xi)$, namely, $x_0$ is the image of the infinite string of 
all 0s.  Since the set $\alpha(\Xi)$ is invariant under the action of the 
group $H$, we obtain that $\Orb_H(x_0)\subset\alpha(\Xi)$.  Further, the 
set $\alpha(\Xi)$ is compact as $\Xi$ is compact and the map $\alpha$ is 
continuous.  As a consequence, $\alpha(\Xi)$ is a closed subset of $X$.  
Since $\Orb_H(x_0)\subset\alpha(\Xi)\subset Y$ and 
$Y=\overline{\Orb_H(x_0)}$, we conclude that $\alpha(\Xi)=Y$.

Since the map $\alpha$ is one-to-one and onto, it is invertible.  The 
inverse map $\alpha^{-1}:Y\to\Xi$ is continuous since $\alpha$ is 
continuous and $\Xi$ is compact.  Hence $\alpha$ is a homeomorphism.  As a 
consequence, $Y$ is a Cantor set.

Now we derive minimality of the action of the group $H$ on $Y$ from 
minimality of the action of the group $\Gamma$ on $\Xi$.  For any closed 
set $Y_0\subset Y$, the set $\alpha^{-1}(Y_0)$ is a closed subset of 
$\Xi$.  Assuming $Y_0$ is invariant under the action of $H$, we have 
$f^{(1)}_n(Y_0)=Y_0$ for all $n\ge1$.  By the above 
$f^{(1)}_n\alpha=\alpha\gamma_n$ on $\Xi$.  Therefore 
$\gamma_n=\alpha^{-1}f^{(1)}_n\alpha$ on $\Xi$.  Then 
$\gamma_n(\alpha^{-1}(Y_0))=\alpha^{-1}f^{(1)}_n(Y_0)=\alpha^{-1}(Y_0)$.  
Since the group $\Gamma$ is generated by maps 
$\gamma_1,\gamma_2,\gamma_3,\dots$, it follows that 
$\gamma(\alpha^{-1}(Y_0))=\alpha^{-1}(Y_0)$ for all $\gamma\in\Gamma$, that 
is, the set $\alpha^{-1}(Y_0)$ is invariant under the action of the group 
$\Gamma$ on $\Xi$.  Minimality of the latter action implies that 
$\alpha^{-1}(Y_0)$ is either $\Xi$ or the empty set.  Then $Y_0$ is either 
$Y$ or the empty set.

It remains to prove that the set $Y$ is nowhere dense in $X$.  Since $Y$ is 
closed, it is enough to show that it has no interior points.  Let $y\in Y$ 
and $U$ be an open neighborhood of the point $y$.  We have $y=\alpha(\xi)$ 
for some $\xi\in\Xi$.  By the above $U$ contains a set $V^{(w)}$ for some 
string $w\in\{0,1\}^*$ that $\xi$ begins with.  Also by the above, 
$V^{(w)}$ is the disjoint union of three nonempty clopen sets $V^{(w0)}$, 
$V^{(w1)}$ and $V^{(w)}\setminus \bigl(V^{(w0)}\cup V^{(w1)}\bigr)$.  Take 
any string $\xi'\in\Xi$ and let $w'$ be the beginning of $\xi'$ that has 
the same number of characters as $w$.  If $w'\ne w$ then the set 
$V^{(w')}$, which contains the point $\alpha(\xi')$, is disjoint from 
$V^{(w)}$ so that $\alpha(\xi')\notin V^{(w)}$.  If $w'=w$ then 
$\alpha(\xi')\in V^{(w0)}\cup V^{(w1)}$.  We conclude that the nonempty set 
$V^{(w)}\setminus \bigl(V^{(w0)}\cup V^{(w1)}\bigr)$ is disjoint from 
$\alpha(\Xi)=Y$.  Therefore the open set $U$ is not fully contained in $Y$.
\end{proof}

One can show that the group $H$ considered in Lemma 
\ref{example-nowhere-dense} is locally finite, that is, any finitely 
generated subgroup of $H$ is finite.

\begin{proposition}\label{plenty-of-examples}
Suppose $X$ is a Cantor set and $\cG\subset\Homeo(X)$ is an ample group 
that acts minimally on $X$.  Then there are uncountably many infinite, 
nowhere dense closed sets $Y\subset X$ such that the stabilizer 
$\St_{\cG}(Y)$ acts minimally when restricted to $Y$.
\end{proposition}

\begin{proof}
Let $\rho:X\times X\to\bR$ be any distance function on the Cantor set $X$ 
compatible with its topology.  For any $x\in X$ and $r>0$ denote by 
$B(x,r)$ the ball of radius $r$ centered at the point $x$: $B(x,r)=\{y\in 
X\mid \rho(y,x)<r\}$.  The ball $B(x,r)$ is an open neighborhood of $x$ of 
diameter at most $2r$.

We are going to construct a nested sequence $U_0\supset U_1\supset 
U_2\supset\dots$ of nonempty clopen subsets of $X$ and two sequences, 
$g_1,g_2,g_3,\dots$ and $h_1,h_2,h_3,\dots$, of elements of the group 
$\cG$.  The elements of $\cG$ will be chosen so that for any $n\ge1$ the 
sets $U_n$, $g_n(U_n)$ and $h_n(U_n)$ are disjoint subsets of $U_{n-1}$.  
Assuming this, generalized $2$-cycles $f^{(1)}_n=\de_{U_n;g_n}$ and 
$f^{(2)}_n=\de_{U_n;h_n}$ are defined.  Since $g_n,h_n\in\cG$, the maps 
$f^{(1)}_n$ and $f^{(2)}_n$ belong to the group $\F(\cG)=\cG$.  We also let 
$f^{(0)}_n=\id_X$.  The clopen sets $U_0,U_1,U_2,\dots$ will be chosen so 
that for any string $\xi_1\xi_2\ldots\xi_n$ of 0s, 1s and 2s, the set 
$f^{(\xi_1)}_1f^{(\xi_2)}_2\ldots f^{(\xi_n)}_n(U_n)$ is of diameter at 
most $2^{-n}$.

The construction is done inductively.  First we let $U_0=X$.  Now assume 
that for some $n\ge1$ the set $U_{n-1}$ is already chosen and so are the 
maps $g_k$ and $h_k$, $1\le k\le n-1$.  Since $X$ is a Cantor set, the 
nonempty clopen set $U_{n-1}$ is infinite.  Therefore we can find three 
disjoint nonempty clopen sets $V,V',V''\subset U_{n-1}$.  Take any point 
$y\in V$.  Since the group $\cG$ acts minimally on $X$, the orbit 
$\Orb_{\cG}(y)$ is dense in $X$.  In particular, this orbit has a point in 
$V'$ and in $V''$.  Hence we can choose $g_n,h_n\in\cG$ such that 
$g_n(y)\in V'$ and $h_n(y)\in V''$.  Then $W=V\cap g_n^{-1}(V')\cap 
h_n^{-1}(V'')$ is a clopen neighborhood of the point $y$.  By construction, 
$W\subset V$, $g_n(W)\subset V'$ and $h_n(W)\subset V''$, which implies 
that $W$, $g_n(W)$ and $h_n(W)$ are disjoint subsets of $U_{n-1}$.  Now 
that the maps $g_k$ and $h_k$ are chosen for any $k$, $1\le k\le n$, we can 
define the maps $f^{(i)}_k$, $i\in\{0,1,2\}$, $1\le k\le n$ as described 
above.  For any string $w=\xi_1\xi_2\ldots\xi_n$ of 0s, 1s and 2s that has 
exactly $n$ characters, we let $f^{(w)}=f^{(\xi_1)}_1f^{(\xi_2)}_2\ldots 
f^{(\xi_n)}_n$ and $y_w=f^{(w)}(y)$.  Further, let $\widetilde{W}$ be the 
intersection of $W$ and all sets of the form 
$\bigl(f^{(w)}\bigr)^{-1}\bigl(B(y_w,2^{-n-1})\bigr)$.  Then 
$\widetilde{W}$ is an open neighborhood of the point $y$.  As $X$ is a 
Cantor set, the set $\widetilde{W}$ contains a clopen neighborhood of $y$.  
We choose the latter as $U_n$.  The clopen set $U_n$ is not empty since 
$y\in U_n$.  Since $U_n\subset\widetilde{W}\subset W$, the sets $U_n$, 
$g_n(U_n)$ and $h_n(U_n)$ are disjoint subsets of $U_{n-1}$.  By 
construction, $f^{(w)}(U_n)\subset B(y_w,2^{-n-1})$ for any string 
$w=\xi_1\xi_2\ldots\xi_n$ of 0s, 1s and 2s, which implies that the diameter 
of the set $f^{(w)}(U_n)$ is at most $2^{-n}$.  The inductive step of the 
construction is complete.

Let $\Xi=\{0,1,2\}^{\bN}$.  We regard any element $\xi\in\Xi$ as an 
infinite string $\xi=\xi_1\xi_2\xi_3\ldots$, where each 
$\xi_i\in\{0,1,2\}$. Note that for any $n\ge1$ the sets 
$f^{(0)}_n(U_n)=U_n$, $f^{(1)}_n(U_n)=g_n(U_n)$ and 
$f^{(2)}_n(U_n)=h_n(U_n)$ are subsets of $U_{n-1}$.  Therefore for any 
infinite string $\xi=\xi_1\xi_2\xi_3\ldots$ of 0s, 1s and 2s we have a 
nested sequence of clopen sets $U_0\supset f^{(\xi_1)}_1(U_1)\supset 
f^{(\xi_1)}_1f^{(\xi_2)}_2(U_2)\supset\dots$.  By construction, the 
diameters of those sets tend to $0$.  Hence their intersection consists of
a single point, which we denote $\beta(\xi)$.  Now we have a map 
$\beta:\Xi\to X$.  Let $x_0=\beta(000\dots)$.  The point $x_0$ is in the 
intersection of sets $U_0,U_1,U_2,\dots$.

Consider the set $\Omega=\{1,2\}^{\bN}$, which is a subset of $\Xi$.  For 
any infinite string $\om=\om_1\om_2\om_3\ldots\in\Omega$, let $H_\om$ 
denote a subgroup of $\cG$ generated by elements $f^{(\om_n)}_n$, $n\ge1$.  
Further, let $Y_\om=\overline{\Orb_{H_\om}(x_0)}$.  Then $Y_\om$ is a
closed subset of $X$ invariant under the action of the group $H_\om$.  
Since for any $n\ge1$ the sets $U_n$, $g_n(U_n)$ and $h_n(U_n)$ are 
disjoint subsets of $U_{n-1}$, it follows that the union of $U_n$ and 
$f^{(\om_n)}_n(U_n)$ is a proper subset of $U_{n-1}$.  Now Lemma 
\ref{example-nowhere-dense} implies that $Y_\om$ is a Cantor set which is 
nowhere dense in $X$, and the group $H_\om$ acts minimally when restricted 
to $Y_\om$.  Since $H_\om\subset\St_{\cG}(Y_\om)$, the stabilizer 
$\St_{\cG}(Y_\om)$ also acts minimally when restricted to $Y_\om$.

As the set $\Omega$ is uncountable, it remains to demonstrate that 
$Y_\om\ne Y_{\om'}$ whenever $\om$ and $\om'$ are different elements of 
$\Om$.  Just as in the proof of Lemma \ref{example-nowhere-dense} we showed 
that the map $\alpha$ is one-to-one, we can show that the map $\beta$ is 
one-to-one.  Just as in the proof of Lemma \ref{example-nowhere-dense} we 
showed that the map $\alpha$ is onto, we can show for any 
$\om=\om_1\om_2\om_3\ldots\in\Omega$ that the set $Y_\om$ consists of all 
points of the form $\beta(\xi_1\xi_2\xi_3\ldots)$, where 
$\xi_n\in\{0,\om_n\}$ for each $n\ge1$.  Since the map $\beta$ is 
one-to-one, it follows that $Y_\om\cap\beta(\Omega)=\{\beta(\om)\}$.  As a 
consequence, $Y_\om\ne Y_{\om'}$ whenever $\om\ne\om'$.
\end{proof}

\section{Property \ref{prop-E}}\label{property-E}

The majority of our results on maximal subgroups of ample groups obtained
in Section \ref{maxsub} require that the ample group has Property
\ref{prop-E} (to be precise, these are Theorems \ref{maxsub-not-clopen}, 
\ref{maxsub-clopen}, \ref{maxsub-3-or-more} and \ref{maxsub-2-or-more}).   
Hence they have no value until we provide examples of ample groups with 
this property.  We should note that Property \ref{prop-E} is complex and 
its verification is not going to be easy.  We begin with establishing a 
weakened version of Property \ref{prop-E} that involves the generalized 
symmetric groups.

\begin{proposition}\label{E-weakened}
Suppose a group $G\subset\Homeo(X)$ acts minimally on $X$.  Then for any
clopen sets $U_1$ and $U_2$ that intersect, the generalized symmetric group
over the local subgroup $\F_{U_1\cup U_2}(G)$ is generated by the union of
generalized symmetric groups over $\F_{U_1}(G)$ and $\F_{U_2}(G)$:
$\sS(\F_{U_1\cup U_2}(G))=\bigl\langle\sS(\F_{U_1}(G))\cup
\sS(\F_{U_2}(G))\bigr\rangle$.
\end{proposition}

\begin{proof}
For any clopen set $W\subset X$ let $D_W$ denote the set of all generalized 
$2$-cycles in the local subgroup $\F_W(G)$.  Since the group $\F_W(G)$ is 
ample, it follows from Lemma \ref{generated-by-2-cycles} that the 
generalized symmetric group $\sS(\F_W(G))$ is generated by $D_W$.  As a 
consequence, for any clopen sets $U_1,U_2\subset X$ the group 
$\sS(\F_{U_1\cup U_2}(G))$ is generated by $D_{U_1\cup U_2}$ while the 
group $\bigl\langle\sS(\F_{U_1}(G))\cup\sS(\F_{U_2}(G))\bigr\rangle$ is 
generated by $D_{U_1}\cup D_{U_2}$.  Note that the sets $D_{U_1}$ and 
$D_{U_2}$ are contained in $D_{U_1\cup U_2}$ since the local subgroups 
$\F_{U_1}(G)$ and $\F_{U_2}(G)$ are contained in $\F_{U_1\cup U_2}(G)$.  
To prove the proposition, it is enough to show that, assuming the 
intersection $U_1\cap U_2$ is nonempty, every element of $D_{U_1\cup U_2}$ 
belongs to the group $\langle D_{U_1}\cup D_{U_2}\rangle$.  This is trivial 
when one of the sets $U_1$ and $U_2$ contains the other.  Hence we may 
further assume that each of the sets $P_1=U_1\setminus U_2$, 
$P_2=U_2\setminus U_1$ and $P_3=U_1\cap U_2$ is nonempty.

Given $f\in D_{U_1\cup U_2}$, we have $f=\de_{U;f'}$ for some clopen set 
$U\subset X$ and homeomorphism $f':X\to X$.  Observe that the generalized 
$2$-cycle $\de_{U;f}$ is defined and $f=\de_{U;f}$.  First we consider a 
special case when one of the sets $U$ and $f(U)$ is contained in $P_1$ and 
the other is contained in $P_2$.  Note that $f^{-1}=f$.  Therefore the 
generalized $2$-cycle $\de_{f(U);f}$ is defined and 
$\de_{f(U);f}=\de_{U;f}=f$.  Hence we may assume without loss of generality 
that $U\subset P_1$ and $f(U)\subset P_2$.  Since the group $G$ acts 
minimally on $X$, Lemma \ref{map-piecewise} implies that for some $k\ge1$ 
there exist disjoint clopen sets $V_1,V_2,\dots,V_k$ and maps 
$g_1,g_2,\dots,g_k\in G$ such that $U=V_1\sqcup V_2\sqcup\dots\sqcup V_k$ 
and $g_i(V_i)\subset P_3$ for all $i$.  As the set $f(U)$ is disjoint from 
$U$, it follows that the sets 
$V_1,V_2,\dots,V_k,f(V_1),f(V_2),\dots,f(V_k)$ are disjoint from one 
another.  In particular, the generalized $2$-cycle $\de_{V_i;f}$ is defined 
for each $i$, $1\le i\le k$.  Since the map $f$ belongs to the ample group 
$\F(G)$, so does $\de_{V_i;f}$.  Since $\supp(\de_{V_i;f})=V_i\cup f(V_i)$, 
the maps $\de_{V_1;f},\de_{V_2;f},\dots,\de_{V_k;f}$ have pairwise disjoint 
supports.  Therefore these maps commute with one another.  Each 
$\de_{V_i;f}$ coincides with $\de_{U;f}$ on $V_i\cup f(V_i)$ and with the 
identity map everywhere else.  Since sets $V_i\cup f(V_i)$, $1\le i\le k$ 
form a partition of the set $U\cup f(U)$, it follows that 
$f=\de_{V_1;f}\,\de_{V_2;f}\dots \de_{V_k;f}$.  We are going to show that 
each $\de_{V_i;f}$ belongs to the group $\langle D_{U_1}\cup 
D_{U_2}\rangle$.  Then $f\in\langle D_{U_1}\cup D_{U_2}\rangle$ as well.

For any $i$, $1\le i\le k$, the sets $V_i$, $f(V_i)$ and $g_i(V_i)$ are 
disjoint since $V_1\subset P_1$, $f(V_i)\subset P_2$ and $g_i(V_i)\subset 
P_3$.  Therefore the generalized permutation 
$h_\pi=\mu[V_i;\id_X,f,g_i;\pi]$ is defined for any permutation 
$\pi\in\sS_3$.  Moreover, $h_\pi\in\F(G)$.  Since 
$(1\,2)=(1\,3)(2\,3)(1\,3)$, it follows from Lemma \ref{gen-perm-homo} that 
$h_{(1\,2)}=h_{(1\,3)}h_{(2\,3)}h_{(1\,3)}$.  Observe that 
$h_{(1\,2)}=\de_{V_i;f}$, $h_{(1\,3)}=\de_{V_i;g_i}$ and 
$h_{(2\,3)}=\de_{f(V_i);g_if}$.  Then $\supp(h_{(1\,3)})=V_i\cup 
g_i(V_i)\subset U_1$ and $\supp(h_{(2\,3)})=f(V_i)\cup g_i(V_i)\subset 
U_2$, which means that $h_{(1\,3)}\in\F_{U_1}(G)$ and 
$h_{(2\,3)}\in\F_{U_2}(G)$.  Hence $h_{(1\,3)}\in D_{U_1}$ and 
$h_{(2\,3)}\in D_{U_2}$.  Consequently, the map 
$\de_{V_i;f}=h_{(1\,2)}h_{(2\,3)}h_{(1\,2)}$ belongs to the group $\langle 
D_{U_1}\cup D_{U_2}\rangle$.

Now consider an arbitrary generalized $2$-cycle $f=\de_{U;f}$ in 
$D_{U_1\cup U_2}$.  Note that the support $\supp(f)=U\cup f(U)$ is 
contained in $U_1\cup U_2$.  For any $i,j\in\{1,2,3\}$ let $Q_{ij}=P_i\cap 
f(P_j)$.  Since $U_1\cup U_2=P_1\sqcup P_2\sqcup P_3$, $f(U_1\cup 
U_2)=U_1\cup U_2$ and $f^{-1}=f$, it follows that $U_1\cup U_2$ is the 
disjoint union of $9$ clopen sets $Q_{ij}$, $i,j\in\{1,2,3\}$.  
Consequently, $U$ is the disjoint union of $9$ clopen sets $U\cap Q_{ij}$, 
$i,j\in\{1,2,3\}$.  Just as above, we obtain that generalized $2$-cycles 
$f_{ij}=\de_{U\cap Q_{ij};f}$, $i,j\in\{1,2,3\}$ are defined, they commute 
with one another, and $f$ is their product.  All $9$ maps belong to 
$D_{U_1\cup U_2}$.  By construction, $U\cap Q_{ij}\subset P_i$ and 
$f_{ij}(U\cap Q_{ij})=f(U\cap Q_{ij})\subset P_j$.  Hence the maps $f_{12}$ 
and $f_{21}$ have been covered by the special case.  By the above both maps 
belong to $\langle D_{U_1}\cup D_{U_2}\rangle$.  As for the other $7$ maps, 
the supports of $f_{11}$, $f_{13}$, $f_{31}$ and $f_{33}$ are contained in 
$U_1$ while the supports of $f_{22}$, $f_{23}$ and $f_{32}$ are contained 
in $U_2$.  Hence $f_{11},f_{13},f_{31},f_{33}\in D_{U_1}$ and 
$f_{22},f_{23},f_{32}\in D_{U_2}$.  We conclude that all $9$ maps $f_{ij}$, 
$i,j\in\{1,2,3\}$ belong to the group $\langle D_{U_1}\cup 
D_{U_2}\rangle$.  Then $f\in \langle D_{U_1}\cup D_{U_2}\rangle$ as well.
\end{proof}

\begin{remark}\label{finite-set}
If $X$ is a finite set with the discrete topology, then all subsets of $X$ 
are clopen.  Hence the group $\Homeo(X)$ coincides with the symmetric group 
$\sS_X$ of all permutations on $X$.  Any permutation on $X$ is also a 
generalized permutation (in the sense of Definition \ref{def-gen-perm}).  
Therefore for any group $G\subset\sS_X$ the ample group $\F(G)$ coincides 
with the generalized symmetric group $\sS(G)$.  Moreover, for any subset 
$U\subset X$ the local subgroup $\F_U(G)$ coincides with $\sS(\F_U(G))$.  
In this case the conclusion of Proposition \ref{E-weakened} means that the 
group $G$ has Property \ref{prop-E}.  The hypothesis of Proposition 
\ref{E-weakened} is that $G$ acts minimally on $X$.  For a finite $X$, this 
means transitivity, that is, the entire set $X$ forms a single orbit.  Then 
every permutation on $X$ is pointwise (and hence piecewise) an element of 
$G$, which implies that $\F(G)=\sS_X$.  Hence Theorems \ref{maxsub-clopen}, 
\ref{maxsub-3-or-more} and \ref{maxsub-2-or-more} apply to the symmetric 
group $\sS_X$ (Theorem \ref{maxsub-not-clopen} is vacuous if $X$ is 
finite).  One of the hypotheses of Theorem \ref{maxsub-2-or-more} is that 
the subgroup $H$ of the ample group $\cG$ contains a local subgroup $\cG_U$ 
for some clopen set $U$ containing more than one point.  In the case 
$\cG=\sS_X$, this is equivalent to the condition that $H$ contains a 
transposition.  Thus, as a byproduct of our proceedings, we obtain the 
following classical result.
\end{remark}

\begin{proposition}[Jordan {\cite{Jordan1870}}]
\label{prim-subgroup}
Any proper primitive subgroup of a finite symmetric group $\sS_n$ contains 
no transposition.
\end{proposition}

Proposition \ref{prim-subgroup} is a direct corollary of the statement 
proved in Note C of the book \cite{Jordan1870} (for an explicit 
formulation, see Theorem 3.3A in the book \cite{DM}).

\begin{proof}[Proof of Proposition \ref{prim-subgroup}]
Suppose $H$ is a maximal subgroup of $\sS_n$ that contains a transposition 
$(x\,y)$.  Then $H$ contains the local subgroup 
$\F_{\{x,y\}}(\sS_n)=\{\id_X,(x\,y)\}$.  The symmetric group $\sS_n$ has 
Property \ref{prop-E} (as explained in Remark \ref{finite-set}).  In the 
case when the subgroup $H$ acts transitively on the set $X$, it follows
from Theorem \ref{maxsub-2-or-more} that $H$ is the stabilizer of a 
partition of $X$ into at least two sets, each consisting of at least two 
points.  Hence $H$ is not primitive.  In the case when the action of $H$ on 
$X$ is not transitive, it is not primitive either.

Now assume $H_0$ is a proper primitive subgroup of $\sS_n$.  The subgroup 
$H_0$ can be extended to a maximal subgroup $H$.  Since $H_0$ is primitive, 
so is $H$.  Then it follows from the above that the group $H$ contains no 
transposition.  The same is true for its subgroup $H_0$.
\end{proof}

\begin{remark}\label{ThompsonV}
Another case when $\F_U(G)=\sS(\F_U(G))$ for any clopen set $U\subset X$ so 
that Proposition \ref{E-weakened} yields Property \ref{prop-E} is when 
every nontrivial local subgroup $\F_U(G)$ is simple.  Indeed, 
$\sS(\F_U(G))$ is always a normal subgroup of $\F_U(G)$ (due to Lemma 
\ref{gen-perm-normal-subgroup}).  One example of such a group is Thompson's 
group $V$.  The group $V$, which is simple, can be represented as an ample 
group acting minimally on the Cantor set $\{0,1\}^{\bN}$ (see, e.g., 
Example 2.3.2.1 in the book \cite{Nek2022} or Definition 1.1 in 
\cite{BBQS22}).  It is not hard to show that for every nonempty clopen set 
$U\subset\{0,1\}^{\bN}$ there exists a homeomorphism $\phi:\{0,1\}^{\bN}\to 
U$ that is piecewise an element of the group $V$.  This homeomorphism 
conjugates the ample group $V$ with its local subgroup $V_U$.  Hence the 
subgroup $V_U$ is isomorphic to the entire group $V$.  In particular, $V_U$ 
is simple.
\end{remark}

In general, Proposition \ref{E-weakened} does not yield Property 
\ref{prop-E}, but it helps in deriving this property from simpler 
properties introduced in Section \ref{prop}.

\begin{lemma}\label{CI-implies-E}
For any ample group acting minimally, Property \ref{prop-CI} implies
Property \ref{prop-E}.
\end{lemma}

\begin{proof}
Suppose $\cG\subset\Homeo(X)$ is an ample group that acts minimally on $X$ 
and has Property \ref{prop-CI}.  For any clopen sets $U_1,U_2\subset X$ the 
local subgroups $\cG_{U_1}$ and $\cG_{U_2}$ are contained in the local 
subgroup $\cG_{U_1\cup U_2}$.  Assuming the intersection $U_1\cap U_2$ is 
nonempty, we need to show that, conversely, every map $f\in\cG_{U_1\cup 
U_2}$ belongs to the group $\langle\cG_{U_1}\cup\cG_{U_2}\rangle$.  This is 
trivial when one of the sets $U_1$ and $U_2$ contains the other.  Hence we 
may further assume that each of the sets $U_1\setminus U_2$, $U_2\setminus 
U_1$ and $U_1\cap U_2$ is nonempty.

Since the group $\cG$ acts minimally and $U_1\cap U_2\ne\emptyset$, 
Proposition \ref{E-weakened} implies that the generalized symmetric group 
$\sS(\cG_{U_1\cup U_2})$ is generated by its subgroups $\sS(\cG_{U_1})$ and 
$\sS(\cG_{U_2})$.  Since $\sS(\cG_{U_1})\subset\cG_{U_1}$ and 
$\sS(\cG_{U_2})\subset\cG_{U_2}$, we conclude that 
$\sS(\cG_{U_1\cup U_2})\subset\langle\cG_{U_1}\cup\cG_{U_2}\rangle$.

Given a map $f\in\cG_{U_1\cup U_2}$, let $V_+=\{x\in U_1\mid f(x)\notin 
U_1\}$ and $V_-=\{x\in U_1\mid f^{-1}(x)\notin U_1\}$.  The sets $V_+$ and 
$V_-$ are clopen as $V_+=U_1\cap f^{-1}(X\setminus U_1)$ and $V_-=U_1\cap 
f(X\setminus U_1)$.  First we consider an easy case when 
$V_+=V_-=\emptyset$.  An equivalent condition is that $f(U_1)=U_1$.  Since 
$\supp(f)\subset U_1\cup U_2$, we also have $f(U_2\setminus 
U_1)=U_2\setminus U_1$.  Let $f_1$ be a map on $X$ that coincides with $f$ 
on $U_1$ and with the identity map everywhere else.  Let $f_2$ be another 
map on $X$ that coincides with $f$ on $U_2\setminus U_1$ and with the 
identity map everywhere else.  Since $f(U_1)=U_1$ and $f(U_2\setminus 
U_1)=U_2\setminus U_1$, it follows that $f_1$ and $f_2$ are 
homeomorphisms.  By construction, both $f_1$ and $f_2$ are piecewise 
elements of the group $\cG$.  Since $\cG$ is ample, it contains them.  Also 
by construction, $\supp(f_1)\subset U_1$, $\supp(f_2)\subset U_2\setminus 
U_1$ and $f=f_1f_2$.  Therefore $f_1\in\cG_{U_1}$, $f_2\in\cG_{U_2}$ and 
$f\in\langle\cG_{U_1}\cup\cG_{U_2}\rangle$.

Next consider the case when one of the sets $V_+$ and $V_-$ coincides with 
$U_1$.  Then the other set coincides with $U_1$ as well.  In this case, the 
image $f(U_1)$ is disjoint from $U_1$.  Hence the generalized $2$-cycle 
$g=\de_{U_1;f}$ is defined.  Moreover, $g\in\F(\cG)=\cG$.  Note that 
$\supp(g)\subset\supp(f)\subset U_1\cup U_2$.  Therefore $g\in\cG_{U_1\cup 
U_2}$.  Then the map $h=gf$ is also in $\cG_{U_1\cup U_2}$.  Since $g$ 
coincides with $f^{-1}$ on the set $f(U_1)$, we obtain that $h$ coincides 
with the identity map on $U_1$.  It follows that $h\in\cG_{U_2\setminus 
U_1}\subset\cG_{U_2}$.  Besides, $g\in\sS(\cG_{U_1\cup 
U_2})\subset\langle\cG_{U_1}\cup\cG_{U_2}\rangle$.  As a consequence, the 
map $f=g^{-1}h$ belongs to $\langle\cG_{U_1}\cup\cG_{U_2}\rangle$.

Now consider the case when $V_+$ and $V_-$ are both different from the 
empty set and $U_1$.  In this case the unions $V_+\cup(X\setminus U_1)$ and 
$V_-\cup(X\setminus U_1)$ are both different from $X$.  It is easy to see 
that $f$ maps the first of these unions onto the second.  Since the set 
$X\setminus U_1$ is disjoint from $V_+$ and $V_-$, Property \ref{prop-CI} 
of the group $\cG$ implies that $f_0(V_+)=V_-$ for some $f_0\in\cG$.  Note 
that the set $f^{-1}(V_-)$ is disjoint from $U_1$.  It follows that the 
generalized $2$-cycle $g=\de_{V_+;f^{-1}f_0}$ is defined and belongs to 
$\F(\cG)=\cG$.  Since $\supp(f)\subset U_1\cup U_2$, we obtain that 
$f^{-1}(V_-)\subset U_2\setminus U_1$.  Then the support $\supp(g)=V_+\cup 
f^{-1}(V_-)$ is contained in $U_1\cup U_2$ as well.  Therefore the map $g$ 
belongs to $\cG_{U_1\cup U_2}$, and so does the map $h=fg$.  By 
construction, $h(V_+)=V_-$ and $h(U_1\setminus V_+)=f(U_1\setminus V_+)$, 
which implies that $h(U_1)=U_1$.  Hence the map $h$ has been covered by the 
first case.  By the above, $h\in\langle\cG_{U_1}\cup\cG_{U_2}\rangle$.  
Since $g\in\sS(\cG_{U_1\cup U_2})$, we conclude that the map $f=hg^{-1}$ 
belongs to $\langle\cG_{U_1}\cup\cG_{U_2}\rangle$ as well.

It remains to consider the case when one of the sets $V_+$ and $V_-$ is 
empty while the other is not.  This case will be reduced to the previous 
one.  Without loss of generality, we may assume that $V_+=\emptyset$ and 
$V_-\ne\emptyset$ (otherwise we could replace $f$ with $f^{-1}$).  Then 
$f(U_1)$ is a proper subset of $U_1$.  It follows that for any $j\in\bZ$ 
the set $f^{j+1}(U_1)$ is a proper subset of $f^j(U_1)$.  Note that 
$V_-=U_1\setminus f(U_1)$.  Since $f^j(V_-)=f^j(U_1)\setminus f^{j+1}(U_1)$ 
for all $j\in\bZ$, we obtain that the sets $f^j(V_-)$, $j\in\bZ$ are 
disjoint from one another.  The set $f^j(V_-)$ is contained in $U_1$ for 
$j\ge0$ and disjoint from $U_1$ for $j<0$.  Since $\supp(f)\in U_1\cup 
U_2$, we have $f^j(V_-)\subset U_2\setminus U_1$ for $j<0$.  Consider a 
generalized $2$-cycle $g=\de_{f(V_-);f^{-2}}=\mu[V_-;f,f^{-1};(1\,2)]$.  
The map $g$ is well defined and belongs to the ample group $\cG$.  Since 
$\supp(g)=f(V_-)\cup f^{-1}(V_-)$, this map is in $\cG_{U_1\cup U_2}$, and 
so is the map $h=gf$.  By construction, $h(V_-)=f^{-1}(V_-)$ and $h=f$ on 
the set $U_1\setminus V_-$.  Furthermore, $h^{-1}(f(V_-))=f^{-2}(V_-)$ and 
$h^{-1}=f^{-1}$ on $U_1\setminus f(V_-)$.  It follows that $\{x\in U_1\mid 
h(x)\notin U_1\}=V_-$ and $\{x\in U_1\mid h^{-1}(x)\notin U_1\}=V_-\cup 
f(V_-)$.  Observe that both of the latter sets are different from 
$\emptyset$ and $U_1$.  Therefore the map $h$ has been covered by the 
previous case.  By the above, $h\in\langle\cG_{U_1}\cup\cG_{U_2}\rangle$.  
Since $g\in\sS(\cG_{U_1\cup U_2})$, we obtain that the map $f=g^{-1}h$ 
belongs to $\langle\cG_{U_1}\cup\cG_{U_2}\rangle$ as well.
\end{proof}

\begin{proposition}\label{who-implies-E}
For an ample group acting minimally on a Cantor set, any of Properties 
\ref{prop-UR}, \ref{prop-NC}, \ref{prop-SC}, \ref{prop-UC}, \ref{prop-CI} 
and \ref{prop-HS} implies Property \ref{prop-E}.
\end{proposition}

\begin{proof}
By Lemma \ref{NC-equiv-UR}, Property \ref{prop-UR} is equivalent to 
Property \ref{prop-NC}.  It follows from Lemma \ref{UC-equiv-M-and-SC} that 
for an ample group acting minimally on a Cantor set, Property \ref{prop-SC} 
is equivalent to Property \ref{prop-UC}.  For an ample group, either of 
Properties \ref{prop-NC} and \ref{prop-UC} implies Property \ref{prop-CI} 
due to Lemmas \ref{NC-implies-CI} and \ref{UC-implies-CI}.  By Lemma 
\ref{HS-implies-CI}, Property \ref{prop-HS} implies Property \ref{prop-CI} 
as well.  Finally,  by Lemma \ref{CI-implies-E}, Property \ref{prop-CI} 
implies Property \ref{prop-E} for any ample group acting minimally. 
\end{proof}

Now, equipped with Proposition \ref{who-implies-E}, we are going to provide 
more examples of ample groups with Property \ref{prop-E}.

\begin{definition}
A homeomorphism $f:X\to X$ of a Cantor set $X$ is called \textbf{minimal} 
if there is no closed set $Y\subset X$ different from the empty set and $X$ 
that is invariant under $f$: $f(Y)\subset Y$.
\end{definition}

An equivalent condition is that for any point $x\in X$, the \emph{forward 
orbit} $x,f(x),f^2(x),\dots$ is dense in $X$.  It is easy to show that the 
homeomorphism $f$ is minimal if and only if the cyclic group $\langle 
f\rangle$ generated by $f$ acts minimally on $X$ (in the sense of 
Definition \ref{def-minimality}).

There are many known examples of minimal homeomorphisms of Cantor sets 
(they are referred to as \emph{Cantor minimal systems}).

\begin{theorem}\label{E-for-cyclic-amplified}
For any minimal homeomorphism $f$ of a Cantor set, the ample group 
$\F(\langle f\rangle)$ has Property \ref{prop-E}.
\end{theorem}

It is not hard to show that the cyclic group $\langle f\rangle$ has 
Property \ref{prop-UR} (and the equivalent Property \ref{prop-NC}).  
However, it is not clear whether the ample group $\F(\langle f\rangle)$ has 
these properties.  In general, it is not clear whether Property 
\ref{prop-NC} survives when the group is amplified.

To prove Theorem \ref{E-for-cyclic-amplified}, we are going to use some 
advanced results of Matui.  In \cite{Matui-homology} he introduced the 
notion of an \emph{almost finite} groupoid (see Definition 6.2 in 
\cite{Matui-homology}).  Later Kerr \cite{Kerr} defined almost finiteness 
for group actions (see Definition 8.2 in \cite{Kerr}).  Any continuous 
action of a countable group on a Cantor set gives rise to a groupoid.  In 
this particular case both notions are equivalent (see Example 8.5 in 
\cite{Kerr}).

\begin{theorem}[Matui {\cite{Matui-homology}}]
\label{almost-finite-implies-HS}
Let $G$ be a group of homeomorphisms of a Cantor set $X$. 
If the natural action of the group $G$ on $X$ is minimal and almost finite, 
then the ample group $\F(G)$ has Property \ref{prop-HS}.
\end{theorem}

Note that Matui's result (Theorem 6.12 in \cite{Matui-homology}) is 
formulated in terms of groupoids.  Theorem \ref{almost-finite-implies-HS} 
is a reformulation of it in a particular case when the groupoid is 
associated to a group action.

\begin{proposition}[Matui {\cite{Matui-homology}}]
\label{Zn-action}
Any free action of the group $\bZ^n$ on a Cantor set is almost finite.
\end{proposition}

Proposition \ref{Zn-action} is a less general version of Lemma 6.3 in 
\cite{Matui-homology}.

\begin{proof}[Proof of Theorem \ref{E-for-cyclic-amplified}]
Suppose $f$ is a minimal homeomorphism of a Cantor set $X$.  Minimality 
implies that $f$ has no periodic points.  Hence the natural action of the 
cyclic group $\langle f\rangle$ on $X$ is free.  Since $\langle f\rangle$ 
is isomorphic to $\bZ$, that action is almost finite due to Proposition 
\ref{Zn-action}.  Besides, the action of $\langle f\rangle$ is minimal.  
Then it follows from Theorem \ref{almost-finite-implies-HS} that the ample 
group $\F(\langle f\rangle)$ has Property \ref{prop-HS}.  Since the group 
$\langle f\rangle$ acts minimally on $X$, so does its amplification.  It 
remains to apply Proposition \ref{who-implies-E}.
\end{proof}

The notion of almost finiteness has been extensively studied.  It is 
expected that any free minimal action of a countable amenable group on the 
Cantor set is almost finite (see Problem 8.9 in \cite{Kerr}).  Let us 
mention a recent result of Kerr and Naryshkin \cite{KerrNar} in that 
direction.

\begin{theorem}[Kerr, Naryshkin {\cite{KerrNar}}]
\label{elem-amenable}
Every free continuous action of a countable elementary amenable group on a 
finite-dimensional compact metrizable space is almost finite.
\end{theorem}

In view of Theorem \ref{elem-amenable} (Theorem A in \cite{KerrNar}), the 
cyclic group $\langle f\rangle$ in Theorem \ref{E-for-cyclic-amplified} can 
be replaced by any countable elementary amenable group $G\subset\Homeo(X)$ 
such that the natural action of $G$ on the Cantor set $X$ is minimal and 
free.

There is an alternative way to prove Theorem \ref{E-for-cyclic-amplified} 
using more elementary means.  Let us outline it.  Suppose $f$ is a minimal 
homeomorphism of a Cantor set $X$.  For any point $x\in X$, consider its 
forward orbit $\{f^k(x)\mid k\ge0\}$ with respect to $f$ and the stabilizer 
of this orbit  under the natural action of the ample group $\cG=\F(\langle 
f\rangle)$.  By Lemma \ref{ample-stab-set}, the stabilizer 
$\St_{\cG}(\{f^k(x)\mid k\ge0\})$ is also an ample group.

\begin{proposition}[Putnam {\cite{Putnam89}}]
\label{stab-forward-orbit}
The stabilizer $\St_{\cG}(\{f^k(x)\mid k\ge0\})$ is a locally finite group.
\end{proposition}

Proposition \ref{stab-forward-orbit} can be extracted from the results of 
Section 5 in \cite{Putnam89} (as asserted in \cite{GPS}).  Note that the 
group $\F(\langle f\rangle)$ appears only implicitly in \cite{Putnam89} as 
everything is formulated in terms of a $C^*$-algebra associated with the 
minimal homeomorphism $f$.  A more general fact is explicitly formulated 
and proved by Juschenko and Monod (see Lemmas 4.1 and 4.2 in 
\cite{JuschMonod}).

\begin{proposition}\label{same-orbits-clopen}
The action of the subgroup $\St_{\cG}(\{f^k(x)\mid k\ge0\})$ on clopen 
subsets of the Cantor set $X$ has the same orbits as the action of the 
entire group $\cG$.
\end{proposition}

Assuming Proposition \ref{same-orbits-clopen} is proved, the proof of 
Theorem \ref{E-for-cyclic-amplified} goes as follows.  Since the stabilizer 
$\St_{\cG}(\{f^k(x)\mid k\ge0\})$ is locally finite, every element has 
finite order.  Hence this group obviously has Property \ref{prop-UR}.  
Also, it has Property \ref{prop-NC} (this is easy to check directly, or we 
can use Lemma \ref{NC-equiv-UR}).  As the stabilizer is an ample group, 
Lemma \ref{NC-implies-CI} implies that it has Property \ref{prop-CI}.  Then 
it follows from Proposition \ref{same-orbits-clopen} that the group $\cG$ 
has Property \ref{prop-CI} as well.  It remains to apply Proposition 
\ref{who-implies-E}. 

We do not include the proof of Proposition \ref{same-orbits-clopen} as it 
is quite lengthy.  It will be published elsewhere.

There are also completely different kinds of ample groups with Property 
\ref{prop-E} where that property holds since the groups allow strong 
contraction of clopen sets (Property \ref{prop-SC}).  First of all, let us 
mention Thompson's group $V$ (already considered above) as well as its 
various generalizations including the Higman-Thompson groups $V_{n,r}$, 
also denoted $G_{n,r}$ (see, e.g., \cite{BCMNO}), the Brin-Thompson groups 
$nV$ (see \cite{Brin}) and others.  Each of those groups can be realized as 
a group of homeomorphisms of a Cantor set that is easily shown to be 
ample.  Some of the groups are not simple so that the argument in Remark 
\ref{ThompsonV} may not apply to them.  However, it is not hard to prove 
that any of the groups $V_{n,r}$ and $nV$ acts transitively on the set of 
all nontrivial clopen subsets of the respective Cantor set.  This 
immediately implies minimality and Properties \ref{prop-SC}, \ref{prop-UC}, 
\ref{prop-CI} and \ref{prop-HS}.

Another class of ample groups with Property \ref{prop-E} are the 
(topological) \emph{full groups} of minimal, purely infinite \'etale 
groupoids considered by Matui in \cite{Matui-pureinf1,Matui-pureinf2}.  The 
full group is ample by construction.  Minimality of the groupoid translates 
into minimality of the natural action of the group while pure infiniteness 
translates into Property \ref{prop-SC}.  For example, a purely infinite 
groupoid can be associated to any one-sided subshift of finite type over a 
finite alphabet (see Section 6 in \cite{Matui-pureinf1}).  Note that the 
ample group obtained from the groupoid associated to the full shift on $n$ 
symbols is isomorphic to the Higman-Thompson group $V_{n,1}$.

\bigskip

{\sc
\begin{raggedright}
Department of Mathematics\\
Texas A\&M University\\
College Station, TX 77843--3368
\end{raggedright}
}

\end{document}